\documentclass[10pt,a4paper,twoside,reqno]{amsart}

\usepackage[a4paper,top=3cm,bottom=3cm,inner=2.5cm,outer=2.5cm]{geometry}
\usepackage{amssymb,mathtools,stmaryrd,amsthm, thmtools, xspace}
\SetSymbolFont{stmry}{bold}{U}{stmry}{m}{n}
\usepackage[shortlabels]{enumitem}
\usepackage[utf8]{inputenc}
\usepackage{microtype}

\usepackage{xcolor}
\usepackage{graphicx}
\definecolor{darkgreen}{rgb}{0,0.45,0}
\usepackage{bm}

\usepackage[colorlinks,citecolor=darkgreen,urlcolor=gray,final,hyperindex,linktoc=page,hyperfootnotes=true]{hyperref}
 \usepackage[capitalize]{cleveref}

\DeclareMathAlphabet{\mathbf}{OT1}{cmr}{b}{n}
\DeclareFontSeriesDefault[rm]{bf}{b}

\usepackage{tikz}
\usetikzlibrary{positioning,decorations.pathreplacing,decorations.markings,decorations.pathmorphing,arrows.meta,calc,fit,quotes,cd,math,arrows,backgrounds,shapes.geometric}

\tikzset{
  tail arrow/.style={
    postaction={
      decorate,
      decoration={
        markings,
        mark=at position 0 with {\arrow{>}}
      }
    }
  }
}

\makeatletter
\DeclareRobustCommand{\tparen}[1]{%
  \mathpalette\tparen@aux{#1}%
}
  \newcommand{\tparen@aux}[2]{%
  \mathopen{%
    \smash{\raisebox{0pt}[0pt][0pt]{\scalebox{0.7}[1]{$#1($}}}%
  }%
  #2%
  \mathclose{%
    \smash{\raisebox{0pt}[0pt][0pt]{\scalebox{0.7}[1]{$#1)$}}}%
  }%
}
\makeatother

\definecolor{mypurple}{rgb}{0.5, 0.0, 0.5}

\theoremstyle{plain}
\newtheorem{thm}{Theorem}[subsection]
\newtheorem{cor}[thm]{Corollary}
\newtheorem{lem}[thm]{Lemma}
\newtheorem{prop}[thm]{Proposition}

\theoremstyle{remark}

\newtheorem{rmk}[thm]{Remark}

\newtheorem{rex}[thm]{Running Example}
\newtheorem{notat}[thm]{Notation}

\theoremstyle{definition}
\newtheorem{defi}[thm]{Definition}
\newtheorem{hyp}{Assumption}

\crefname{equation}{}{}
\crefname{lem}{Lemma}{Lemmas}
\crefname{thm}{Theorem}{Theorems}
\crefname{rmk}{Remark}{Remarks}
\crefname{defi}{Definition}{Definitions}
\crefname{conj}{Conjecture}{Conjectures}
\crefname{ex}{Example}{Examples}
\crefname{sec}{Section}{Sections}
\crefname{prop}{Proposition}{Propositions}
\crefname{hyp}{Assumption}{Assumptions}
\crefname{rex}{Running Example}{Running Example}

\newcommand{\ie}{i.e.\xspace}
\newcommand{\cf}{cf.\xspace}
\newcommand{\eg}{e.g.\xspace}

\newcommand{\defeq}{\mathrel{\mathop:}=}
\newcommand{\co}{\colon}
\newcommand{\op}{\mathrm{op}}
\newcommand{\id}{\mathrm{id}}

\newcommand{\hcomp}{\circ} 
\newcommand{\bimcomp}{\bullet} 
\newcommand{\thg}{{\mathord{\text{--}}}}

\DeclareFontFamily{U}{mathx}{}
\DeclareFontShape{U}{mathx}{m}{n}{<-> mathx10}{}
\DeclareSymbolFont{mathx}{U}{mathx}{m}{n}
\DeclareMathAccent{\wh}{0}{mathx}{"70}
\DeclareMathAccent{\wc}{0}{mathx}{"71}
\newcommand{\comp}[1]{\wh{#1}}
\newcommand{\coj}[1]{\wc{#1}}

\newcommand{\one}{\mathbf{1}}
\newcommand{\mi}{\textrm{-}}
\newcommand{\ot}{\otimes}
\newcommand{\ev}{\mathsf{ev}}
\newcommand{\bxt}{\boxtimes} 
\newcommand{\nctensor}{\mathbin{\Box}} 
\newcommand{\ctensor}{\otimes} 
\DeclarePairedDelimiter\boxhom\langle\rangle
\DeclarePairedDelimiter\nchom\llbracket\rrbracket
\DeclarePairedDelimiter\chom\lbrack\rbrack

\newcommand{\ca}{\mathcal} 
\newcommand{\dc}{\mathbb} 
 
\newcommand{\defcat}{\mathsf}
 \newcommand{\dcC}{\dc{C}}

\newcommand{\Bim}{\mathbb{B}\defcat{im}}
\newcommand{\BimC}{\Bim(\dcC)}
\newcommand{\End}{\defcat{End}}
\newcommand{\EndC}{\End(\dc{C})}
\newcommand{\uEndC}{\underline{\smash{\End}}(\dc{C})}
\newcommand{\Mnd}{\defcat{Mnd}}
\newcommand{\uMndC}{\underline{\smash{\Mnd}}(\dc{C})}
\newcommand{\MndC}{\Mnd(\dc{C})}
\newcommand{\EndCA}[1][A]{\End_{#1}(\dc{C})}
\newcommand{\MndCA}[1][A]{\Mnd_{#1}(\dc{C})}

\newcommand{\MultMndC}{\defcat{MultMnd}(\dcC)}
\newcommand{\CommMultMnd}{\defcat{CMultMnd}}
\newcommand{\CommMultMndC}{\defcat{CMultMnd}(\dcC)}

\newcommand{\CommMultBimC}{\defcat{CMultBim}(\dcC)}

\newcommand{\Set}{\defcat{Set}}

\newcommand{\Cat}{\defcat{Cat}}
\newcommand{\VCat}{\defcat{Cat}_\ca{V}}

\newcommand{\BMod}{\dc{B}\defcat{Mod}}

\newcommand{\Mat}{\dc{M}\defcat{at}}
\newcommand{\VMat}{\dc{M}\defcat{at}_\ca{V}}
\newcommand{\Prof}{\dc{P}\defcat{rof}}
\newcommand{\VProf}{\dc{P}\defcat{rof}_\ca{V}}

\newcommand{\Sym}{\dc{S}\defcat{ym}}
\newcommand{\VSym}{\dc{S}\defcat{ym}_\ca{V}}

\newcommand{\VCatSym}{\dc{C}\defcat{atSym}_\ca{V}}

\newcommand{\VSMultCat}{\defcat{SMultCat}_\ca{V}}

\newcommand{\OpdBim}{\dc{O}\defcat{pdBim}}
\newcommand{\VOpdBim}{\OpdBim_\ca{V}}

\newcommand{\SMultProf}{\dc{S}\defcat{MultProf}}
\newcommand{\VSMultProf}{\dc{S}\defcat{MultProf}_\ca{V}}

\newcommand{\freemnd}{\textsf{Free}}

\newcommand{\disc}[1]{\Delta_{#1}}
\newcommand{\coend}{\int}

\tikzset{tick/.style={postaction={decorate,decoration={markings,mark=at
position 0.5 with {\draw[-] (0,.4ex) -- (0,-.4ex);}}}}}
\tikzset{bigtick/.style={postaction={decorate,decoration={markings,mark=at
position 0.5 with {\draw[-] (0,.6ex) -- (0,-.6ex);}}}}}
\newcommand{\tickar}{\begin{tikzcd}[baseline=-0.5ex,cramped,sep=small,ampersand
replacement=\&]{}\ar[r,tick]\&{}\end{tikzcd}}

\newcommand{\sbul}{\scriptstyle\bullet}
\tikzset{bul/.style={postaction={decoration={markings,mark=at position 0.5 with
{\node{$\sbul$};}},decorate}}}

\tikzset{Rightarrow/.style={double equal sign distance,>={Implies},->},
triple/.style={-,preaction={draw,Rightarrow}}}

\newcommand{\tor}{\ensuremath{\relbar\joinrel\mapstochar\joinrel\rightarrow}}

\newcommand{\hid}{\mathrm{id}}
\newcommand{\Two}{\scriptstyle\Downarrow}

\newcommand{\twocong}[2][0.5]{\ar@{}[#2] \save ?(#1)*{\cong}\restore}
\newcommand{\twoeq}[2][0.5]{\ar@{}[#2] \save ?(#1)*{=}\restore}
\newcommand{\rtwocell}[3][0.5]{\ar@{}[#2] \ar@{=>}?(#1)+/l 0.2cm/;?(#1)+/r
0.2cm/^{#3}}
\newcommand{\ltwocell}[3][0.5]{\ar@{}[#2] \ar@{=>}?(#1)+/r 0.2cm/;?(#1)+/l
0.2cm/^{#3}}
\newcommand{\ltwocello}[3][0.5]{\ar@{}[#2] \ar@{=>}?(#1)+/r 0.2cm/;?(#1)+/l
0.2cm/_{#3}}
\newcommand{\dtwocell}[3][0.5]{\ar@{}[#2]|{\Two {#3}}}
\newcommand{\dltwocell}[3][0.5]{\ar@{}[#2] \ar@{=>}?(#1)+/ur  0.2cm/;?(#1)+/dl
0.2cm/^{#3}}
\newcommand{\drtwocell}[3][0.5]{\ar@{}[#2] \ar@{=>}?(#1)+/ul  0.2cm/;?(#1)+/dr
0.2cm/^{#3}}
\newcommand{\dthreecell}[3][0.5]{\ar@{}[#2] \ar@3{->}?(#1)+/u  0.2cm/;?(#1)+/d
0.2cm/^{#3}}
\newcommand{\utwocell}[3][0.5]{\ar@{}[#2] \ar@{=>}?(#1)+/d 0.2cm/;?(#1)+/u
0.2cm/_{#3}}
\newcommand{\dtwocelltarg}[3][0.5]{\ar@{}#2 \ar@{=>}?(#1)+/u  0.2cm/;?(#1)+/d
0.2cm/^{#3}}
\newcommand{\utwocelltarg}[3][0.5]{\ar@{}#2 \ar@{=>}?(#1)+/d  0.2cm/;?(#1)+/u
0.2cm/_{#3}}

\begin{document}
\title[Commuting tensor products]{A unified treatment of commuting tensor products of categories, operads, symmetric multicategories and their bimodules}

\author[N. Gambino]{Nicola Gambino}
\address{Department of Mathematics, The University of Manchester}
\email{nicola.gambino@manchester.ac.uk}

\author[R. Garner]{Richard Garner}
\address{School of Mathematical and Physical Sciences, Macquarie University}
\email{richard.garner@mq.edu.au}

\author[C. Vasilakopoulou]{Christina Vasilakopoulou}
\address{School of Applied Mathematical and Physical Sciences, National Technical University of Athens}
\email{cvasilak@math.ntua.gr}

\begin{abstract} 
We provide a unified treatment of several commuting tensor products considered in the literature,
including the tensor product of enriched categories and the Boardman--Vogt tensor product of 
operads and symmetric multicategories, subsuming work of Elmendorf and Mandell. 
We then show how a commuting tensor product  extends to bimodules, generalising results of Dwyer and Hess. In particular, we construct a double category of symmetric multicategories, symmetric 
multifunctors and bimodules and show that it admits a symmetric oplax monoidal structure.
These applications are obtained as instances of a general construction of commuting 
tensor products on double categories of monads, monad morphisms and bimodules.
\end{abstract}


\maketitle

\setcounter{tocdepth}{1}
\tableofcontents

\section{Introduction}

\subsection*{Context, aims and motivation} Eighty years since its birth~\cite{EilenbergS:gentne}, Category Theory continues to demonstrate the importance of abstraction and compositionality in Mathematics,
helping us to make analogies precise and to construct complex structures from simple ones. Here, we offer additional evidence of this idea, by providing a general analysis of commuting tensor products, unifying and extending a variety of results in the literature, including some by Elmendorf and Mandell~\cite{ElmendorfA:percma} and by Dwyer and Hess~\cite{DwyerW:BoardmanVtpo}. Before describing our contributions in more detail, we provide some context and motivation.

The general idea of a commuting tensor product can be understood in a simple example. For categories $A, B, C$, a functor in two variables 
(often called a sesquifunctor) $F \co A, B \to C$ consists of two families of functors, 
$( F(a, -) \co B \to C)_{a \in A}$ and $( F(-, b) \co A \to C)_{b \in B}$, which agree on objects, \ie such 
that~$F(a,-)(b) = F(a, -)(b)$, so that we can write $F(a,b)$ for their common value for $a \in A$ and $b \in B$. 
We say that such $F$ is commuting (or that it is  a bifunctor) if for all morphisms $f \co a \to a'$ in $A$ and~$g \co b \to b'$ in $B$, the diagram
\[
\begin{tikzcd}[column sep = large]
F(a,b) \ar[r, "{F(f,b)}"] \ar[d, "{F(a, g)}"'] & F(a',b) \ar[d, "{F(a', g)}"] \\
F(a,b') \ar[r, "{F(f, b')}"'] & F(a',b') 
\end{tikzcd}
\]
commutes. This is a commutation condition since the two paths around
the diagram involve the actions of $f$ and $g$ in different order
This is even more apparent when $A$, $B$ and $C$ are one-object
categories, hence monoids: for then, $F$ is given by a pair of monoid
morphisms $f \colon A \rightarrow C$, $g \co B \rightarrow C$, and is commuting just when
each $f(a)$ commutes with each $g(b)$ in $C$.
It is a classical result that, just as bilinear maps of rings are classified by
the tensor products of rings, commuting functors in two variables as above are classified by the cartesian product $A \times B$ of $A$ and~$B$, in the sense that there is a commuting functor in two variables 
$A, B \to A \times B$ such that composition with it induces a bijection between functors
$A \times B \to C$ and commuting functors in two variables $A, B \to C$, for every category $C$.

An important property of the commuting tensor product of categories is that it extends naturally to act on profunctors, also known as bimodules or distributors~\cite{BenabouJ:dis}. Recall that a 
profunctor $F \co A \tickar B$ is a functor $F \co B^\op \times A \to \Set$.
The extension of the commuting tensor product means that, for profunctors $F_1 \co A_1 \tickar B_1$ and $F_2 \co A_2 \tickar B_2$, 
we can define a profunctor $F_1 \times F_2 \co A_1 \times A_2 \tickar B_1 \times B_2$,
 so that the bicategory of profunctors becomes symmetric monoidal~\cite{ConstrSymMonBicatsFun}.
This is in analogy with the way in which the tensor product of rings can be extended to ring bimodules.
Crucially, this involves proving a suitable interaction between composition
of profunctors and the commuting tensor product.

The initial motivation for this work was to investigate whether a similar phenomenon occurs for symmetric multicategories---also known as coloured operads~\cite{yau2016colored} (an operad is exactly 
a symmetric multicategory with one object~\cite{Mayoperad}). Symmetric multicategories were introduced by Boardman and Vogt as part of their work on homotopy theory~\cite{BoardmanJ:homias}, building on the notion of a 
multicategory, originally defined by Lambek within categorical
logic~\cite{LambekJ:dedscs}. They continue to play an important role in algebraic topology to date (see  \cite{BergerC:rescorh,ElmendorfA:rinmai} for example).
Much as with the situation for categories, there is a commuting tensor product for symmetric multicategories, commonly known as the Boardman--Vogt tensor product~\cite{BoardmanJ:homias}. This can be understood
as a commuting tensor product since, for symmetric multicategories $M$ and $N$, an algebra for 
their tensor product $M \ctensor N$ is an object equipped with operations which make it into an $M$-algebra and an $N$-algebra and
commute with each other in a suitable sense. Equivalently, they are $M$-algebras in the category of $N$-algebras (and conversely). Importantly, fundamental classes of spaces can be characterised 
as algebras for symmetric multicategories obtained via (derived version of) the Boardman--Vogt tensor
product (see~\cite{MoerdijkI:mystp} for a recent overview).

There is also an analogue of profunctor between symmetric multicategories;
these were called bimodules in~\cite[Section~4.4]{GambinoJoyal}, as they generalise 
the operad bimodules of~\cite{Fresse,Kapranov,Rezk}. It is thus natural to ask
whether, just as the commuting tensor product of categories extends to profunctors, the Boardman--Vogt tensor product of symmetric multicategories extends to their bimodules. This means that, for bimodules $P_1 \co M_1 \tickar N_1$ and $P_2 \co M_2 \tickar N_2$,  there is a bimodule $P_1 \otimes P_2 \co M_1 \otimes N_1 \tickar M_2 \otimes P_2$, interacting suitably with composition of bimodules. 
The only result in this direction that we are aware of was obtained by Dwyer and Hess in~\cite{DwyerW:BoardmanVtpo}, where they showed that the Boardman--Vogt tensor product of 
(single-coloured) \emph{operads} extends to their bimodules. Even in the case of operads, however, it is not known what kind of monoidal structure this is and how it interacts with the composition of bimodules. This paper provides
answers to these questions.

Our first goal in this paper is to provide a unified account of a variety of commuting tensor products of the kind
described above for categories and symmetric multicategories. Our results will unify several
known results and derive some new ones.
Our second goal is to show that, in the general setting of our analysis, commuting tensor products can be extended to act on bimodules. As an application, we will show that the Boardman--Vogt tensor product extends to bimodules (which we will call \emph{symmetric multiprofunctors} here, to highlight  the analogy with profunctors), extending the work of Dwyer and Hess. We will also establish the kind of monoidal
structure this gives rise to, in particular describing precisely the  interaction of the Boardman--Vogt tensor product with the composition of symmetric multiprofunctors.

This work therefore contributes to the study of the bicategory of
of symmetric sequences~\cite{FioreM:carcbg} and the bicategory of 
symmetric multicategories and symmetric multiprofunctors~\cite{GambinoJoyal},
which have attracted significant interest in recent years, thanks to the
applications they have found in algebraic topology, mathematical logic and theoretical
computer science. Indeed, these bicategories have been shown to admit a rich structure, including being cartesian closed~\cite{FioreM:carcbg,GambinoJoyal},
to admit fixpoint operators~\cite{GalalZ:fixocs}, and support a notion of differentiation~\cite{FioreM:monbdll}. 
Our work here extends this work by showing the existence of additional structure on the double
in which these bicategories naturally embed, which we hope will foster new applications. 

\begin{table}[htb]
\begin{tabular}{|c|ccc|}
\hline
Double category & $\BMod$ & $\Prof$ & $\SMultProf$  \\ \hline
Objects &  Rings & Categories & Symmetric multicategories \\
Vertical maps & Ring homomorphisms & Functors & Symmetric multifunctors \\
Horizontal maps & Ring bimodules & Profunctors & Symmetric multiprofunctors \\
Non-commuting tensor product & $R + S$ & $A \nctensor B$ & $M \sharp N$ \\
Commuting tensor product & $R \otimes S$ & $A \times B$ & $M \ctensor N$ \\ \hline
\end{tabular}
\medskip
\caption{Commuting and non-commuting tensor products.}
\label{tab:commuting}
\end{table}

\subsection*{Main results} In order to develop our general analysis of commuting tensor products,
we shall work in the setting of a double category~\cite{Ehresmanndouble}, in which one has two classes of morphisms, vertical and horizontal. This is useful for us because it matches our examples, where one has two kinds of maps, \eg functors and profunctors. Another feature of double categories that is important for us is that they provide a a more general and convenient framework for
the construction of monoidal structures than bicategories~\cite{GarnerR:lowdsf,ConstrSymMonBicatsFun}, \cf our comments below.

Crucially for our purposes, the double categories $\BMod$, $\Prof$ and $\SMultProf$ (whose definitions are outlined in \cref{tab:commuting}) can be described
homogeneously.  For a double category $\dcC$ with stable local reflexive coequalisers, there is a double category $\BimC$ with horizontal monads in $\dcC$ as objects,
monad morphisms as vertical maps, and monad bimodules as horizontal maps~\cite{CruttwellShulman,Framedbicats}. 
If $\dcC = \Mat$, the double category of matrices, then $\BimC = \Prof$. In particular, as is
well-known, monads in $\Mat$ are categories. If $\dcC = \Sym$, the double category of symmetric sequences~\cite{Paper1}, then $\BimC = \SMultProf$. In particular, monads in $\Sym$ are symmetric multicategories. In light of this, our unified treatment will be achieved by understanding under what assumptions on a double category $\dcC$ one can define a `commuting' tensor product on the double category  $\BimC$. Our theory will apply also to the $\ca{V}$-enriched versions of the examples above, where $\ca{V}$ is a fixed symmetric monoidal category satisfying appropriate axioms. We will
write $\VMat$  for the enriched counterpart of~$\Mat$ and use similar notation for the other double
categories under consideration. 

To achieve our goal, we shall extend the analysis of commutativity carried out in~\cite{GarnerLopezFranco}.
This work was developed
in the context of a normal duoidal category, \ie a category equipped with two monoidal structures $(\ca{E}, \circ, J, \bxt, I)$ related by (not necessarily invertible) interchange 
maps, subject to appropriate axioms; for example, any symmetric (or braided) monoidal category is normal duoidal, with the two tensor products taken to be the same. One of the main theorems in~\cite{GarnerLopezFranco} shows that, for a normal duoidal category $\ca{E}$ satisfying appropriate assumptions, the category of $\circ$-monoids and $\circ$-monoid morphisms in $\ca{E}$  admits a `commuting' monoidal structure $\otimes$ such that $\circ$-monoid maps $A \otimes B \rightarrow C$ classify pairs of $\circ$-monoid maps $A \rightarrow C$, $B \rightarrow C$ which commute with respect to the other tensor product $\bxt$ of $\ca{E}$. Here, the meaning of ``commuting'' generalises the motivating case of monoids in $(\Set, \times, 1, \times, 1)$, for which it means ``with images that commute elementwise''.

This result of~\cite{GarnerLopezFranco}, applied to suitable bases $\ca{E}$, provides a unified treatment of the commuting tensor product of operads and Lawvere theories, but does not apply to categories or symmetric multicategories. Now,~\cite{GarnerLopezFranco} also considers a more general situation which, rather than the category of $\circ$-monoids, starts from the category of $(\ca{E}, \circ, J)$-enriched categories, and constructs a `commuting' tensor product thereon. Applied to $(\Set, \times, 1, \times, 1)$, this reconstructs perfectly the commuting tensor product (= cartesian product) of categories, and more generally, the usual tensor product of $\ca{V}$-categories over a symmetric (or braided) monoidal base $\ca{V}$. Yet even this generalisation does not suffice to capture symmetric multicategories, which are not $\ca{V}$-categories for any known enrichment base $\ca{V}$.

Thus, our development will provide a different  generalisation of the `one-object' results of~\cite{GarnerLopezFranco} to the `many-object' case.
We will replace the monoidal category~$(\ca{E}, \circ, J)$ with a double category~$\dcC$, using the horizontal composition in $\dcC$ as the counterpart of the substitution monoidal structure of~$\ca{E}$ and correspondingly, replacing $\circ$-monoids in 
the monoidal category with horizontal monads in the double category (see \cref{sec:preliminiaries} for details).
The second monoidal structure $\bxt$ on the duoidal category $\ca{E}$ is then replaced by a \emph{normal oplax monoidal structure}~\cite{Paper1}
on the double category $\dcC$.
Here, the oplaxity means that the functoriality of the tensor functor holds only up to an interchange 2-cell of the form
\begin{equation}
\label{equ:first} 
\begin{tikzcd}[column sep=.4in]
 A_1\bxt A_2\ar[rr,tick,"(y_1\circ x_1)\bxt(y_2\circ x_2)"]\ar[d,equal]\ar[drr,phantom,"\scriptstyle\Downarrow\xi"] && C_1\bxt C_2\ar[d,equal] \\
 A_1\bxt A_2\ar[r,tick,"x_1\bxt x_2"'] & B_1\bxt B_2\ar[r,tick,"y_1\bxt y_2"'] & C_1\bxt C_2 \mathrlap{.}
\end{tikzcd}
\end{equation}
For example, any symmetric monoidal double category $\dcC$ is normal oplax monoidal. In particular, if 
$\ca{V}$~is a symmetric monoidal category, then the double category $\VMat$ is symmetric monoidal in an evident way, and hence normal oplax monoidal. Less trivially,~\cite{Paper1} showed that, for a \emph{cartesian} monoidal category~$\ca{V}$, the double category of $\ca{V}$-enriched symmetric sequences $\VSym$ is normal oplax monoidal under a tensor product $\bxt$ which generalises the \emph{arithmetic product} of one-object symmetric sequences from~\cite{DwyerW:BoardmanVtpo}.

In the context of a normal oplax monoidal double category, we introduce a notion of a monad multimorphism, by extending to the abstract
setting the concept of a functor in many variables, and say what it means for a monad multimorphism to be commuting, by generalising the definition for a family of monoid morphisms to be commuting in~\cite{GarnerLopezFranco}.
The essential content of this definition can be understood in the binary case, as follows. First of all, a binary monad multimorphism $(f, \phi, \psi) \co 
(A, a), (B, b) \to (C, c)$ involves a vertical map $f \colon A \bxt B \rightarrow C$ and monad maps $(f, \phi) \colon (A \bxt B, a \bxt \hid_B) \rightarrow (C,c)$ and $(f, \psi) \colon (A \bxt B, \hid_A \bxt b) \rightarrow (C,c)$. Now such a multimorphism is said to be commuting if 
\begin{displaymath}
\begin{tikzcd}[column sep = large]
A \bxt B \ar[rr, tick, "a \bxt b"] \ar[d, equal] 
\ar[drr,  phantom, pos=.5,"\Two \sigma"] 
& & A \bxt B \ar[d, equal] \\
A \bxt B \ar[r, tick,"a \bxt \id_{B}"] \ar[d, "f"'] 
\ar[dr,  phantom, pos=.4,"\Two \phi"]  &
A \bxt B \ar[r, tick,"\id_{A} \bxt b"] \ar[d, "f"] 
\ar[dr,  phantom, pos=.4,"\Two \psi"] &
A \bxt B \ar[d, "f"] \\
C \ar[r, tick,"c"'] \ar[d, equal] \ar[drr,  phantom, pos=.5,"\Two \mu"]  &
C \ar[r, tick,"c"'] &
C  \ar[d, equal] \\
C \ar[rr, tick,"c"']  & & 
C 
  \end{tikzcd}  \quad = \quad
  \begin{tikzcd}[column sep = large]
A \bxt B \ar[rr, tick,"a \bxt b"] \ar[d, equal] 
\ar[drr,  phantom, pos=.5,"\Two \tau"] 
& & A \bxt B \ar[d, equal] \\
A \bxt B \ar[r, tick,"\id_{A} \bxt b "] \ar[d, "f"'] 
\ar[dr,  phantom, pos=.4,"\Two \psi"]  &
A \bxt B \ar[r, tick,"a \bxt  \id_{B}"] \ar[d, "f"] 
\ar[dr,  phantom, pos=.4,"\Two \phi"] &
A \bxt B \ar[d, "f"] \\
C \ar[r, tick,"c"'] \ar[d, equal] \ar[drr,  phantom, pos=.5,"\Two \mu"]  &
C \ar[r, tick,"c"'] &
C  \ar[d, equal] \\
C \ar[rr, tick,"c"']  & & 
C \mathrlap{,}
  \end{tikzcd} 
  \end{displaymath}
where the 2-cells $\sigma$ and $\tau$ can be
obtained via the interchange 2-cell $\xi$ in~\eqref{equ:first}.
Here, commutativity can be intuitively seen by the swapping of $\phi$ and $\psi$ between
the two diagrams (see \cref{thm:monad-multimorphism,defi:commuting-monad-multimorphism} for details).
As we will see, this captures the appropriate
notions of commutativity in the examples of interest. Our motivating question then becomes one of identifying conditions under which commuting monad multimorphisms can be classified, meaning that 
there exists a monad $(A, a) \ctensor (B, b)$ with a commuting multimorphism
$(A, a), (B, b) \to (A, a)  \otimes (B,  b)$ which is \emph{universal}, in that it induces, via composition, a bijection between monad morphisms
$(A, a) \ctensor (B, b) \to (C,c)$ and commuting monad multimorphisms
$(A, a), (B, b) \to (C, c)$ for every monad $(C,c)$.

Our first main result, \cref{thm:ctensor-monads}, shows that, under appropriate assumptions,
the category of
monads and monad morphisms in $\dcC$, \ie the vertical category of $\BimC$,
admits a commuting tensor product, which is moreover part of a symmetric monoidal
closed structure. When applied to the normal oplax monoidal double category of matrices $\VMat$,
this gives back the tensor product of enriched categories and enriched
functors. When applied to the normal oplax monoidal double category of symmetric sequences $\VSym$, this
gives back the construction of the Boardman--Vogt tensor product of enriched symmetric multicategories and symmetric multifunctors obtained in~\cite{ElmendorfA:percma}.
The second main result proved here, \cref{thm:ctensor-bimodules}, shows that 
the commuting tensor product of \cref{thm:ctensor-monads} extends to monad
bimodules and  that the double category $\BimC$ of bimodules in $\dcC$
admits a symmetric oplax monoidal structure. 

We then provide two applications of our main results. First, we consider the double category of matrices $\VMat$. When $\ca{V}$ is symmetric monoidal, our results then give an abstract construction of the tensor product of $\ca{V}$-categories and its
extension to profunctors (\cref{thm:app-prof}). While this is known, it is pleasing to have a conceptual
proof of the existence of the symmetric monoidal structure on the bicategory of $\ca{V}$-profunctors, in 
particular taking care of the daunting coherence axioms. 
However, when $\ca{V}$ is merely normal duoidal, we still have normal oplax monoidal structure on $\VMat$, and in this context, we re-find exactly the previous `many-object' theory of commutativity developed in~\cite{GarnerLopezFranco}---but, in addition, an extension of that theory to bimodules between enriched categories, which appears to be new.

For our second main application, we consider the normal oplax monoidal double category of symmetric sequences $\VSym$. As noted above, our results here give
back the Boardman--Vogt tensor product of  symmetric $\ca{V}$-multicategories, but furthermore provide a new result
in extending it to their bimodules, which we we call here \emph{symmetric $\ca{V}$-multiprofunctors} (\cref{thm:bv-bimod}).
Our approach with double categories is useful here:  the double category $\VSMultProf$ admits a symmetric oplax monoidal structure that does not seem to satisfy the
extra condition of normality that is needed in order to obtain a symmetric monoidal structure on its horizontal bicategory (\cf \cref{thm:not-normal}).

Even when restricted to operads (\cref{corollary:oplaxoperads}), this provides something new, in that it describes precisely
the interaction between the Boardman--Vogt tensor product and the composition operation of bimodules.
Given that the definition of the Boardman--Vogt tensor product of operads and operad morphisms goes back to~\cite{BoardmanJ:homias} and its extension to operad bimodules to~\cite{DwyerW:BoardmanVtpo}, this
appears to fill a gap in our knowledge of operads.

Let us also remark that our theory can be applied to yield commuting
tensor products for variants of symmetric multicategories, such as cartesian multicategories~\cite{LambekJ:mulr} and potentially also to their affine and relevant counterparts, which are of interest in mathematical logic
and theoretical computer science~\cite{Olimpieri}. Indeed, cartesian multicategories with a single object
are exactly Lawvere theories~\cite{LawvereFW:funsat}. Since this requires substantial additional work, we do not include it here.

\subsection*{Technical aspects} Our main results establish the definition and key properties of the commuting tensor product of monads
and of bimodules. However, they come at the end of a development that involves
a number of auxiliary notions and results, some of which may be of independent interest. 
Throughout, we have found it useful to use symmetric multicategories not only as source of examples, but also as tool in our proofs. Thus, rather than construct symmetric monoidal structures on a category directly, 
we instead enhance the category to a 
symmetric multicategory and apply the notion of representability due to Hermida~\cite{RepresentableMulticats}. This not only matches naturally what
happens in examples, but allows us to organise some of the more complex
proofs in conceptual terms. 

As a first step, we construct and study the symmetric multicategory of monads and 
(not necessarily commuting) monad multimorphisms. In particular, we show that, under suitable hypotheses on our base double category $\dcC$, 
this symmetric multicategory is closed and representable. We obtain
what may be called the non-commuting tensor product of monads (\cref{thm:nctensor-smc}).
In the example of $\VMat$, where monads are enriched categories, this gives back what is sometimes called the ``funny tensor product''
of enriched categories~\cite{Categoricalstructures}, whose internal hom consists of functors and not-necessarily-natural
transformations. In the example of $\VSym$, where monads are symmetric multicategories,
this is the auxiliary tensor product written $M \sharp N$ in \cite{ElmendorfA:percma}, \cf \cref{tab:commuting}.

Building on this, we introduce our notion of a commuting monad multimorphism (\cref{defi:commuting-monad-multimorphism}) and proceed in parallel to what we did for the non-commuting tensor
product. In particular, we introduce a symmetric multicategory of monads and 
commuting monad multimorphisms and prove that it is closed and representable.
While we can leverage the earlier work on non-commuting tensor product to simplify
some of the work, the results here involve several additional delicate checks on the preservation of
commutativity by the transposition operations that are involved in the representability.
 
We next turn to extending the commuting tensor product to bimodules, abstracting and generalising the work of Dwyer and Hess in~\cite{DwyerW:BoardmanVtpo}. We use the multicategorical approach again,
introducing a notion of commuting bimodule morphism and showing that it can be classified. This may be of interest in that it gives a universal characterisation of the commuting tensor product of bimodules, an aspect which was not considered in~\cite{DwyerW:BoardmanVtpo}. We then study the interaction between the composition of bimodules and the commuting tensor of monads, showing that we have the structure of symmetric oplax monoidal double category. Here, it should be pointed out that the interaction between the composition of bimodules and the commuting tensor of monads is entirely novel and
 was not considered, even in the case of operads, in~\cite{DwyerW:BoardmanVtpo}.
  
Last, but not least, when it comes to examples, we will have to carry out some additional work to 
ensure that the hypotheses of our theorems are satisfied.  For 
our application to symmetric multicategories, we extend the work in~\cite{Paper1} and
show that the  oplax monoidal structure double category $\VSym$ of symmetric sequences is symmetric and closed, a result that is of independent
interest as it generalises the `divided powers' closed structure of~\cite{DwyerW:BoardmanVtpo}.

\subsection*{Outline} \cref{sec:preliminiaries} sets out the preliminaries needed for the paper, including
the key facts from~\cite{Paper1} that we will use. In \cref{sec:non-commuting} we construct the
non-commuting tensor product of monads. This is then used in \cref{sec:commuting} to define
the commuting tensor product. We extend this to bimodules in \cref{sec:bimodules}. We end
the paper in \cref{sec:applications} with applications.

\subsection*{Acknowledgements} Nicola Gambino acknowledges that this material is based upon work supported by the US Air Force Office for Scientific Research under award number FA9550-21-1-0007, by EPSRC via grant EP/V002325/2, and ARIA via grant MSAI-PR01-P12. Christina Vasilakopoulou acknowledges that this work was implemented in the framework of H.F.R.I call ``3rd Call for H.F.R.I.’s Research Projects to
Support Faculty Members \& Researchers'' (H.F.R.I. Project Number: 23249).

\section{Monads and bimodules in double categories} 
\label{sec:preliminiaries}

\subsection{Double categories and fibrant double categories}
We begin by recalling the notion of (pseudo) double category
\cite{Limitsindoublecats,ConstrSymMonBicatsFun}, in part to establish
our notational conventions.

\begin{defi}
  A \emph{double category} $\dc{C}$ consists of a graph of categories
  $\mathfrak{s}, \mathfrak{t} \colon \dc{C}_1 \rightrightarrows
  \dc{C}_0$, together with composition and identity functors
  $\circ\colon\dc{C}_1\times_{\dc{C}_0}\dc{C}_1\to\dc{C}_1$ and
  $\hid\colon\dc{C}_0\to\dc{C}_1$ which are associative and unital up
  to coherent natural isomorphism. We refer to objects
  $A,B,\dots$ and morphisms $f,g,\dots \colon A \rightarrow B$ of
  $\dcC_0$ as \emph{objects} and \emph{vertical $1$-cells} of $\dcC$,
  and refer to objects $x,y,\dots$ and morphisms
  $\alpha, \beta, \dots \colon x \rightarrow y$ of $\dcC_1$ as
  \emph{horizontal $1$-cells} and \emph{$2$-morphisms} of $\dcC$. We
  typically write horizontal $1$-cells as $x \colon A \tor B$ to
  indicate that $\mathfrak{s}(x) = A$ and $\mathfrak{t}(x) = B$, and
  write $2$-morphisms $\alpha \colon x \rightarrow y$ as to the left in:
  \begin{equation}
    \label{eq:two-morphisms}
    \begin{tikzcd}
      A\ar[d,"f"']\ar[r,tick,"x"]\ar[dr,phantom,"\Two\alpha"] &
      B\ar[d,"g"] \\
      C\ar[r,tick,"y"'] &
      D
    \end{tikzcd} \qquad \qquad 
    \begin{tikzcd}
      A\ar[d,equal]\ar[r,tick,"x"]\ar[dr,phantom,"\Two\alpha"] &
      B\ar[d,equal] \\
      A\ar[r,tick,"y"'] &
      B
    \end{tikzcd}
  \end{equation}
  to indicate that $\mathfrak{s}(\alpha) = f$ and
  $\mathfrak{t}(\alpha) = g$. The invertible witnesses to the weak
  associativity of composition and identities are so-called
  \emph{globular} $2$-morphisms, i.e., ones living in the fibres of
  the functor
  $(\mathfrak{s}, \mathfrak{t}) \colon \dcC_1 \rightarrow \dcC_0
  \times \dcC_0$; a typical globular $2$-morphism is displayed to the
  right above. We will not name these globular coherence cells, and
  will typically suppress associativity isomorphisms
  completely.
\end{defi}
In order to illustrate the notions to be introduced throughout the
paper, we begin here the development of a simple running example. As
the paper progresses, we will slowly add further hypotheses to this
example, in parallel with the hypotheses we require in the general theory.

\begin{rex} \label{thm:running-1}
  Let $\ca{V}$ be a monoidal category with coproducts preserved by tensor product in each variable. We write $\VMat$ for the double category of $\ca{V}$-matrices with sets as objects, functions as vertical maps, and $\ca{V}$-matrices as horizontal maps. Explicitly, a $\ca{V}$-matrix $x \co A \tickar B$ from~$A$ to $B$ is a $\ca{V}$-functor $x \co B \times A \to \ca{V}$ where $A$ and $B$ are considered as discrete $\ca{V}$-categories. The value of such an $x$ at $(b,a) \in B \times A$ will be written as~$x[b;a]$\footnote{We use this slightly unorthodox notation to draw an analogy with symmetric sequences of~\cref{subsec:symseq}.}. With this notation, the composition of $x$ with $y \colon B \tickar C$ is given as to the left in 
  \begin{equation*}
    (y \circ x)[c; a] = \sum_{b \in B} y[c; b] \otimes x[b; a] \qquad \qquad (\hid_A)[a'; a] = \delta_{a'a}\rlap{ ,}
  \end{equation*}
  while the horizontal identity at $A$ is given as to the
  right; here, $\delta_{aa} = I$ and $\delta_{a'a} = 0$ for
  $a' \neq a$. Finally, a $2$-morphism $\alpha$ in $\VMat$ as to the
  left in~\eqref{eq:two-morphisms} is a $B \times A$-indexed
  family of morphisms
  $\alpha_{b,a} \co x[b;a] \to y[gb; fa]$ in $\ca{V}$.
For more details on this double category, see~\cite{Framedbicats,VCocats}.
\end{rex}

Any double category $\dcC$ has an underlying \emph{horizontal
  bicategory} $\mathcal{H}\dcC$ comprising the objects, horizontal
$1$-cells and globular $2$-morphisms of $\dcC$. We will make use of
the hom-categories of this bicategory, which we denote $\dcC[A,B]$
rather than $\mathcal{H}\dcC(A,B)$ for simplicity. Thus, $\dcC[A,B]$ is the category of horizontal
$1$-cells from $A$ to $B$ and globular $2$-morphisms between them, and we often refer to these as ``hom-categories'' of the double category.
For example, in $\VMat$, we have
  \begin{equation}
  \label{Matlp}
    \VMat[A, B] = \VCat(B \times A, \ca{V})
  \end{equation}
and $\ca{H}\VMat$ is the bicategory of matrices, see~\cite{BenabouJ:dis,Varthrenr}.

Fibrancy is one of the most fundamental assumptions
we make on our base double category. It is useful for numerous reasons, for example since 
it often allows us to deduce ``global''
universal properties from ``local'' ones, \cf \cref{prop:strongerfreemonads} below.

 \begin{defi}\label{defi:fibrant}
  A double category $\dcC$ is \emph{fibrant} if
  $(\mathfrak{s}, \mathfrak{t}) \colon \dcC_1 \rightarrow \dcC_0
  \times \dcC_0$ is a Grothendieck fibration.
\end{defi}
A fibrant double category is the same as a \emph{framed bicategory} in
the sense of~\cite[Section~4]{Framedbicats}, and as shown in
Theorem~4.1 of \emph{op.~cit.}, there are various ways of
reformulating the condition. In particular, we have for any vertical
$1$-cell $f \colon A \rightarrow B$, a pair of horizontal $1$-cells
$\comp{f} \colon A \tor B$ and $\coj{f} \colon B \tor A$ 
 and associated
$2$-morphisms $p_1$ and $q_1$, obtained as cartesian liftings of
$\hid_B \in \dcC_1$ along the maps $(f,1)$ and $(1,f)$ in $\dcC_0
\times \dcC_0$, as to the left in:
\begin{equation*}
  \begin{tikzcd}[column sep = scriptsize]
    A\ar[r,tick,"\comp{f}"]\ar[d,"f"']\ar[dr,phantom,"\Two p_1"] & {B}\ar[d,equal] & & & 
    B\ar[r,tick,"\coj{f}"]\ar[d,equal]\ar[dr,phantom,"\Two q_1"] & A\ar[d,"f"] \\
    {B}\ar[r,tick,"\hid_{B}"'] & {B} & & & 
    B\ar[r,tick,"\hid_B"'] & {B}
  \end{tikzcd} \qquad \qquad \qquad 
  \begin{tikzcd}[column sep = scriptsize]
    A\ar[r,tick,"\hid_{A}"]\ar[d,equal]\ar[dr,phantom,"\Two p_2"] & {A}\ar[d,"f"] & & & 
    A\ar[r,tick,"\hid_A"]\ar[d,"f"']\ar[dr,phantom,"\Two q_2"] & A\ar[d,equal] \\
    {A}\ar[r,tick,"\comp{f}"'] & {B} & & &
    B\ar[r,tick,"\coj{f}"'] & {A}\rlap{ .}
  \end{tikzcd}
\end{equation*}
Furthermore, we have $2$-morphisms $p_2, q_2$ obtained by factoring
the $2$-morphism $\hid_f \colon \hid_B \rightarrow \hid_B$ through these
respective cartesian lifts, as to the right above. The $2$-morphisms
$p_1$ and $p_2$, respectively $q_1$ and $q_2$, exhibit the horizontal
$1$-cells $\comp{f}$ and $\coj{f}$ as respectively the
\emph{companion} and the \emph{conjoint} of the vertical $1$-cell $f$.
So a fibrant double category has all companions and conjoints, and in
fact, this characterises fibrancy: for indeed, the cartesian lift of a
general map $(f,g) \colon (A,B) \rightarrow (C,D)$ of
$\dcC_0 \times \dcC_0$ at an object $y$ of $\dcC_1$ over $(C,D)$ can
be constructed by pasting the composite
\begin{equation*}
  \begin{tikzcd}
    A\ar[r,tick,"\comp{f}"]\ar[d,"f"']\ar[dr,phantom,"\Two p_1"] &
    {C}\ar[d,equal]  \ar[r,tick,"y"] &
    {D}\ar[d,equal]  \ar[r,tick,"\coj g"] \ar[dr,phantom,"\Two q_1"]&
    {C} \ar[d, "g"]
    \\
    {C}\ar[r,tick,"\hid_{C}"'] &
    {C} \ar[r,tick,"y"'] &
    {D}\ar[r,tick,"\hid_{D}"'] &
    {D}
  \end{tikzcd}
\end{equation*}
with the identity coherence $2$-morphisms of $\dcC$. So fibrant double
categories are precisely those with all companions and conjoints; and
by dualising appropriately, we see that these are, in turn, the same
as double categories for which
$(\mathfrak{s}, \mathfrak{t}) \colon \dcC_1 \rightarrow \dcC_0 \times
\dcC_0$ is a Grothendieck \emph{op}fibration. Notice that the fibre categories are precisely the above mentioned hom-categories $\dcC[A,B]$ of the double category.

\begin{rex} \label{thm:running-2}
  $\VMat$ is fibrant: given a vertical map $f \colon A \rightarrow B$, i.e., a function between sets, the corresponding $\comp f \colon A \tickar B$ and $\coj f \colon B \tickar A$ are the matrices given as follows, where $\delta$ is defined as before:
  \begin{equation*}
    \comp f[b; a] = \delta_{b, fa} \qquad \text{and} \qquad \coj f[a; b] = \delta_{fa, b}\rlap{ .}
  \end{equation*}
\end{rex}

\subsection{Monads}  In this paper we will be concerned with monads (and later bimodules)
in a double
category. It will be convenient to approach monads via the category of
horizontal endo-$1$-cells, see, \eg, \cite{Monadsindoublecats}, which we now recall.
For a double category $\dcC$, the ordinary category $\EndC$ of \emph{endomaps} in $\dcC$ is defined as follows:
  \begin{itemize}
  \item the objects are pairs $(A, x)$ of an object $A \in \dcC$ and a
    horizontal map $x \colon A \tickar A$;
  \item the morphisms $(f, \phi) \co (A, x) \rightarrow (B,y)$, are pairs consisting of
    a vertical map $f \colon A \rightarrow B$ and a $2$-morphism
    \begin{displaymath}
      \begin{tikzcd}
        A\ar[r,tick,"x"]\ar[d,"f"']\ar[dr,phantom,"\Two\phi"] & A\ar[d,"f"] \\
        B\ar[r,tick,"y"'] & B\rlap{ .}
      \end{tikzcd}
    \end{displaymath}
  \end{itemize}
 Equivalently, $\EndC$ can be defined as the following pullback of categories:
\begin{equation}\label{eq:endc-pullback}
  \begin{tikzcd}
    \EndC \ar[r]\ar[d,"\pi"'] & \dcC_1 \ar[d,"{(\mathfrak{s},\mathfrak{t})}"] \\
    \dcC_0 \ar[r,"\Delta"'] & \dcC_0 \times \dcC_0 \rlap{ .}
  \end{tikzcd}
\end{equation}

We now recall the notions of monad and monad map in $\dc{C}$, as
found, for example, in \cite[\S11]{Framedbicats} (there termed monoids and
monoid homomorphisms) or \cite[Def.~2.4]{Monadsindoublecats}
(termed monads and vertical monad maps).

\begin{defi}\label{def:monads}
  A \emph{monad} $(A,a, \mu, \eta)$ in $\dcC$ consists of an object
  $(A, a \co A \tickar A)$ of $\EndC$ together with globular
  2-morphisms
  \begin{equation*}
    \begin{tikzcd}
      A \ar[r, tick, "a"]  \ar[d, equal]  \ar[drr, phantom, "\Two \mu"] & A \ar[r, tick, "a"]  & A \ar[d, equal] \\
      A \ar[rr, tick, "a"'] &  & A 
    \end{tikzcd} \qquad \qquad 
    \begin{tikzcd}
      A \ar[r, tick, "\id_A"] \ar[d, equal]  \ar[dr, phantom, "\Two \eta"] & A \ar[d, equal] \\
      A \ar[r, tick, "a"'] & A 
    \end{tikzcd}
  \end{equation*}
  subject to the following
  standard associativity and unitality axioms:
  \begin{displaymath}
    \begin{tikzcd}[column sep = 1.5em]
      A\ar[r,tick,"a"]\ar[d,equal]\ar[drr,phantom,"\Two\mu"] & A\ar[r,tick,"a"] & A\ar[r,tick,"a"]\ar[d,equal] & A\ar[d, equal] \\
      A\ar[rr,tick,"a"]\ar[d,equal]\ar[drrr,phantom,"\Two\mu"] && A\ar[r,tick,"a"] & A\ar[d,equal] \\
      A\ar[rrr,tick,"a"'] &&& A \mathrlap{,}
    \end{tikzcd}=
    \begin{tikzcd}[column sep = 1.5em]
      A\ar[r,tick,"a"]\ar[d,equal] & A\ar[r,tick,"a"]\ar[d,equal]\ar[drr,phantom,"\Two\mu"] & A\ar[r,tick,"a"] & A\ar[d, equal] \\
      A\ar[r,tick,"a"]\ar[d,equal]\ar[drrr,phantom,"\Two\mu"] & A\ar[rr,tick,"a"] && A\ar[d,equal] \\
      A\ar[rrr,tick,"a"'] &&& A
    \end{tikzcd} \qquad 
    \begin{tikzcd}[column sep = 1.5em]
      A\ar[r,tick,"\id_A"]\ar[dr,phantom,"\Two\eta"]\ar[d,equal] & A\ar[r,tick,"a"]\ar[d,equal] & A\ar[d,equal] \\
      A\ar[r,tick,"a"]\ar[d,equal]\ar[drr,phantom,"\Two\mu"] & A\ar[r,tick,"a"] & A\ar[d,equal] \\
      A\ar[rr,tick,"a"'] &&A \mathrlap{,}
    \end{tikzcd}=\id_a=
    \begin{tikzcd}[column sep = 1.5em]
      A\ar[r,tick,"a"]\ar[d,equal] & A\ar[r,tick,"\id_A"]\ar[dr,phantom,"\Two\eta"]\ar[d,equal] & A\ar[d,equal] \\
      A\ar[r,tick,"a"]\ar[d,equal]\ar[drr,phantom,"\Two\mu"] & A\ar[r,tick,"a"] & A\ar[d,equal] \\
      A\ar[rr,tick,"a"'] &&A.
    \end{tikzcd}
  \end{displaymath}
  A \emph{monad map} $(f,\phi)$ between monads $(A,a, \mu, \eta)$
  and $(B,b,\mu,\eta)$ is a map
  $(f, \phi) \colon (A,a) \rightarrow (B,b)$ of $\EndC$ satisfying the
  following multiplication- and unit-preservation axioms:
  \begin{equation}
    \label{eq:monad-map-axioms}
    \begin{tikzcd}
      A\ar[r,tick,"a"]\ar[d,equal]\ar[drr,phantom,"\Two\mu"] & A\ar[r,tick,"a"] & A\ar[d,equal] \\
      A\ar[rr,tick,"a"] \ar[d,"f"']\ar[drr,phantom,"\Two\phi"] && A\ar[d,"f"] \\
      B\ar[rr,tick,"b"'] && B
    \end{tikzcd}=
    \begin{tikzcd}
      A\ar[r,tick,"a"]\ar[d,"f"']\ar[dr,phantom,"\Two\phi"] & A\ar[r,tick,"a"]\ar[d,"f"]\ar[dr,phantom,"\Two\phi"] & A\ar[d,"f"] \\
      B\ar[r,tick,"b"] \ar[d,equal]\ar[drr,phantom,"\Two\mu"] & B\ar[r,tick,"b"] & B\ar[d,equal] \\
      B\ar[rr,tick,"b"'] && B
    \end{tikzcd} \qquad\qquad   \begin{tikzcd}
      A\ar[r,tick,"\id_A"]\ar[d,equal]\ar[dr,phantom,"\Two\eta"] & A\ar[d,equal] \\
      A\ar[r,tick,"a"']\ar[d,"f"']\ar[dr,phantom,"\Two\phi"] & A\ar[d,"f"] \\
      B\ar[r,tick,"b"'] & B
    \end{tikzcd}=
    \begin{tikzcd}
      A\ar[r,tick,"\id_A"]\ar[d,"f"'] & A\ar[d,"f"] \\
      B\ar[r,tick,"\id_B"']\ar[d,equal]\ar[dr,phantom,"\Two\eta"] & B\ar[d,equal] \\
      B\ar[r,tick,"b"'] & B. 
    \end{tikzcd}
  \end{equation}
  We write $\MndC$ for the category of monads and monad maps in
  $\dcC$. We will generally write objects of $\MndC$ as $(A,a)$,
  suppressing the multiplication and unit $2$-morphisms.
\end{defi}
\begin{rex} \label{thm:running-3}
  The double category $\End(\VMat)$ is the category of $\ca{V}$-graphs, i.e., pairs of a set $A$ of ``vertices'' and an $A \times A$-indexed set of ``$\ca{V}$-objects of edges'' between them. On the other hand, a monad $(A,a)$ in $\VMat$ is precisely a $\ca{V}$-category $\ca{A}$ with object-set $A$, and hom-sets $\ca{A}(x,y) = a[x; y]$. The $2$-morphisms $\mu$ and $\eta$ provide the composition and identities, for which the monad axioms express associativity and unitality. A monad map $(f, \phi) \colon (A,a) \rightarrow (B,b)$ is precisely a $\ca{V}$-functor $\ca{A} \rightarrow \ca{B}$, with action on objects $f$ and action on homs $\phi$. Therefore $\Mnd(\VMat)=\VCat$.
\end{rex}
We have a commutative triangle of forgetful functors
\begin{equation}
  \label{eq:mnd-to-end-forgetful}
  \begin{tikzcd}[column sep = small]
    \MndC \ar[rr] \ar[dr] & & \EndC \ar[dl] \\ & \dcC_0 \mathrlap{,}
  \end{tikzcd}
\end{equation}
where the map across the top forgets the monad structure, and the maps
down to $\dcC_0$ forget the horizontal structure. We will write
$\EndCA$ and $\MndCA$ for the fibre categories of $\EndC$ and $\MndC$
over $A \in \dcC_0$. Note that $\EndCA$ is the hom-category $\dcC[A,A]$;
and, if we
view $\EndCA$ as monoidal under horizontal composition $\circ$, then $\MndCA$
is precisely the category of $\circ$-monoids in $\EndCA$.

\begin{rmk}
  \label{rmk:discrete}
  It is a trivial, but important, observation that the forgetful
  functor $\MndC \rightarrow \dcC_0$ has a left adjoint $\disc{}$.
  This adjoint sends $A \in \dcC_0$ to the identity monad
  $\disc A = (A, \hid_A)$, with unit and multiplication given by the
  appropriate coherence constraints in $\dcC$. When $\dcC = \VMat$,
  this $\disc A$ is equally well the discrete $\ca{V}$-category on the
  set $A$, and so we may refer to $\disc A$ as the \emph{discrete} monad on $A$.
\end{rmk}

The preceding remark tells us that the free monad on an object of
$\dcC_0$ always exists. However, in this paper, we will also require
free monads on endomaps, whose existence is an extra hypothesis -- see also \cite[Def.~2.8]{Monadsindoublecats}.

\begin{defi}\label{def:free-monads}
We say that the double category $\dcC$ \emph{admits free monads} if the forgetful functor $\MndC \rightarrow \EndC$
in~\cref{eq:mnd-to-end-forgetful} has a left adjoint over $\dcC_0$.
\end{defi}

This definition expresses the familiar universal property of a free
monad on an endomap, and will be used in \cref{sec:commuting} to
construct commuting tensor products of monads. However, it is often
easier in practice to verify the weaker condition that each forgetful
functor $\MndCA \rightarrow \EndCA$ has a left adjoint for every $A \in \dcC$. For example,
in $\VMat$, an object of $\EndCA$ is a $\ca{V}$-graph 
$x \in [A \times A, {\ca{V}}]$, and the free monad thereon is given by
the
usual free $\ca{V}$-category $\ca X \in \MndCA$ on this $\ca{V}$-graph (see e.g.~\cite{Wolff}), with homs
\begin{equation*}
  \ca X(a, a') = \sum_{\substack{a_0,\dots, a_n \in A\\ a_0 = a, a_n = a'}} x[a_{n-1}; a_n] \otimes \cdots \otimes x[a_{0}; a_1]
\end{equation*}
However, the weaker condition just described will imply the stronger
property of \cref{def:free-monads} if the double category $\dcC$ is
fibrant. 
 This result is due
to~\cite{Monadsindoublecats}; we sketch the proof as we will need some
of the details. Most importantly, we will have further use for the
following lemma:

\begin{lem}
  \label{lem:EndMndfib}
Assume $\dcC$ is fibrant. The functor $\EndC\to\dcC_0$ is a bifibration, $\MndC\to\dcC_0$ is a fibration and the
triangle~\cref{eq:mnd-to-end-forgetful} is a fibred functor.
\end{lem}
\begin{proof}
  Since $(\mathfrak{s},\mathfrak{t})$ is a Grothendieck bifibration,
  so is its pullback $\pi \colon \EndC \rightarrow \dcC_0$ as
  in~\cref{eq:endc-pullback}. Each fibre category $\EndCA = \dcC[A,A]$
  is, as noted above, monoidal under $\circ$, and for any
  $f \colon A \rightarrow B$ in $\dcC_0$, the reindexing functor
  $\comp f \circ (\thg) \circ \coj f \colon \EndCA[B] \rightarrow
  \EndCA$ can be made lax monoidal via the globular $2$-morphisms
  $p_1 \circ q_1 \colon \comp{f}\circ \coj{f} \Rightarrow \hid_B$ and
  $q_2 \circ p_2 \colon \hid_A \Rightarrow \coj{f}\circ \comp{f}$. It
  follows that taking monoids in each fibre yields a new fibration
  over the same base; this fibration is easily seen to be
  $\MndC \rightarrow \dcC_0$, which is thus a Grothendieck fibration,
  with~\cref{eq:mnd-to-end-forgetful} being the fibred forgetful
  functor. For more details, we refer the reader
  to~\cite[\S3]{Monadsindoublecats} and~\cite[\S3.2]{VCocats}.
\end{proof}

From this, we quickly conclude:
\begin{prop}\label{prop:strongerfreemonads}
 Assume $\dcC$ is fibrant and that for every $A \in \dcC$, the forgetful functor $\MndCA \rightarrow \EndCA$ has a left adjoint.
Then $\dcC$ admits free monads, namely those local free objects have the stronger universal property of providing the values of a left adjoint to the functor
  $\MndC \rightarrow \EndC$.
\end{prop}
\begin{proof}
  The result follows from the fact that a fibred functor between
  fibrations has a left adjoint if and only if it has a left adjoint
  on each fibre; see, for example~\cite[Proposition~8.4.2]{Handbook2}
  or~\cite[Lemma~1.8.9]{Jacobs}\footnote{Note that these references
    assume a further Beck--Chevalley condition on the fibrewise left
    adjoints, and conclude that the global left adjoint is itself a
    fibred functor; however, an inspection of the proofs show that
    they carry through without these further assumptions on each side.
    Indeed, the weaker result can be seen as a special case of the
    \emph{doctrinal adjunction} of~\cite{Doctrinal}.
  }.
\end{proof}

\subsection{Bimodules}\label{sec:bimodules-defn}  Under suitable assumptions on $\dcC$, the category of
monads in $\dcC$ is the category of objects and vertical
morphisms of a new double category $\BimC$, whose horizontal arrows
are the \emph{bimodules} between monads:

\begin{defi}\label{defi:bimodules}  \leavevmode
\begin{enumerate}[(i)]
\item Let $(A, a)$ and $(B, b)$ be monads in the double category $\dcC$. A $(b,a)$-\emph{bimodule} $p \co (A, a) \tickar (B,b)$ is a horizontal map $p \co A \tickar B$
equipped with a left $b$-action and a right $a$-action, \ie 2-cells
\[
\begin{tikzcd}
A \ar[r,tick, "p"] \ar[d, equal]  \ar[drr, phantom, "\Two \lambda"] & B \ar[r,tick, "b"] & B \ar[d, equal] \\
A \ar[rr, tick, "p"']  & & B \mathrlap{,}
\end{tikzcd} \qquad
\begin{tikzcd}
A \ar[r, tick, "a"] \ar[d, equal] \ar[drr, phantom, "\Two \rho"] & A \ar[r,tick, "p"] & B \ar[d, equal] \\
A \ar[rr, "p"']  & & B \mathrlap{,}
\end{tikzcd} 
\]
respectively, subject to associativity and unitality axioms, and the compatibility condition
\begin{equation}
\label{equ:action}
\begin{tikzcd}
b \circ p \circ a \ar[r, "b \circ \rho"] \ar[d, "\lambda \circ a"'] & b \circ p \ar[d, "\lambda"] \\
p \circ a \ar[r, "\rho"'] & p  \mathrlap{.}
\end{tikzcd}
\end{equation}
\item Let $p \co (A, a) \tickar (B, b)$ and $p' \co (A', a')  \tickar (B', b')$ be bimodules. A \emph{bimodule 2-morphism} 
\[
\begin{tikzcd}
(A,a) \ar[r, tick, "p"] \ar[d, "{(f, \phi)}"'] \ar[dr, phantom, "\Two \theta"] & (B, b) \ar[d, "{(g, \psi)}"]  \\
(A',a')  \ar[r, tick, "p' "'] & (B', b') 
\end{tikzcd}
\]
consists of monad maps $(f, \phi) \co (A, a) \to (A', a')$, $(g, \psi) \co (B,b) \to (B', b')$ and a 2-morphism
\[
\begin{tikzcd}
A \ar[r, tick, "p"] \ar[d, "f"'] \ar[dr, phantom, "\Two \theta"] & B \ar[d, "g"]  \\
A' \ar[r, tick, "p' "'] & B'
\end{tikzcd}
\]
in $\dcC$ subject to standard axioms (see \cite[Def.~11.1]{Framedbicats}).
\end{enumerate}
\end{defi}

\begin{notat} \label{thm:action-notation} For a bimodule $p \co (A,a) \tickar (B, b)$ as above,  we write 
\begin{equation*}
\alpha \co b \circ p \circ a \to p
\end{equation*} 
for the common value of the two composites in \eqref{equ:action} and refer to it
as the \emph{action} of the bimodule.
\end{notat}

The bimodules and bimodule $2$-morphisms in $\dcC$ organise themselves
into a category ${\BimC}_1$ with evident projection functors
$\mathfrak{s}, \mathfrak{t} \colon {\BimC}_1 \rightrightarrows \MndC$.
For suitable $\dcC$ this double graph will underlie a double
category $\BimC$, wherein the horizontal composition $q \bullet p$ of
bimodules $p \co (A,a) \tickar (B, b)$ and $q \co (B,b) \tickar (C,c)$
classifies ``bilinear module $2$-morphisms'', in the sense that
bimodule morphisms $q \bullet p \rightarrow r$ are in bijection with
$2$-morphisms $q \circ p \rightarrow r$ in $\dcC$ which are
equivariant with respect to the outer $a$- and $c$-actions but also
with respect to the inner $b$-actions on $p$ and $q$.

Sufficient assumptions on $\dcC$ to define $\BimC$ are that each hom-category
$\dcC[A,B]$ should have reflexive coequalisers, and these should be
preserved by pre- and post-composition by horizontal $1$-cells; this
is a special case of~\cref{defi:parallel_local} below, in whose
nomenclature we will say that $\dcC$ has ``stable local reflexive
coequalisers''. For such a $\dcC$, we can construct $q \bullet p
\colon (A, a) \tickar (C,c)$ as a
reflexive coequaliser
\begin{equation}
\label{equ:bim-comp}
  \begin{tikzcd}
    q \circ b \circ p  
    \ar[r, shift right = 1, "b \circ \rho"'] \ar[r, shift left =1, "\lambda \circ p"] &
    q \circ p \ar[r] &
    q \bimcomp p
  \end{tikzcd}
\end{equation}
in $\dcC[A,C]$, and equip it with the structure of a $(c, a)$-bimodule
using its universal property and the stability of the local reflexive
coequalisers. We can also have horizontal units: for a monad $(A,a)$,
its identity $(a,a)$-bimodule is the horizontal map
$a \co A \tickar A$ of $\dcC$ equipped with its regular left and right
actions. In this way, one obtains a double category of monads, monad
maps and bimodules:
 
\begin{prop}[Shulman] \label{thm:bim-c} Let $\dcC$ be a double
  category with stable local reflexive coequalisers. There is a double
  category $\BimC$ with monad as objects, monad morphisms as vertical
  maps, monad bimodules as horizontal maps and bimodule 2-morphisms as
  2-morphisms. Furthermore, $\BimC$ itself has stable local
  reflexive coequalisers.
\end{prop}

\begin{proof} The proof in~\cite[Theorem~11.5]{Framedbicats} carries
  over. The development therein assumes stable local coequalisers,
  rather than stable local \emph{reflexive} coequalisers, but only
  only the latter are used. Of course, since we only assume stable
  local reflexive coequalisers, that is all we get in $\BimC$ as well.
\end{proof}

Moreover, we have:
\begin{prop}[Shulman] \label{thm:bim-c-fibrant} Let $\dcC$ be a double
  category with stable local reflexive coequalisers. If $\dcC$ is
  fibrant, then so is $\BimC$. \qed
\end{prop}
Thus, for instance, we deduce in the situation of our running example
that $\VProf$ is fibrant when it exists.

\begin{rex} \label{thm:running-4}
  If $\ca{V}$ is a symmetric monoidal category whose tensor product
  preserves all colimits in each variable, then the double category
  $\VMat$ will have stable local reflexive coequalisers, and so
  supports the construction of the double category of bimodules. This
  is exactly the double category $\VProf$ of $\ca{V}$-profunctors. We
  have already seems that monads in $\VMat$ are
  $\ca{V}$-categories, and monad morphisms are $\ca{V}$-functors. Given
  monads $(A,a)$ and $(B,b)$, corresponding to $\ca{V}$-categories
  $\ca{A}$ and $\ca{B}$, a monad bimodule $(A,a) \tickar (B,b)$
  involves a horizontal arrow $p \colon A \tickar B$, i.e., a matrix
  of $\ca{V}$-objects $p[y; x]$ for $y \in B$ and $x \in A$, together
  with associative and unital actions by $a$ and $b$, which amounts to
  giving $\ca{V}$-morphisms
  \begin{equation*}
    b[y'; y] \otimes p[y; x] \rightarrow p[y'; x] \qquad \text{and} \qquad p[y; x] \otimes a[x; x'] \rightarrow p[y; x']
  \end{equation*}
  satisfying the evident axioms. These are precisely the data of a
  $\ca{V}$-functor $p \colon \ca{B}^\mathrm{op} \otimes \ca{A}
  \rightarrow \ca{V}$, \ie a $\ca{V}$-profunctor; see
  e.g.~\cite[Example~11.8]{Framedbicats} for more details.
\end{rex}
\begin{rmk} The restriction to stable local reflexive coequalisers (in~\cref{thm:bim-c}) rather than all coequalisers is necessary for
  our applications in \cref{sec:applications}. Indeed, we will see
  there that the double category $\VCatSym$ of categorical symmetric
  sequences has all local coequalisers, but only the reflexive ones
  need be stable.
\end{rmk}

We now recall the construction of free bimodules and some of their properties, as these will play an important
role in our development. We write
\begin{equation*}
U \co \BimC \to \dcC
\end{equation*}
for the evident forgetful double functor, whose components fit into the commutative diagram of categories and (ordinary) functors
\begin{equation}
\label{equ:forget-is-fibred}
\begin{tikzcd}[column sep = large]
\BimC_1 \ar[r, "U_1"] \ar[d,  "{(\mathfrak{s}, \mathfrak{t})}"'] & \dcC_1 \ar[d,  "{(\mathfrak{s}, \mathfrak{t})}"] \\
\BimC_0 \times \BimC_0 \ar[r, "U_0 \times U_0"'] & \dcC_0 \times \dcC_0 
\end{tikzcd}
\end{equation}
which amounts to $U=(U_0,U_1)$ preserving source and target. Notice that $U_0\colon\MndC\to\dcC_0$ coincides with the left leg of \cref{eq:mnd-to-end-forgetful}.

Let $(A,a)$ and $(B, b)$ be monads. By the notation introduced in \cref{sec:preliminiaries},
the hom-category $\BimC[(A,a), (B,b)]$ has $(b,a)$-bimodules as objects 
and  globular bimodule 2-morphisms as maps. There is a
forgetful functor 
\[
U_{a,b} \co \BimC[(A,a), (B,b)] \to \dcC[A,B]
\]
mapping a bimodule to its underlying horizontal map and a bimodule 2-morphism to its underlying 2-morphism. 

\begin{lem} \label{thm:free-bimodule} 
The functor $U_{a,b} \co \BimC[(A,a), (B,b)] \to \dcC[A,B]$ is monadic.
\end{lem}

\begin{proof}  We can define a monad on
$\dcC[A,B]$ mapping $x \co A \tickar B$ to $b \hcomp x \hcomp a \co A \tickar B$, with evident multiplication
and unit. The category of algebras for this monad is precisely the category $\BimC[(A,a), (B,b)]$
and thus the claim follows. Explicitly, the left adjoint functor sends $x \co A \tickar B$ to the free $(b,a)$-bimodule on $x$, which we write $F_{a,b}(x) \co (A, a) \to (B, b)$, or simply $F(x) \co (A, a) \to (B, b)$
when this does not cause confusion. This is defined as the horizontal composite
\begin{equation}
\label{equ:free-bimodule}
F_{a,b}(x) \defeq b \hcomp x \hcomp a \mathrlap{,} 
\end{equation}
with left and right actions given by the multiplication of the monads $a$ and $b$, respectively.
\end{proof} 

\begin{rmk} By \cref{thm:free-bimodule}, 
as a special case of a general result for the Eilenberg-Moore algebras of
a monad, every bimodule  is a reflexive coequaliser of free bimodules. Explicitly,
for a bimodule $p \co (A,a) \tickar (B, b)$, we have the reflexive coequaliser diagram
\begin{equation}
\label{equ:bim-as-coeq-of-free}
\begin{tikzcd}
F_{a,b}F_{a,b}(p) \ar[r, shift right, "\delta"'] \ar[r, shift left, "\gamma"] & 
F_{a,b}(p)  \ar[r, "\alpha"] &
p  
\end{tikzcd}
\end{equation}
 in $\BimC[(A,a), (B,b)]$, where  $\alpha$ is the action of the bimodule (\cf \cref{thm:action-notation}),
 $\gamma  = \mu \circ p \circ \mu$, and $\delta = b \circ \alpha \circ a$. This fact
will be used repeatedly in what follows.
\end{rmk}

\begin{rmk} The horizontal composite of two free bimodules is again a free bimodule.
Indeed, for free bimodules $F_{a,b}(x) \co (A,a) \to (B,b)$ and $F_{b,c}(y) \co (B, b) \to (C,c)$,
where $x \co A \tickar B$ and $y \co B \tickar C$ are horizontal maps in $\dcC$, 
there is a globular  isomorphism
\begin{equation}
\label{equ:bimcomp-of-free-is-free}
F_{b,c}(y) \bimcomp F_{a,b}(x) \cong   F_{a,c}(y \circ b \circ  x) \mathrlap{.} 
\end{equation}
in $\BimC[(A,a), (C,c)]$. 
\end{rmk}

It will be useful to record that, thanks to the fibrancy of $\dcC$ and of $\BimC$, 
the free bimodule on a horizontal map enjoys also a universal 
property that is stronger than the one given by the adjointness of \cref{thm:free-bimodule}.
By~\cref{equ:forget-is-fibred}, there is a diagram 
\[
\begin{tikzcd}[column sep = large]
\BimC_1 \ar[ddr, bend right = 8,  "{(\mathfrak{s}, \mathfrak{t})}"'] \ar[drr, bend left = 8,"U_1"] \ar[dr, "U'"] & & \\
 & \mathcal{E} \ar[r] \ar[d]  &  \dcC_1 \ar[d, "{(\mathfrak{s}, \mathfrak{t})}"]  \\
& \BimC_0 \times \BimC_0 \ar[r, "U_0 \times U_0"']  & \dcC_0 \times \dcC_0 \mathrlap{,}
\end{tikzcd}
\]
where the inner square is a pullback of categories.
Explicitly, $\mathcal{E}$ is the category
whose objects are triples of the form $( (A,a), (B, b), x )$, where $(A,a)$ and $(B,b)$ are 
monads and $x \co A \tickar B$ is a horizontal map, and whose
morphisms $(f, g, \phi) \co ( (A,a), (B, b), x ) \to ( (A',a'), (B', b'), x' )$ are 2-morphisms
\[
\begin{tikzcd}
A \ar[r,tick, "x"] \ar[d, "f"'] \ar[dr, phantom,"\Two \phi"] & B \ar[d, "g"]  \\
A' \ar[r,tick,"x'"'] & B' 
\end{tikzcd}
\]
in $\dcC$. The functor $U'$ acts like $U_1$. We can now state the
desired universal property, which will be useful in the proof of \cref{prop:commuting-bimodules-prerepresentable}. 
This is an instance of \cite[Theorem~3.4]{FibredAdjunctions}, but we give a proof for concreteness.

\begin{lem} \label{thm:extra-fibred-free-bimod}
The functor $U' \co \BimC_1 \to \mathcal{E}$ has a left adjoint.
\end{lem}

\begin{proof} Recall that $\BimC_1$ is fibred over $\BimC_0 \times \BimC_0$ since $\BimC$ is a fibrant double category and note that $\mathcal{E}$ is  fibred over $\BimC_0 \times \BimC_0$ as a pullback of a fibration. It can be checked that the functor $U'$, in 
\[
\begin{tikzcd}[column sep = small]
(\BimC)_1 \ar[rr, "U'"] \ar[dr, "{(\mathfrak{s}, \mathfrak{t})}"'] & & \ca{E} \ar[dl, "{(\mathfrak{s}, \mathfrak{t})}"] \\
 & \BimC_0 \times \BimC_0 & 
 \end{tikzcd} 
 \]
 is a fibred functor. Therefore, as in the proof of \cref{prop:strongerfreemonads}, it is sufficient to prove that $U'$ has a fibrewise left adjoint. 
 
 Let $(A,a)$ and $(B, b)$ be monads  and consider the fibers of $\BimC_1$ and $\ca{E}$ over
 $( (A,a), (B, b))$ together with the fiberwise $U'_{(A,a), (B,b)}$ between them. A left adjoint $F'_{(A,a), (B,b)}$
to $U'_{(A,a), (B,b)}$ can be defined by mapping a horizontal map $x \co A \tickar B$ to the free $(b,a)$-bimodule on it,
 $F_{a,b}(x) \co (A,a) \tickar (B, b)$. The required universal property asserts that
there is a natural family of bijections between  2-morphisms in $\BimC$ of the form
\[
\begin{tikzcd}
(A,a) \ar[r,tick, "F x"] \ar[d, "{(f, \phi)}"'] \ar[dr, phantom, "\Two"] & (B,b) \ar[d, "{(g, \psi)}"]  \\
(A',a')  \ar[r,tick, "p'"'] & (B', b')
\end{tikzcd}
\]
and 2-morphisms in $\dcC$ of the form
\[
\begin{tikzcd}
A \ar[r,tick, "x"] \ar[d, "f"'] \ar[dr, phantom, "\Two"] & B \ar[d, "g"]  \\
A' \ar[r, tick, "p'"'] & B' \mathrlap{,}
\end{tikzcd}
\]
 where  $x \co A \tickar B$ is horizontal map and  $p' \co (A', a') \tickar (B',b')$ is a bimodule.
This can be shown using companions and conjoints and \cref{thm:free-bimodule}, crucially exploiting the compatibility of forgetful functor $U$ with companions and conjoints.
\end{proof} 

\subsection{Limits and colimits in a double category}
For many constructions in this paper, we will require various kinds of
completeness and cocompleteness of our base double category. Firstly,
we have notions which relate to ``ordinary'' $1$-categorical limits
and colimits.

\begin{defi}\label{defi:parallel_local}
  Let $\dcC$ be a double category, and let $\mathcal{D}$ be a class of
  diagram-shapes (i.e., a class of small categories).
  \begin{itemize}[itemsep=0.25\baselineskip]
  \item We say that $\dcC$ has \emph{parallel
      $\mathcal{D}$-(co)limits} if $\dcC_1$ and $\dcC_0$ are
    equipped with chosen (co)limits of $\mathcal{D}$-shaped diagrams, and
    these are preserved strictly by the source and target functors
    $\mathfrak s$ and $\mathfrak t$. 
  \item We say that $\dcC$ has \emph{local $\mathcal{D}$-(co)limits}
    if each hom-category $\dcC[A,B]$ is equipped with chosen
    (co)limits of $\mathcal{D}$-shaped diagrams. We say that these local
    (co)limits are \emph{stable} if they are preserved by each pre-
    and postcomposition functor $\dcC[A,B] \rightarrow \dcC[A',B]$ and
    $\dcC[A,B] \rightarrow \dcC[A,B']$.
  \end{itemize}
\end{defi}

The term ``parallel (co)limit'' is due to Par\'e, see also \cite[Def.~3.1.9]{SweedlerTheoryDouble}. In fact, the strict preservation of chosen (co)limits by $\mathfrak s$ and $\mathfrak t$ is ensured when $\dc{C}$ is fibrant and $\mathfrak s$ and $\mathfrak t$ preserve (co)limits in the usual -- up to isomorphism -- sense.
Moreover, notice that what we here call ``stable local'' (co)limits are elsewhere called only ``local'', see e.g. \cite[Def.~11.4]{Framedbicats}, and coincide with what are commonly termed local (co)limits in the horizontal bicategory.
  
Parallel limits and colimits are ``global'', and so it should not be a
surprise that we can use fibrancy to deduce their existence from that
of ``local'' limits and colimits.

\begin{lem}
  \label{lem:local-to-global-limits}
  Assume that $\dcC$ is fibrant and let $\mathcal{D}$ be a class of
  diagram-shapes. If $\dcC_0$ has $\mathcal{D}$-(co)limits, and $\dcC$ has local
  $\mathcal{D}$-(co)limits, then $\dcC$ has parallel $\mathcal{D}$-(co)limits. 
\end{lem}
Thus, in our running example of $\VMat$, if $\ca{V}$ has
$\ca D$-(co)limits, then each hom-category
$\VMat[A,B] = \ca{V}^{B \times A}$ will have
$\ca D$-(co)limits; whence by this result $\VMat$ will have
parallel $\ca D$-(co)limits.
\begin{proof}
  The following is a standard fact about Grothendieck bifibrations
  $p \colon \ca{E} \rightarrow \ca{B}$: if $\ca{B}$ has a choice of
  $\mathcal{D}$-(co)limits, and each fibre category $\ca{E}_b$ has a
  choice of $\mathcal{D}$-(co)limits, then the total category $\ca{E}$
  has a choice of $\mathcal{D}$-(co)limits that is strictly preserved
  by $p$\footnote{This can be deduced, for example, from \cite[Cor.~3.7]{FibredAdjunctions} together with the fact that in a bifibration, reindexing functors for the fibration and the opfibration structure are adjoints to one another. See also~\cite[Prop.~3.1.13]{SweedlerTheoryDouble}.}. Apply this fact to the bifibration
  $(\mathfrak{s}, \mathfrak{t}) \colon \dcC_1 \rightarrow \dcC_0
  \times \dcC_0$, whose fibre categories are $\dcC[A,B]$.
\end{proof}

In this paper we will also make use of limits and colimits in $\EndC$
and $\MndC$. The following lemma allows us to ensure the necessary
completeness and cocompleteness, see also \cite[Prop.~3.1.11\&3.3.6]{SweedlerTheoryDouble}.

\begin{lem}
  \label{lem:limits-in-end-mnd}
  Let $\dcC$ be a double category, and $\mathcal{D}$ a class of diagram-shapes.
  \begin{enumerate}[(i)]
  \item If $\dcC$ has parallel $\mathcal{D}$-(co)limits, then $\EndC$
  has $\mathcal{D}$-(co)limits strictly preserved by $\EndC
  \rightarrow \dcC_0$.
  \item If $\dcC$ has parallel $\mathcal{D}$-limits, then $\MndC$
  has $\mathcal{D}$-limits which are strictly preserved by $\MndC
  \rightarrow \dcC_0$.
  \item If $\dcC$ has stable local reflexive coequalisers, free
    monads, and local $\mathcal{D}$-colimits, then each $\MndCA$ has $\mathcal{D}$-colimits.
  \end{enumerate}
\end{lem}
For example, if $\ca{V}$ has reflexive coequalisers preserved by tensor
in each variable, then $\VMat$ will have stable local reflexive
coequalisers, and we have already seen that it has free monads. Thus,
each $\MndCA$ will have any class of colimits that $\ca{V}$ has.
\begin{proof}
  For (i), the assumptions imply that~\eqref{eq:endc-pullback} is a
  pullback of strict $\mathcal{D}$-(co)limit-preserving functors: whence the
  pullback category $\EndC$ also has $\mathcal{D}$-(co)limits which are
  preserved by the projections to $\dcC_0$ and $\dcC_1$.
  For (ii), it is easy to see that, under no assumptions at all,
  the forgetful functor $\MndC \rightarrow \EndC$ creates limits
  strictly; now combine this with (i).

  For (iii), the assumptions imply that each $\EndCA=\dcC[A,A]$ has reflexive
  coequalisers preserved by the composition monoidal structure. It
  follows by~\cite[Theorem~2.3]{ColimitsMonoids} that the category of monoids $\MndCA$ has all
  coequalisers. Now since $U \colon \MndCA \rightarrow \EndCA$ has, as a result of our assumptions (\cref{def:free-monads}), a left adjoint $F$, it is monadic. It follows by
  the argument of~\cite{Linton1969Coequalizers} that $\MndCA$ has all
  $\mathcal{D}$-colimits, since $\EndCA$ does; explicitly, for $\mathcal{I} \in \mathcal{D}$,
  the colimit of a diagram $D \colon \ca{I} \rightarrow \MndCA$
  can be
  computed as the coequaliser of the  parallel arrows
  \begin{align*}
    F(\mathrm{colim}(UFUD)) &\xrightarrow{\qquad \qquad \ \ \quad F(\mathrm{colim}(U\varepsilon_D))\qquad \qquad \ \ \quad } F(\mathrm{colim}(UD)) \qquad \text{and} \\
    F(\mathrm{colim}(UFUD)) &\xrightarrow{F(\text{can})} FUF(\mathrm{colim}(UD)) \xrightarrow{\varepsilon_{UF(\mathrm{colim}(UD))}} F(\mathrm{colim}(UD))
  \end{align*}
  in $\MndCA$, where
  $\text{can} \colon \mathrm{colim}(UFUD) \rightarrow
  UF(\mathrm{colim}\ D)$ is the canonical comparison map.
\end{proof}

The following lemma that concerns connected colimits in the fibre and total category of $\dcC_1$ and $\Mnd(\dcC)$ above $\dcC_0$, will also be of use.
Notice that here, colimits in the base category are not required. 
 
\begin{lem}
  \label{lem:local-to-global-connected-limits}\label{thm:fibrancy-coequalizers}
  Assume that $\dcC$ is fibrant and that $\mathcal{D}$ is a class of
  \emph{connected} diagram-shapes. Then 
  \begin{enumerate}[(i)]
  \item a $\mathcal{D}$-colimit in $\dcC[A,B]$ is a colimit in $\dcC_1$, and   
  \item a $\mathcal{D}$-colimit in $\MndCA$ is a colimit in $\MndC$.
  \end{enumerate}
\end{lem}
In particular, if $\dcC$ has local $\mathcal{D}$-colimits (resp., each
$\MndCA$ has $\mathcal{D}$-colimits) then $\dcC_1$ (resp., $\MndC$)
has colimits of all identity-on-objects $\mathcal{D}$-shaped diagrams.

\begin{proof}
  This follows from another general fact about Grothendieck fibrations
  $p \colon \ca{E} \rightarrow \ca{B}$: namely, that the colimit of a
  connected diagram $D \colon \ca{I} \rightarrow \ca{E}_b$ in any
  fibre category is also a colimit in $\ca{E}$. Applying this to the
  fibrations
  $(\mathfrak s,\mathfrak t)\colon\dcC_1\to\dcC_0\times\dcC_0$ and
  $\Mnd(\dcC)\to\dcC_0$ yields the result.

  To verify the general fact, let
  $\bigl(q_i \colon Di \rightarrow \mathrm{colim}(D)\bigr)$ be a colimiting cocone in
  $\ca{E}_b$, and $(f_i \colon Di \rightarrow W)$ any cocone under
  $D$ in $\ca{E}$. By connectedness, we must have that
  $p(f_i) = p(f_j) \colon b \rightarrow p(W)$ for all
  $i,j \in \ca{I}$. Thus, by the universal property of
  reindexing we have the first natural bijection in:
  \begin{equation*}
    [\ca{I}, \ca{E}](D, \Delta W) \cong \sum_{h \colon b \rightarrow pW} [\ca{I}, \ca{E}_b](D, \Delta (h^\ast W)) \cong \sum_{h \colon b \rightarrow pW} \ca{E}_b(\mathrm{colim}(D), h^\ast W) \cong  \ca{E}(\mathrm{colim}(D), W)
  \end{equation*}
  wherein the other bijections are the universal property of the
  colimit, and the universal property of reindexing again.  
  It follows
  that $\bigl(q_i \colon Di \rightarrow \mathrm{colim}(D)\bigr)$ is
  also a colimit in $\ca{E}$. 
\end{proof}

Finally in this section, we record an additional kind of limit,
peculiar to double categories, which will be used in this paper; see \cite[Prop.~4.1]{Niefieldspancospan}.

\begin{defi}\label{def:1tabulators}
  A double category $\dcC$ has \emph{1-tabulators} if the
  identities functor $\hid \colon \dcC_0 \rightarrow \dcC_1$ has a
  right adjoint $T \colon \dcC_1 \rightarrow \dcC_0$. More explicitly,
  this means that for every horizontal arrow $x \colon A \tor B$ of $\dcC$,
  there is an object $Tx \in \dcC$, called the \emph{$1$-tabulator} of $x$,
  equipped with a $2$-morphism
  \begin{equation*}
    \begin{tikzcd}[column sep = scriptsize]
      Tx\ar[r,tick,"\hid_{Tx}"]\ar[d,"\pi_{\mathfrak s}"']\ar[dr,phantom,"\Two \pi"] & {Tx}\ar[d,"\pi_{\mathfrak{t}}"] \\
      {A}\ar[r,tick,"x"'] & {B}
    \end{tikzcd}
  \end{equation*}
  which is universal among $2$-morphisms into $x$ from a horizontal
  identity $1$-cell.
\end{defi}

\begin{rex} \label{thm:running-5}
  In $\VMat$, each horizontal arrow
  $x \colon A \tickar B$ has a 1-tabulator given by the set
  \begin{equation*}
    Tx \defeq \{\, (a, b, t) \mid a \in A, b \in B, t \colon I \rightarrow x[b; a] \text{ in } \ca{V}\,\}\rlap{ .}
  \end{equation*}
  The maps $\pi_{\mathfrak s} \co Tx \to A$ and $\pi_{\mathfrak t} \co Tx \to B$ are the two evident projections, and $\pi$
  is the family of $\ca{V}$-morphisms
  \begin{equation*}
    \pi_{(a,b,t), (a',b',t')} \colon \delta_{(a,b,t), (a',b',t')} \rightarrow x[b'; a]
  \end{equation*}
  composed of the unique map out of the initial object when $(a,b,t) \neq (a',b',t')$, and of $t \colon I \rightarrow x[b; a]$ when $(a,b,t) = (a',b',t')$.
\end{rex}

\subsection{Normal oplax monoidal structure}
The other main requirement we will impose on our base double category
is what in~\cite{Paper1} was termed \emph{normal oplax monoidal
  structure}. In this section, we recall the relevant definitions, and
extend these by introducing notions of \emph{closedness} and
\emph{symmetry} for normal oplax monoidal double categories, generalising
those for (pseudo) monoidal double categories.

In the following definition, we make use of the standard notions of
\emph{oplax double functor} and \emph{vertical transformation}; see,
for example,~\cite[Definitions.~3.1~\&~3.5]{Paper1}.

\begin{defi}\label{def:oplaxmonoidal}
Let $\dc{C}$ be a double category. An \emph{oplax monoidal structure} on $\dc{C}$ consists of:
\begin{itemize}
\item an oplax double functor $\bxt\colon\dc{C}\times\dc{C}\to\dc{C}$;
\item an oplax double functor $I\colon\one\to\dc{C}$; and
\item invertible vertical transformations $\alpha \co \mathord{\bxt}\circ(1\times\mathord{\bxt}) \Rightarrow \mathord{\bxt}\circ(\mathord{\bxt}\times1)$,
$\lambda \co \mathord{\bxt}\circ(I\times 1) \Rightarrow 1$ and $\rho \co \mathord{\bxt}\circ(1 \times I)\Rightarrow 1$,
\end{itemize}
where these data are subject to the usual  coherence axioms for $\alpha$, $\lambda$ and $\rho$.
This amounts to the following:
\begin{itemize}
 \item monoidal structures $(\bxt_0,I_0)$ and $(\bxt_1,I_1)$ on the categories $\dc{C}_0$ and $\dc{C}_1$;
 \item strict monoidality of $\mathfrak{s}, \mathfrak{t} \colon \dc{C}_1 \rightrightarrows \dc{C}_0$;
 \item globular 2-morphisms
 \begin{displaymath}
\begin{tikzcd}[column sep=.4in]
 A_1\bxt A_2\ar[rr,tick,"(y_1\circ x_1)\bxt(y_2\circ x_2)"]\ar[d,equal]\ar[drr,phantom,"\scriptstyle\Downarrow\xi"] && C_1\bxt C_2\ar[d,equal] \\
 A_1\bxt A_2\ar[r,tick,"x_1\bxt x_2"'] & B_1\bxt B_2\ar[r,tick,"y_1\bxt y_2"'] & C_1\bxt C_2
\end{tikzcd}\quad\text{and} \quad
\begin{tikzcd}[column sep= scriptsize]
 A_1\bxt
A_2\ar[rr,tick,"\hid_{A_1}\bxt\hid_{A_2}"]\ar[d,equal]\ar[drr,phantom,
"\scriptstyle\Downarrow\upsilon"] && A_1\bxt A_2\ar[d,equal] \\
 A_1 \bxt A_2 \ar[rr,tick,"\hid_{A_1\bxt A_2}"'] && A_1\bxt A_2 \mathrlap{,}
\end{tikzcd}
\end{displaymath}
\begin{displaymath}
\begin{tikzcd}
I_0\ar[d,equal]\ar[rr,tick,"I_1"]\ar[drr,phantom,
"\scriptstyle\Downarrow\delta"] && I_0\ar[d,equal] \\
 I_0\ar[r,tick,"I_1"'] & I_0\ar[r,tick,"I_1"'] & I_0
 \end{tikzcd}\quad\text{and} \quad
 \begin{tikzcd}
I_0\ar[d,equal]\ar[rr,tick,"I_1"]\ar[drr,phantom,
"\scriptstyle\Downarrow\iota"]  && I_0\ar[d,equal] \\
I_0\ar[rr,tick,"\hid_{I_0}"'] && I_0 \mathrlap{,}
 \end{tikzcd}
\end{displaymath}
subject to axioms that make $\bxt$ and $I$ into oplax double functors,
and make the associativity and unitality constraints of $\bxt_0$ and
$\bxt_1$ into the components of vertical transformations $\alpha$,
$\lambda$ and $\rho$.
\end{itemize}
\end{defi}

If $\dcC$ has an oplax monoidal structure $(\bxt, I)$, then it is
clear that the monoidal structure on $\dcC_1$ restricts back to a
monoidal structure on the subcategory $\EndC$: but due to the
oplaxity, this monoidal structure will \emph{not} lift to $\MndC$.
However, we can obtain a partial
lifting of $\bxt$ by assuming that the oplax monoidal structure is
\emph{normal} in the following sense (\cf~\cite[Definition~4.3]{Paper1}):

\begin{defi}
\label{def:normality}
An oplax monoidal double category $\dcC$ is said to be \emph{normal} if:
\begin{enumerate}[(i)]
\item \label{normal-i} $I \colon \mathbf{1} \rightarrow \dc{C}$ is a (pseudo) double functor;
\item \label{normal-ii} for all objects $A_1,A_2 \in \dc{C}$, the following restricted oplax double functors are (pseudo) double functors:
  \begin{equation*}
    \begin{aligned}
      A_1 \bxt (\thg) &\colon \dc{C} \cong \mathbf{1} \times \dc{C} \xrightarrow{A_1 \times \dc{C}} \dc{C} \times \dc{C} \xrightarrow{\bxt} \dc{C} \\[-0.3em]
      \text{and} \qquad 
      (\thg) \bxt A_2 &\colon \dc{C} \cong \dc{C} \times \mathbf{1}
      \xrightarrow{\dc{C} \times A_2} \dc{C} \times \dc{C} \xrightarrow{\bxt} \dc{C}\mathrlap{ .}
    \end{aligned}
  \end{equation*}
\end{enumerate}
Said another way, all three $2$-morphisms $\delta$, $\iota$ and $\upsilon$
from \cref{def:oplaxmonoidal} are invertible, while each~$\xi$ for which
$x_1 = y_1 = \hid_{A_1}$ or $x_2 = y_2 = \hid_{A_2}$ is also
invertible. 
\end{defi}

\begin{rex} \label{thm:running-6}
  If the $\ca{V}$ of our running example is \emph{symmetric} (or
  braided) monoidal, then the double category $\VMat$ is normal oplax
  monoidal, with unit the one-object set $1$, with tensor product
  $\bxt$ given on objects and vertical morphisms by cartesian product
  of sets, and with tensor of horizontal morphisms
  $x \colon A \tickar A'$ and $y \colon B \tickar B'$ given by
  \begin{equation*}
    (x \bxt y)\bigl[(a',b'); (a,b)\bigr] = x[a'; a] \otimes y[b'; b]\rlap{ ;}
  \end{equation*}
  the extension of this to $2$-morphisms is evident and left to the
  reader. In this case, it is obvious that there are invertible cells
  $\nu$, $\iota$ and $\delta$ as in~\cref{def:oplaxmonoidal}, and we
  can define $\xi$ using the symmetry isomorphisms of $\ca{V}$. In
  this case, every component of $\xi$ is invertible, so that we have
  not just a normal oplax monoidal double category, but a (pseudo)
  monoidal double category: see, for instance,
  \cite[Examples~3.2.2]{SweedlerTheoryDouble}.

  A more general situation is that $\ca{V}$ is a (non-symmetric,
  non-braided) normal duoidal category $(\ca{V}, \circ, I, \bxt, J)$
  wherein $\circ$ distributives over coproducts. Then we can use the
  $\circ$-monoidal structure to form $\VMat$, and can use the
  $\bxt$-monoidal structure to define a normal oplax monoidal
  structure on $\VMat$; so now
  $(x \bxt y)\bigl[(a',b'); (a,b)\bigr] = x[a'; a] \bxt y[b'; b]$. In this
  case, the maps $\xi$ are defined using the 
  interchangers of the duoidal category $\ca{V}$, which are not always
  invertible, and so we get a normal oplax monoidal double category
  which is \emph{not} (pseudo) monoidal.

  For the purposes of our running example, it will be clearer if we
  focus on the more familiar case of a symmetric monoidal $\ca{V}$,
  despite it not illustrating the full generality of our results;
  however, we will also comment on the normal duoidal case where it
  may add further insight.
\end{rex}

Recall from Remark~\ref{rmk:discrete} that $\disc A$ is
our notation for $(A, \hid_A)$, seen as an object 
of $\MndC$.

\begin{lem}
  \label{lem:tensoring-with-id-lifts}
  Assume the oplax monoidal structure $(\bxt, I)$ on the double category
  $\dcC$ is normal. Then for each $A \in \dcC$, the functor $(\thg) \bxt
  B \colon \dcC_0 \rightarrow \dcC_0$ lifts to a
  functor $(\thg) \bxt \disc B \colon \MndC \rightarrow \MndC$. Similarly, there is a functor $\disc A\bxt(\thg)\colon\MndC\to\MndC$.
\end{lem}
In our running example, the functor $(\thg) \bxt \disc{B}$ sends a
$\ca{V}$-category $\ca A$ to the coproduct of $B$ copies of $\ca A$; thus, it yields the $\ca{V}$-category with object-set $A \times B$ and hom-object
from $(a,b)$ to $(a',b')$ given by $\ca{A}(a,a')$ if $b = b'$ and $0$
otherwise.
\begin{proof}
  Since 
  $(\thg) \bxt B \colon \dcC \rightarrow \dcC$ is a
  \emph{pseudo} double functor by Definition~\ref{def:normality}\ref{normal-ii},
  it sends monads to monads, e.g.~\cite[Prop.~3.12]{VCocats}. Thus, if
  $(A, a)$ is a monad in $\dcC$, then we make $(A, a) \bxt \disc{B}  =
  (A \bxt B, a \bxt \hid_B)$ into a monad via the maps
  \begin{equation*}
    \begin{tikzcd}[column sep= scriptsize]
      A\bxt
      B\ar[rr,tick,"\hid_{A \bxt B}"]\ar[d,equal]\ar[drr,phantom,
      "\scriptstyle\Downarrow\upsilon^{-1}"] && A\bxt B\ar[d,equal] \\
      A\bxt B \ar[rr,tick,"\hid_{A}\bxt \hid_{B}"]\ar[d,equal]\ar[drr,phantom,
      "\scriptstyle\Downarrow\eta \bxt \hid_B"] &&A\bxt B \ar[d,equal]\\
      A\bxt B \ar[rr,tick,"a \bxt \hid_{B}"'] && A\bxt B
    \end{tikzcd} \qquad \qquad
    \begin{tikzcd}[column sep=.4in]
      A\bxt B\ar[r,tick,"a \bxt \hid_B"] 
      \ar[d,equal]\ar[drr,phantom,"\scriptstyle\Downarrow\xi^{-1}"] &
      A\bxt B\ar[r,tick,"a \bxt \hid_B"] &
      A\bxt B \ar[d,equal]\\
      A\bxt B\ar[d,equal]\ar[rr,tick,"(a\circ a)\bxt(\hid_B\circ \hid_B)"]
      \ar[drr,phantom,"\scriptstyle\Downarrow\mu \circ \mathfrak{l}"]&&
      A\bxt B \ar[d,equal]\\
      A\bxt B\ar[rr,tick,"a \bxt \hid_B"'] &&
      A\bxt B
\end{tikzcd}
  \end{equation*}
  where $\mathfrak l$ is a unit coherence constraint for the double
  category $\dcC$.
\end{proof}

We will also need a notion of \emph{closedness} for oplax monoidal
double categories. For (pseudo) monoidal double categories, this was
given in~\cite[Definition~3.2.9]{SweedlerTheoryDouble}, and the below is the
evident adaptation of this to the oplax monoidal case.

\begin{defi}\label{def:closeddouble}
  An oplax monoidal double category $(\dc{C}, \bxt, I)$ is said to be
  \emph{left closed} if
  the monoidal categories $(\dc{C}_0, \bxt_0, I_0)$ and $(\dc{C}_1,
  \bxt_1, I_1)$ admit left closures
  \begin{equation*}
    (\thg) \bxt_0 B \dashv \boxhom{B, \thg}_0 \qquad \text{and} \qquad 
    (\thg) \bxt_1 y \dashv \boxhom{y, \thg}_1
  \end{equation*}
  which are strictly preserved by source and target; \ie, we have a serially commuting diagram
  \begin{equation}
    \label{eq:hom-double-functor}
    \begin{tikzcd}
      \dcC_1^\mathrm{op} \times \dcC_1
      \ar[r,"\boxhom{\thg, \thg}_1"]
      \ar[d,shift left=1,"\mathfrak{t}"]
      \ar[d,shift right=1,"\mathfrak{s}"'] &
      \dcC_1
      \ar[d,shift left=1,"\mathfrak{t}"]
      \ar[d,shift right=1,"\mathfrak{s}"'] \\
      \dcC_0^\mathrm{op} \times \dcC_0
      \ar[r,"\boxhom{\thg, \thg}_0"] &
      \dcC_0
    \end{tikzcd}
  \end{equation}
  for which $\mathfrak{s}$ and $\mathfrak{t}$ furthermore send the
  evaluation maps $\ev_{y,z} \colon \boxhom{y, z}_1 \bxt_1 y \rightarrow z$ in
  $\dcC_1$ to the corresponding evaluation maps
  $\ev_{\mathfrak{s}(y), \mathfrak{s}(z)}$ and $\ev_{\mathfrak{t}(y), \mathfrak{t}(z)}$ in $\dcC_0$.
\end{defi}
\begin{rex} \label{thm:running-7}
  If $\ca{V}$ is a symmetric monoidal closed category with products,
  then $(\VMat, \bxt, 1)$ is left closed;
  see~\cite[Examples~3.2.12]{SweedlerTheoryDouble}. At the level of
  objects and vertical maps, this closure is simply the cartesian
  closure of $\Set$; we write $\boxhom{B,C}$ for the set of functions
  from $B$ and $C$ and $\lambda(f) \co A \to \boxhom{B, C}$ for the
  adjoint transpose of a function $f \co A \times B \to C$.

  At the level of horizontal maps, the internal hom of matrices
  $ y \co B \tickar B'$ and $z \co C \tickar C'$ is the matrix
  $\boxhom{y, z} \co \boxhom{B, C} \tickar \boxhom{B', C'}$ defined as follows; here, $[\thg, \thg]$ denotes the internal hom in $\ca{V}$.
  \begin{equation*}
    \boxhom{y,z}[g'; g] \defeq
    \prod_{b \in B, b'  \in B'}
    \Big[
    y[b'; b] , \
    z\big[g'(b'); g(b)\big] 
    \Big]\rlap{ .}
  \end{equation*}
  Note that all of this continues to work for a normal duoidal $(\ca{V}, \circ, I, \bxt, J)$---so long as we now assume that the $\bxt$-tensor product is left closed.
\end{rex}

When the oplax monoidal structure on $\dcC$ is closed, the internal
hom $\boxhom{\thg, \thg}$ of $\dcC_1$ clearly restricts
back to one on the subcategory $\EndC$; but once again, this internal
hom does \emph{not} lift to $\MndC$. Now, as explained
in~\cite[Theorem~3.2.7]{SweedlerTheoryDouble}, the oplax double
functoriality constraints of $\bxt$ transport under the mates
correspondence to \emph{lax} double functoriality constraints on the
morphism of $\Cat$-graphs~\cref{eq:hom-double-functor}. In other
words, we have a lax ``hom double functor''
$\boxhom{\thg, \thg} \colon \dcC^\mathrm{op} \times \dcC \rightarrow
\dcC$, and using this, we are able to obtain a \emph{partial} lifting of the
internal hom to monads, by analogy with
\cref{lem:tensoring-with-id-lifts}.

\begin{lem}
  \label{lem:homming-out-of-id-lifts}
  Let $(\dcC, \bxt, I)$ be an oplax monoidal double category.
  \begin{enumerate}[(i),itemsep=0.25\baselineskip]
  \item If
  $(\dcC, \bxt, I)$ is closed, then, for every $B \in \dcC$, the functor
  $\boxhom{ B, \thg } \colon \dcC_0 \rightarrow \dcC_0$ lifts to
  a functor $\boxhom{ \disc{B}, \thg } \co \MndC \rightarrow \MndC$.
\item If $(\dcC, \bxt, I)$ is also normal, then the
  adjunction
  $(\thg) \bxt B \dashv \boxhom{B, \thg} \colon \dcC_0 \rightarrow
  \dcC_0$ lifts to an adjunction
  \begin{equation*}
    (\thg) \bxt \disc{B} \dashv \boxhom{\disc{B}, \thg} \colon \MndC \rightarrow \MndC\rlap{ .}
  \end{equation*}
  \end{enumerate}
\end{lem}
\begin{rex}
  \label{ex:running16}In the case of $\VMat$, the functor $\boxhom{\disc{B}, \thg}$ sends a $\ca{V}$-category $\ca{C}$ to the $\ca{V}$-category $\ca{C}^B$, whose objects are $B$-indexed families of objects of $\ca{C}$, and whose homs are given by $\ca{C}^B(\vec x, \vec y) = \prod_{b \in B} \ca{C}(x_b, y_b)$.
\end{rex}
\begin{proof}
  For (i), composing the lax double functor
  $\boxhom{\thg, \thg} \colon \dcC^\mathrm{op} \times \dcC \rightarrow
  \dcC$ with the (pseudo) double functor $\mathbf{1} \rightarrow \dcC$
  picking out the object $B$ yields a lax double functor
  $\boxhom{ B, \thg } \colon \dcC \rightarrow \dcC$, which thus sends
  monads to monads. More explicitly, if $(C,c)$ is a monad in $\dcC$,
  then we make
  $\boxhom{ \disc{B}, (C,c) } = \bigl( \boxhom{B,C}, \boxhom{ \hid_B,
    c} \bigr)$ into a monad via the transposes of the following
  $2$-morphisms under the adjunction
  $(\thg) \bxt \hid_B \dashv \boxhom{\hid_B, \thg} \colon
  \dcC_1 \rightarrow \dcC_1$:
  \begin{equation}
    \label{eq:mu-for-internal-hom}
    \begin{tikzcd}[column sep = 5em]
      \boxhom{B,C} \bxt B \ar[r,tick,"{\hid_{\boxhom{B,C}} \bxt \hid_B}"]  \ar[d,equal]
      \ar[dr,phantom,"\Two \upsilon"] &
      \boxhom{B,C} \bxt B \ar[d,equal] \\
      \boxhom{B,C} \bxt B \ar[r,tick,"{\hid_{\boxhom{B,C} \bxt B}}"]  \ar[d,"\ev"'] 
      \ar[dr,phantom,"\Two\hid_{\ev}"] &
      \boxhom{B,C} \bxt B \ar[d,"\ev"] \\
      C \ar[r,tick,"{\hid_C}"]  \ar[d,equal]
      \ar[dr,phantom, "\Two \eta"] &
      C \ar[d,equal] \\
      C \ar[r,tick,"{c}"'] &
      C
    \end{tikzcd}
\quad \qquad     \begin{tikzcd}[column sep = 10em]
      \boxhom{B,C} \bxt B \ar[r,tick,"{(\boxhom{\hid_B, c} \circ \boxhom{\hid_B, c}) \bxt \hid_B}"]  \ar[d,equal]
      \ar[dr,phantom,"\Two \zeta"] &
      \boxhom{B,C} \bxt B \ar[d,equal] \\
       \boxhom{B,C}  \bxt B \ar[r,tick,"{(\boxhom{\hid_B, c} \bxt \hid_B) \circ (\boxhom{\hid_B, c} \bxt \hid_B)}"]  \ar[d,"\ev"']
      \ar[dr,phantom, "\Two \ev \circ \ev"] &
       \boxhom{B,C}  \bxt B \ar[d,"\ev"] \\
      C \ar[r,tick,"{c \circ c}"]  \ar[d,equal]
      \ar[dr,phantom, "\Two \mu"] &
      C \ar[d,equal] \\
      C \ar[r,tick,"{c}"'] &
      C\rlap{ ,}
    \end{tikzcd} 
  \end{equation}
  where $\zeta$ to the right is the composite of 
  $(\boxhom{\hid_B, c} \circ \boxhom{\hid_B, c}) \bxt \hid_B \cong
  (\boxhom{\hid_B, c} \circ \boxhom{\hid_B, c}) \bxt (\hid_B \circ
  \hid_B)$ with the interchanger $\xi$.
  
  For (ii), note that the (pseudo) double functor $(\thg) \bxt B
  \colon \dcC \rightarrow \dcC$ and
  the lax double functor $\boxhom{B, \thg} \colon \dcC \rightarrow
  \dcC$ are adjoint in the $2$-category of double categories, lax
  double functors and vertical transformations. The process of taking
  monads is a $2$-functor from this $2$-category to the $2$-category
  of categories, and so we have $(\thg) \bxt \disc{B} \dashv
  \bxt{\disc{B}, \thg}$ as desired.
\end{proof}

Finally in this section, we define what it means for an oplax monoidal
double category to be \emph{braided} and
\emph{symmetric}---generalising, for example,
\cite{ConstrSymMonBicatsFun} to our more general oplax monoidal context.

\begin{defi}\label{def:symmetricoplaxmon}
Let $(\dcC,\bxt,I)$ be an oplax monoidal double category. A
\emph{braiding} on $\dcC$ is an invertible vertical transformation
$\beta\colon\bxt\Rightarrow\bxt\circ(\pi_2, \pi_1)$,
where $(\pi_2, \pi_1)\colon\dcC\times\dcC\to\dcC\times\dcC$ is the
double functor which interchanges its two arguments, satisfying the usual axioms.
More elementarily, this amounts to:
\begin{itemize}
\item braidings for the  monoidal structures on the categories $\dcC_0$ and $\dcC_1$;
\item preservation of the braiding by the strict monoidal functors
  $\mathfrak{s}, \mathfrak{t} \colon \dc{C}_1 \rightrightarrows \dc{C}_0$,
  which is to say that the components of the braiding for $\dcC_1$
  have those for $\dc{C}_0$ as their vertical components:
\begin{equation}\label{eq:braidingcomponents}
\begin{tikzcd}[column sep=.6in]
A_1\bxt A_2\ar[d,"\beta_{A_1,A_2}"']\ar[r,tick,"x_1\bxt x_2"]\ar[dr,phantom,"\Two\beta_{x_1,x_2}"] & B_1\bxt B_2\ar[d,"\beta_{B_1,B_2}"] \\
A_2\bxt A_1\ar[r,tick,"x_2\bxt x_1"'] & B_2\bxt B_1;
\end{tikzcd}
\end{equation}
\item the following axioms, expressing that the braiding maps $\beta$
  constitute a vertical natural transformation:
\begin{displaymath}
\begin{tikzcd}[column sep=.5in]
A_1\bxt A_2\ar[d,equal]\ar[rr,tick,"(y_1\circ x_1)\bxt(y_2\circ x_2)"]\ar[drr,phantom,"\scriptstyle\Downarrow\xi"] && C_1\bxt C_2\ar[d,equal] \\
A_1\bxt A_2\ar[r,tick,"x_1\bxt x_2"']\ar[d,"\beta"']\ar[dr,phantom,"\Two\beta"] & B_1\bxt B_2\ar[r,tick,"y_1\bxt y_2"']\ar[d,"\beta"']\ar[dr,phantom,"\Two\beta"] & C_1\bxt C_2\ar[d,"\beta"] \\
A_2\bxt A_1\ar[r,tick,"x_2\bxt x_1"'] & B_2\bxt B_1\ar[r,tick,"y_2\bxt y_1"'] & C_2\bxt C_1
\end{tikzcd}=
\begin{tikzcd}[column sep=.5in]
A_1\bxt A_2\ar[d,"\beta"']\ar[rr,tick,"(y_1\circ x_1)\bxt(y_2\circ x_2)"]\ar[drr,phantom,"\Two\beta"] && C_1\bxt C_2\ar[d,"\beta"] \\
A_2\bxt A_1\ar[d,equal]\ar[rr,tick,"(y_2\circ x_2)\bxt(y_1\circ x_1)"']\ar[drr,phantom,"\Two\xi"] && C_2\bxt C_1\ar[d,equal] \\
A_2\bxt A_1\ar[r,tick,"x_2\bxt x_1"'] & B_2\bxt B_1\ar[r,tick,"y_2\bxt y_1"'] & C_2\bxt C_1 \mathrlap{,}
\end{tikzcd}
\end{displaymath}
\begin{displaymath}
\begin{tikzcd}[column sep=.5in]
A\bxt B\ar[d,"\beta"']\ar[r,tick,"\id_A\bxt\id_B"]\ar[dr,phantom,"\Two\beta"] & A\bxt B\ar[d,"\beta"] \\
B\bxt A\ar[d,equal]\ar[r,tick,"\id_B\bxt\id_A"']\ar[dr,phantom,"\Two\eta"] & B\bxt A\ar[d,equal] \\
B\bxt A\ar[r,tick,"\id_{B\bxt A}"'] & B\bxt A
\end{tikzcd}=
\begin{tikzcd}[column sep=.5in]
A\bxt B\ar[d,equal]\ar[r,tick,"\id_A\bxt\id_B"]\ar[dr,phantom,"\Two\eta"] & A\bxt B\ar[d,equal] \\
A\bxt B\ar[r,tick,"\id_{A\bxt B}"']\ar[d,"\beta"']\ar[dr,phantom,"\Two\id_\beta"] & A\bxt B\ar[d,"\beta"] \\
B\bxt A\ar[r,tick,"\id_{B\bxt A}"'] & B\bxt A\rlap{ .}
\end{tikzcd}
\end{displaymath}
\end{itemize}

We say that the braiding $\beta$ is a \emph{symmetry} when
$\beta^{-1} = \beta \circ (\pi_2, \pi_1)$; equivalently, we may ask
that the braided monoidal structures on $\dcC_0$ and $\dcC_1$ are
symmetric.
\end{defi}

\begin{rex} \label{thm:running-8}
  If $\ca{V}$ is a braided or symmetric monoidal category, then $\VMat$ will be braided or symmetric as a normal oplax double category, with the braiding maps given by those of $\Set$ at the vertical level, and those of $\ca{V}$ at the horizontal level. Of course, the situation here is slightly degenerate; for the more general case of a normal duoidal base $(\ca{V}, \circ, I, \bxt, J)$, the relevant assumption is that the $\bxt$-tensor product is braided or symmetric as appropriate.
\end{rex}

\subsection{Overview of assumptions} In the following sections, we will work with 
a symmetric normal oplax monoidal closed double category, as per \cref{def:oplaxmonoidal,def:normality,def:closeddouble,def:symmetricoplaxmon}.
We list the key assumptions on the (underlying) double category that we use for our 
main theorems in \cref{tab:assumptions} for the convenience of the readers. These
will be assumed progressively, so as to highlight where they are
needed.

\begin{table}[htb]
\caption{List of assumptions.}\label{tab:assumptions}
\fbox{
\begin{minipage}{12cm}
\begin{enumerate}[(I)]
\item $\dcC$ is fibrant.
\item Each $\MndCA$ has finite coproducts.
\item $\dcC_0$ has equalisers.
\item $\dcC$ has $1$-tabulators.  
\item Each $\MndCA \rightarrow \EndCA$ has a left adjoint.
\item Each $\MndCA$ has coequalisers.
\item $\dcC$ has local equalisers.
\item $\dcC$ has stable local reflexive coequalisers.
\end{enumerate} 
\end{minipage}}
\end{table}

In~\cref{sec:applications} we will show that, for a symmetric monoidal category $\ca{V}$ satisfying appropriate assumptions, 
 the double categories $\VMat$ of $\ca{V}$-matrices, and $\VCatSym$ of categorical symmetric sequences, 
 admit a symmetric oplax monoidal closed structure and satisfy the assumptions listed in~\cref{tab:assumptions}. This will enable us to apply our theory and obtain results about the 
 double categories $\VProf$ and $\VSMultProf$ of profunctors and symmetric multiprofunctors, respectively.

\section{The non-commuting tensor product of monads}
\label{sec:non-commuting}

In this section, we begin work towards our first main result, establishing the existence of a commuting monoidal structure on the category $\Mnd(\dc{C})$ of monads in a suitable oplax monoidal double category. This goal will only be achieved in the next section; in the present one, we instead establish the existence of a ``non-commuting'' monoidal structure on $\Mnd(\dc{C})$ which generalises of the so-called ``funny tensor product'' of categories. While this monoidal structure has some independent interest, for us it is largely a stepping-stone towards the commuting tensor product. This is true at a technical level---since the commuting monoidal structure will be constructed using the non-commuting one---but also at an expositional level. Indeed, the proof strategy in both sections is deliberately the same, so that the case of the non-commuting structure may serve as a useful warm-up to the reader for the more involved commuting one.

The main obstacle we face is that of constructing the associativity and unit coherences for the commuting tensor product. Rather than attempt this directly, we instead choose to make $\MndC$ into the underlying category of a symmetric multicategory, and to show that this is \emph{representable}; then, by general arguments, $\MndC$ becomes a symmetric monoidal category under the tensor product representing these multimorphisms. Even this is not entirely straightforward: the problematic aspect is that representability must be verified with parameters. Thus, again, we take a different approach: we verify a weaker form of representability without parameters, which we term \emph{pre-representability}, and then show that our multicategory is \emph{closed}. By an argument familiar from categorical logic and proof theory, this allows us to bootstrap representability with parameters from that without and so, finally, obtain the desired symmetric monoidal structure---which, in addition, will now also be known to be closed.

As stated, the above proof strategy will be deployed twice, once for the non-commuting and once for the commuting case, with the differentiating factor being the choice of multicategory structure imposed on 
monads. For the non-commuting case, it is not particularly hard to proceed in a much more direct way, and indeed, we sketch such an approach in \cref{rmk:directlymonoidal}. However, for the reasons adumbrated above, we prefer to mirror precisely the structure of the proof for the commuting case: as we now begin to do.

\subsection{The multicategory of monads and multimorphisms}
Let $(\dc{C},\bxt,I)$ be a symmetric normal oplax monoidal double category, as in \cref{def:oplaxmonoidal,def:normality,def:symmetricoplaxmon}. In this section, we will make $\MndC$ into the underlying category of a symmetric multicategory $\MultMndC$ which, in the case where $\dcC = \Mat_{\Set}$, will have as its $n$-ary multimorphisms the functors in $n$ variables in the sense described in the introduction.

In order to ease the notational burden, we introduce some
conventions for working with indices.

\begin{notat}\label{not:indexconventions}
  Given a list of objects $A_i, A_{i+1}, \dots, A_{j\mi 1},A_j$ in
  $\dcC_0$, we may write:
  \begin{itemize}
  \item $A_{[i, j]}$ as an abbreviation for
  $A_i \bxt A_{i+1} \bxt \cdots \bxt A_{j \mi 1} \bxt A_j$;
  \item $\tparen{1_A}_{[i,j]}$ as an abbreviation for
  $1_{A_i} \bxt \dots \bxt 1_{A_j}$;
\item $\tparen{\id_A}_{[i,j]}$ as an abbreviation for
  $\id_{A_i} \bxt \dots \bxt \id_{A_j}$;
\item $\tparen{1_{\id_A}}_{[i,j]}$ as an abbreviation for
  $1_{\id_{A_i}} \bxt \dots \bxt 1_{\id_{A_j}}$;
\item $\tparen{\disc{A}}_{[i,j]}$ as an abbreviation for
  $\disc{A_i} \bxt \dots \bxt \disc{A_j}$;
  \end{itemize}
  and so on.
\end{notat}
For the next definition, recall from~\cref{rmk:discrete} that $\disc A$ denotes the discrete monad on an object $A \in \dcC_0$ and that, by \cref{lem:tensoring-with-id-lifts}, if $(B,b)$ is a monad then $\disc A \bxt (B,b) = (A\bxt B, \id_A\bxt b)$ is also a monad (and dually).

\begin{defi} \label{thm:monad-multimorphism}
Let $(A_1,a_1), \ldots, (A_n,a_n), (B,b)$ be monads in $\dcC$. A \emph{monad multimorphism}
$$(f, \vec \phi)  \co (A_1,a_1), \ldots, (A_n,a_n) \to (B,b)$$ is a vertical map $f \co A_1 \boxtimes \cdots \boxtimes A_n \to B$ together
with 2-morphisms
\begin{displaymath}
  \begin{tikzcd}[column sep = 6cm]
    A_{[1,n]} 
    \ar[r, tick,"\tparen{\id_A}_{[1,i \mi 1]}\bxt a_i \bxt \tparen{\id_A}_{[i{+}1,n]}"] 
    \ar[d, "f"']   \ar[dr,  phantom, "\Two \phi_i"] & 
    A_{[1,n]} \ar[d, "f"] \\
    B \ar[r, tick,"b"'] & B
  \end{tikzcd}
\end{displaymath}
for every $1 \leq i \leq n$, that provide the data for monad
maps
\begin{equation}\label{eq:multcomp}
  (f, \phi_i) \colon \disc{A_1} \bxt \cdots \bxt \disc{A_{i-1}} \bxt \,(A_i, a_i)\, \bxt \disc{A_{i+1}} \bxt \cdots \bxt \disc{A_n} \rightarrow (B,b)\rlap{ .}
\end{equation}
In particular, a unary monad multimorphism is the same as a monad
morphism $(A,a) \rightarrow (B,b)$; while a nullary monad
multimorphism is the same as a map $I \rightarrow B$ in $\dcC_0$.
\end{defi}
\begin{rex} \label{thm:running-9}
  In the case of $\VMat$, a binary monad multimorphism $(f, \phi, \psi) \colon (A,a), (B,b) \rightarrow (C,c)$ is the same as a $\ca{V}$-functor of two variables $F \colon \ca A, \ca B \rightarrow \ca C$ (or $\ca V$-sesquifunctor) in the sense of the introduction. Indeed:
  \begin{itemize}[itemsep=0.4\baselineskip]
  \item The function $f \colon A \times B \rightarrow C$ encodes an assignment on objects $(a,b) \mapsto F(a,b)$;
  \item The monad morphism $(f, \phi) \colon (A,a) \bxt \disc{B} \rightarrow (C,c)$ is equally a $\ca V$-functor $\sum_{b \in B} \ca A \rightarrow \ca C$ with action $f$ on objects, so equally a family of $\ca{V}$-functors $(F(\thg, b) \colon \ca A \rightarrow \ca C)_{a \in A}$ satisfying $F(\thg, b)(a) = F(a,b)$;
  \item Likewise, $(f, \psi) \colon \disc{A} \bxt (B,b) \rightarrow (C,c)$ is equally a family of $\ca{V}$-functors $(F(a, \thg) \colon \ca B \rightarrow \ca C)_{b \in B}$ satisfying $F(a, \thg)(b) = F(a,b)$.
  \end{itemize}
  Thus, a binary monad multimorphism comprises $\ca{V}$-functors $(F(\thg, b) \colon \ca A \rightarrow \ca C)_{a \in A}$ and $(F(\thg, b) \colon \ca A \rightarrow \ca C)_{a \in A}$ satisfying $F(a, \thg)(b) = F(\thg, b)(a)$ for all $a,b$: which is precisely a functor of two variables as in the introduction. The $n$-ary case is similar: for example, a ternary monad multimorphism corresponds to a functor of three variables $F \colon \ca{A}, \ca{B}, \ca{C} \rightarrow \ca{D}$ involving families of $\ca{V}$-functors $F(a, b, \thg)$ and $F(\thg, b, c)$ and $F(a, \thg, c)$ which agree on objects in the evident way.
\end{rex}

For the definition of a symmetric multicategory, see for example \cite{Leinster:2004a,yau2016colored}. Below we use the `circle-$i$' or `pseudo-operad' formulation of the definition, which is equivalent to the standard one (see e.g. \cite[Thm.~16.4.1]{yau2016colored}).

\begin{prop}\label{prop:MultMndmulticategory}
There is a symmetric multicategory $\MultMndC$ with monad as
objects and monad multimorphisms as maps.
\end{prop}

\begin{proof}
  For every monad $(A,a)$, the (unary) identity multimorphism is given
  by the identity monad map $(1_A,\id_a)\colon (A,a)\to(A,a)$.
  Moreover, given monad multimorphisms
$$(g,\vec \theta)\colon(B_1,b_1),\ldots,(B_m,b_m)\to(C,c) \qquad \text{and} \qquad 
(f,\vec \phi)\colon(A_1,a_1),\ldots, (A_n,a_n)\to (B_i,b_i),$$ 
the composite $(g, \vec \theta) \circ_i (f, \vec \phi)$ is the monad multimorphism $$(B_1,b_1),\ldots,(B_{i\mi1},b_{i\mi1}),(A_1,a_1),\ldots(A_n,a_n),(B_{i+1},b_{i+1}),\ldots,(B_m,b_m)\to(C,c)$$
defined as follows. The vertical map is formed as
\begin{displaymath}
  B_{[1, i \mi 1]}\bxt A_{[1,n]}\bxt B_{[i+1, m]}\xrightarrow{\tparen{1_B}_{[1, i\mi 1]} \bxt \,f\,\bxt \tparen{1_B}_{[i+1, m]}}
  B_{[1, i \mi 1]}\bxt B_i \bxt B_{[i+1, m]}\xrightarrow{\ \ g\ \ } C
\end{displaymath}
while the 2-morphisms providing the monad map components split up into
two cases. The first case involves one of the $b_k$-monads in the
domain; we illustrate with the case $k = 1$, but there are analogous
diagrams for $k = 2, \dots, i - 1, i+1, \dots, m$:
\begin{displaymath}
\scalebox{0.9}{
\begin{tikzcd}[ampersand replacement=\&,column sep=1in]
  B_{[1,i\mi 1]}\bxt A_{[1,n]}\bxt B_{[i+1,m]}\ar[d,"\tparen{1_B}_{[1,i\mi1]}\bxt \,f\,\bxt\tparen{1_B}_{[i+1,m]}"']\ar[rr, tick,"b_1 \bxt \tparen{\id_B}_{[2,i\mi 1]} \bxt \tparen{\id_A}_{[1,n]} \bxt \tparen{\id_B}_{[i+1,m]}"]
  \ar[drr,phantom,"\Two 1_{b_1} \bxt \tparen{1_{\hid_B}}_{[2, i\mi 1]} \bxt \id_f \bxt \tparen{1_{\hid_B}}_{[i+1, m]}"] \&\& 
  B_{[1,i\mi 1]}\bxt A_{[1,n]}\bxt B_{[i+1,m]}\ar[d,"\tparen{1_B}_{[1,i\mi1]}\bxt \,f\,\bxt\tparen{1_B}_{[i+1,m]}"] \\
  B_{[1,m]}\ar[d,"g"']\ar[rr,tick,"b_1 \bxt \tparen{\id_B}_{[2,m]}"']\ar[drr,phantom,"\Two\theta_1"] \&\&
B_{[1,m]}\ar[d,"g"] \\
C\ar[rr,tick,"c"'] \&\& C\rlap{ .}
\end{tikzcd}}
\end{displaymath}
The second case involves one of the $a_j$-monads in the domain; again
we illustrate with the case $j = 1$, but there are analogous
diagrams for $j = 2, \dots, n$:

\begin{displaymath}
\scalebox{0.9}{
\begin{tikzcd}[ampersand replacement=\&,column sep=1in]
  B_{[1,i\mi 1]}\bxt A_{[1,n]}\bxt B_{[i+1,m]}
  \ar[d,"\tparen{1_B}_{[1,i\mi1]}\bxt \,f\,\bxt\tparen{1_B}_{[i+1,m]}"']
  \ar[rr, tick,"\tparen{\hid_B}_{[i,i\mi 1]} \bxt a_1 \bxt \tparen{\hid_A}_{[2,n]} \bxt \tparen{\hid_B}_{[i+1,m]}"]
  \ar[drr,phantom,"\Two \tparen{1_{\hid_B}}_{[1, i\mi 1]} \bxt \phi_1 \bxt \tparen{1_{\hid_B}}_{[i+1, m]}"] \&\& 
  B_{[1,i\mi 1]}\bxt A_{[1,n]}\bxt B_{[i+1,m]}
  \ar[d,"\tparen{1_B}_{[1,i\mi1]}\bxt \,f\,\bxt\tparen{1_B}_{[i+1,m]}"] \\
  B_{[1,m]}
  \ar[d,"g"']
  \ar[rr,tick,"\tparen{\hid_B}_{[1,i\mi 1]} \bxt b_i \bxt \tparen{\hid_B}_{[i+1,m]}"']
  \ar[drr,phantom,"\Two\theta_i"] \&\&
  B_{[1,m]}\ar[d,"g"] \\
  C\ar[rr,tick,"c"'] \&\& C\rlap{ .}
\end{tikzcd}}
\end{displaymath}
Finally, for any multimap $(f,\vec \phi)  \co (A_1,a_1), \ldots, (A_n,a_n) \to (B,b)$ and any permutation $\sigma \in \mathfrak{S}_n$, the action of $\sigma$ on $(f, \vec \varphi)$ yields the monad multimorphism
$$(A_{\sigma(1)},a_{\sigma(1)}),\ldots,(A_{\sigma(n)},a_{\sigma(n)})\to(B,b)$$ that consists of the vertical map $A_{\sigma(1)}\bxt\ldots\bxt A_{\sigma(n)}\xrightarrow{\beta_\sigma}A_1\bxt\ldots\bxt A_n\xrightarrow{f} B$---where here and below, $\beta_\sigma$ represents appropriate applications of the symmetry of $\dc{C}$ as in \cref{eq:braidingcomponents}---and the monad map components
\begin{displaymath}
\begin{tikzcd}[column sep = 6cm]
\tparen{A_\sigma}_{[1,n]} \ar[r,tick,"\tparen{\id_{A_\sigma}}_{[1, \sigma^{\mi 1}(i)\mi 1]}\bxt a_{i}\bxt\tparen{\id_{A_\sigma}}_{[\sigma^{\mi 1}(i)+1, n]}"]\ar[d,"\beta_\sigma"']
\ar[dr,phantom,"\Two\beta"] & \tparen{A_\sigma}_{[1,n]}
\ar[d,"\beta_\sigma"] \\
A_{[1,n]} \ar[d, "f"'] \ar[r, tick,"\tparen{\id_A}_{[1,i \mi 1]}\bxt a_i \bxt \tparen{\id_A}_{[i{+}1,n]}"] 
\ar[dr,  phantom, "\Two \phi_i"] & A_{[1,n]} \ar[d,"f"] \\
B \ar[r, tick,"b"'] & B\rlap{ .}
\end{tikzcd}
\end{displaymath}

It can now be verified that these data form monad maps, due to the axioms that $\beta_\sigma$ satisfies, and that this assignment constitutes an action of the symmetric group $\mathfrak{S}_n$ on $n$-ary monad multimorphisms. Furthermore, one can verify the associativity, unitality and equivariance axioms (found \eg in \cite[Def.~16.2.1]{yau2016colored}), thus showing that monads and monad multimorphisms form a symmetric multicategory.
\end{proof}

To illustrate the definition of composition in $\MultMndC$, it is perhaps useful to consider a concrete example. Given a ternary map
$(g,\theta_1,\theta_2,\theta_3)\colon
(B_1,b_1),(B_2,b_2),(B_3,b_3)\to(C,c)$ and a binary map
$(f,\phi_1,\phi_2)\colon (A_1,a_1),(A_2,a_2)\to(B_2,b_2)$ as in
\cref{thm:monad-multimorphism}, the composite monad multimorphism $(g, \vec \theta) \circ_2 (f, \vec \phi)$ is
a quaternary map $b_1,a_1,a_2,b_3\to c$, with first component the vertical map $$B_1\bxt
A_1\bxt A_2\bxt B_3\xrightarrow{1\bxt f\bxt 1} B_1\bxt B_2\bxt
B_3\xrightarrow{\ \ g\ \ } C$$ and with monad map components specified by the following four 2-morphisms 
\begin{displaymath}
\begin{tikzcd}[column sep=.9in]
B_1 \!\bxt\! A_1 \!\bxt\! A_2 \!\bxt\! B_3\ar[d,"1_{B_1} \bxt f \bxt 1_{B_3}" description]\ar[r,tick,"b_1 \bxt \id_{A_1} \bxt \id_{A_2} \bxt \id_{B_3}"]\ar[dr,phantom,"\Two 1_{b_1} \bxt \id_{f} \bxt 1_{\id_{B_3}}"] & B_1 \!\bxt\! A_1 \!\bxt\! A_2 \!\bxt\! B_3\ar[d,"1_{B_1} \bxt f \bxt 1_{B_3}" description] \\
B_1 \!\bxt\! B_2 \!\bxt\! B_3\ar[d,"g"']\ar[r,tick,"b_1 \bxt \id_{B_2} \bxt \id_{B_3}"']\ar[dr,phantom,"\Two\theta_1"] & B_1 \!\bxt\! B_2 \!\bxt\! B_3\ar[d,"g"] \\
C\ar[r,tick,"c"'] & C
\end{tikzcd}\quad
\begin{tikzcd}[column sep=.9in]
B_1 \!\bxt\! A_1 \!\bxt\! A_2 \!\bxt\! B_3\ar[d,"1_{B_1} \bxt f \bxt 1_{B_3}" description]\ar[r,tick,"\id_{B_1} \bxt \id_{A_1} \bxt \id_{A_2} \bxt b_3"]\ar[dr,phantom,"\Two 1_{\id_{B_1}} \bxt \id_{f} \bxt 1_{b_3}"] & B_1 \!\bxt\! A_1 \!\bxt\! A_2 \!\bxt\! B_3\ar[d,"1_{B_1} \bxt f \bxt 1_{B_3}" description] \\
B_1 \!\bxt\! B_2 \!\bxt\! B_3\ar[d,"g"']\ar[r,tick,"\id_{B_1} \bxt \id_{B_2} \bxt b_3"']\ar[dr,phantom,"\Two\theta_3"] & B_1 \!\bxt\! B_2 \!\bxt\! B_3\ar[d,"g"] \\
C\ar[r,tick,"c"'] & C
\end{tikzcd}
\end{displaymath}
\begin{displaymath}
\begin{tikzcd}[column sep=.9in]
B_1 \!\bxt\! A_1 \!\bxt\! A_2 \!\bxt\! B_3\ar[d,"1_{B_1} \bxt f \bxt 1_{B_3}" description]\ar[r,tick,"\id_{B_1} \bxt a_1 \bxt \id_{A_2} \bxt\id_{B_3}"]\ar[dr,phantom,"\Two 1_{\id_{B_1}} \bxt \phi_1 \bxt 1_{\id_{B_3}}"] & B_1 \!\bxt\! A_1 \!\bxt\! A_2 \!\bxt\! B_3\ar[d,"1_{B_1} \bxt f \bxt 1_{B_3}" description] \\
B_1 \!\bxt\! B_2 \!\bxt\! B_3\ar[d,"g"']\ar[r,tick,"\id_{B_1} \bxt b_2 \bxt \id_{B_3}"']\ar[dr,phantom,"\Two\theta_2"] & B_1 \!\bxt\! B_2 \!\bxt\! B_3\ar[d,"g"] \\
C\ar[r,tick,"c"'] & C
\end{tikzcd}\quad
\begin{tikzcd}[column sep=.9in]
B_1 \!\bxt\! A_1 \!\bxt\! A_2 \!\bxt\! B_3\ar[d,"1_{B_1} \bxt f \bxt 1_{B_3}" description]\ar[r,tick,"\id_{B_1} \bxt \id_{A_1} \bxt a_2 \bxt\id_{B_3}"]\ar[dr,phantom,"\Two 1_{\id_{B_1}} \bxt \phi_2 \bxt 1_{\id_{B_3}}"] & B_1 \!\bxt\! A_1 \!\bxt\! A_2 \!\bxt\! B_3\ar[d,"1_{B_1} \bxt f \bxt 1_{B_3}" description] \\
B_1 \!\bxt\! B_2 \!\bxt\! B_3\ar[d,"g"']\ar[r,tick,"\id_{B_1} \bxt b_2 \bxt \id_{B_3}"']\ar[dr,phantom,"\Two\theta_2"] & B_1 \!\bxt\! B_2 \!\bxt\! B_3\ar[d,"g"] \\
C\ar[r,tick,"c"'] & C\rlap{ .}
\end{tikzcd}
\end{displaymath}

To make this even more concrete, we continue our:
\begin{rex} \label{thm:running-10}
  Given a $\ca{V}$-functor of three variables $G \colon \ca{B}_1, \ca{B}_2, \ca{B}_3 \rightarrow \ca{C}$ and a $\ca{V}$-functor of two variables $F \colon \ca{A}_1, \ca{A}_2 \rightarrow \ca{B}_2$, the composite $G \circ_2 F$ is the $\ca{V}$-functor of four variables $\ca{B}_1, \ca{A}_1, \ca{A}_2, \ca{B}_3 \rightarrow \ca{C}$ with
  \begin{align*}
    (G \circ F)(\thg, a_1, a_2, b_3) &= G(\thg, F(a_1, a_2), b_3) &
    (G \circ F)(b_1, \thg, a_2, b_3) &= G(b_1, F(\thg, a_2), b_3) \\
    (G \circ F)(b_1, a_1, \thg, b_3) &= G(b_1, F(a_1, \thg), b_3) &
    (G \circ F)(b_1, a_1, a_2, \thg) &= G(b_1, F(a_1, a_2), \thg) \rlap{ .}
  \end{align*}
\end{rex}

\subsection{Pre-representability of $\MultMndC$}\label{sub:prerepMultMnd}

We now start building towards \cref{thm:nctensor-representable}, which
proves that the symmetric multicategory $\MultMndC$ of monads in
$\dcC$ is \emph{representable} in the sense of
\cite{RepresentableMulticats}. 

\begin{defi}
  \label{def:representable-multicategory} A symmetric multicategory
  $\ca{M}$ is \emph{representable} if for every $n$-tuple of objects
  $A_1\dots,A_n$, there is an object $A_1 \otimes \dots \otimes A_n$
  and multimorphism
  $A_1, \dots, A_n \rightarrow A_1 \otimes \dots \otimes A_n$ which is
  \emph{universal}, in the sense that composition with it induces
  bijections of multimorphism-sets
\begin{equation}\label{eq:representability}
  \begin{array}{@{\,\,}r@{\,}l@{\ }l@{\ }c@{\ }l@{\ }l@{\,\,}}
      f \colon & Z_1, \dots, Z_k, & A_1, \ \,\dots,\, \ A_n,&  Z_{k+1}, \dots, Z_n &\longrightarrow& C \\
      \noalign{\vskip 2pt}
      \hline
      \noalign{\vskip 2.5pt} 
      f^\sharp \colon & Z_1, \dots, Z_k, & A_1 \otimes \dots \otimes A_n,&  Z_{k+1}, \dots, Z_n &\longrightarrow& C
    \end{array}
\end{equation}
for all codomains $C$ and parameters $Z_1, \dots, Z_n$. 
\end{defi}
As explained in the introduction to this section, the presence of the
parameters complicates the task of checking universality. What is
easier is to establish the existence of \emph{pre-universal}
morphisms, which exhibit the bijections~\cref{eq:representability}
only when there are no parameters $Z_i$\footnote{Note that universal
  and pre-universal morphisms are referred to as strong universal and
  universal, respectively, in~\cite{RepresentableMulticats}; our nomenclature
  follows, \eg\cite[Def.~3.3.2]{Leinster:2004a}.}. We will do this for
$\MultMndC$, and deduce their full universality in the next section
by showing that $\MultMndC$ is a \emph{closed} multicategory.

In fact, in the presence of a closed structure, it suffices for
representability that we have \emph{binary} and \emph{nullary}
pre-universal morphisms; as such, it will be convenient for us to
adopt the---perhaps slightly misleading---term
\emph{pre-representable} for a multicategory in which just these
pre-universal morphisms exist.

In order to establish the pre-representability of $\MultMndC$, we make the following assumptions.

\begin{hyp} \label{hyp:fibrancy}  The double category $\dcC$ is fibrant, as in \cref{defi:fibrant}.
\end{hyp}

\begin{hyp} \label{hyp:coproducts-in-mnd}
Each category $\MndCA$ has finite coproducts.
\end{hyp}

In this context, the codomain of the pre-universal binary multimorphism $(A, a), (B,b) \rightarrow (A,a) \nctensor (B,b)$, which is the desired \emph{non-commuting tensor product} of $(A,a)$ and $(B,b)$, will be the monad $a \nctensor b$ on the object $A \bxt B$ obtained as the coproduct
\begin{equation}\label{eq:nctensor}
(A \bxt B, a \nctensor b) =(A,a) \bxt \disc{B}\,+\,\disc{A} \bxt (B,b)
\end{equation} 
in $\Mnd_{A\bxt B}(\dc{C})$
\footnote{Recall that colimits in the fibers $\dcC[A,A]$ of $\EndC$ and $\Mnd_A(\dcC)$ of $\MndC$ are not the same, hence writing the monad $a\nctensor b\colon A\bxt B\tickar A\bxt B$ as $a\bxt\id_B+\id_A\bxt b$ in $\Mnd_{A\bxt B}(\dcC)$ should be used with caution.}. Of course, like any coproduct, this can also be expressed as a pushout over the initial object, which, in the case of $\MndCA[A
\bxt B]$, is the discrete monad $\disc{A \bxt B}$. Thus, the above coproduct is equally well the pushout
\begin{equation}\label{eq:pushout}
\begin{tikzcd}
\disc{A \bxt B}\ar[r,"\eta_2"]\ar[d,"\eta_1"] & \disc{A}\bxt (B,b)\ar[d,"\iota_2"] \\
(A,a)\bxt\disc{B}\ar[r,"\iota_1"'] & (A \bxt B,a \nctensor b)\ar[ul,phantom,"\ulcorner", very near start]
\end{tikzcd}
\end{equation}
in $\Mnd_{A\bxt B}(\dc{C})$, which
by~\cref{lem:local-to-global-connected-limits} is also a pushout in
$\MndC$.

\begin{prop} \label{thm:nctensor-prerepresentable} 
The symmetric multicategory $\MultMndC$ is pre-representable.
\end{prop}

\begin{proof}
  First of all, since nullary monad multimorphisms into $(B,b)$ are
  the same as maps $f \colon I \rightarrow B$ in $\dcC_0$, and these
  are in turn the same as monad morphisms
  $\disc{I} \rightarrow (B,b)$, we see that there is a pre-universal
  nullary map given by $1_I \colon \ \longrightarrow \disc{I}$.
  
  As for the binary case, let $(A,a), (B,b) \in \MndC$. We show that there is a pre-universal binary multimorphism
  \begin{equation}
    \label{eq:universal-binary-multimorphism}
    (1_{A \bxt B}, \iota_1, \iota_2) \colon (A,a), (B,b) \rightarrow (A \bxt B, a \nctensor b)
  \end{equation}
where $(A,a) \nctensor (B,b)$ is defined as in \cref{eq:nctensor}, and $\iota_1$ and $\iota_2$ are the coproduct inclusions in the fibre $\MndC_{A\bxt B}$, as in~\eqref{eq:pushout}.
For this, we must show that composition with~\eqref{eq:universal-binary-multimorphism} induces for each monad $(C,c)$ a bijection
  \begin{equation}\label{eq:binarypre}
  \begin{split}
    (A,a),(B,b)&\to (C,c)\qquad\textrm{in }\MultMndC \\
  \noalign{\smallskip} \hline \noalign{\smallskip}
  (A \bxt B, a \nctensor b)&\to (C,c)\qquad\textrm{in }\MndC
  \end{split}
\end{equation}
or in other words, that each binary monad multimorphism
$(f, \phi, \psi) \colon (A,a), (B,b)\to(C,c)$ factors uniquely through
$(1,\iota_1,\iota_2)$. Now, given such an $(f, \phi, \psi)$, we have a
square in $\MndC$
\begin{equation*}
  \begin{tikzcd}
\disc{A \bxt B}\ar[r,"{(1,\eta_2)}"]\ar[d,"{(1,\eta_1)}"'] & \disc{A}\bxt (B,b)\ar[d,"{(f,\psi)}"] \\
(A,a)\bxt\disc{B}\ar[r,"{(f, \phi)}"'] & (C,c)
\end{tikzcd}
\end{equation*}
which evidently commutes at the level of objects and so, by the
adjointness of $\disc{(\thg)}$ to the forgetful functor
$\MndC \rightarrow \dcC_0$, commutes simpliciter.
Since~\eqref{eq:pushout} is a pushout, we induce a monad map
$(A \bxt B, a \nctensor b) \rightarrow (C,c)$ which is unique such that it
precomposes with $(1, \iota_1)$, respectively $(1, \iota_2)$, to yield
$(f, \phi)$, respectively $(f, \psi)$. Clearly, this is the desired
unique factorisation of $(f, \phi, \psi)$ through $(1, \iota_1, \iota_2)$.
\end{proof}
\begin{rex} \label{thm:running-11}
  $\VMat$ is always fibrant; and if $\ca{V}$ has colimits preserved by tensor in each variable, then we will also have finite coproducts in each $\MndCA$. Indeed, given monads $(A, a_1)$ and $(A, a_2)$, their coproduct is found by first forming the free monad on $(A, a_1 \circ a_2)$, and then taking a suitable coequaliser, formed in $\EndCA$ in the first instance, but after the fact lifting to $\MndCA$. In the case of $\ca{V} = \Set$, this can be described as the coequaliser which forces the equalities
  \begin{align*}
      \cdots x_{2k-2} \xrightarrow{f_k} x_{2k-1} \xrightarrow{\mathrm{id}} x_{2k} \xrightarrow{f_{k+1}} x_{2k+1} \xrightarrow{g_{k+1}} x_{2k+2} \cdots  \ \  
      &=  \ \   \cdots x_{2k-2} \xrightarrow{f_{k+1} f_k}  x_{2k+1} \xrightarrow{g_{k+1}} x_{2k+2} \cdots  \\
      \text{and }\cdots x_{2k-2} \xrightarrow{f_k} x_{2k-1} \xrightarrow{g_k} x_{2k} \xrightarrow{\mathrm{id}} x_{2k+1} \xrightarrow{g_{k+1}} x_{2k+2} \cdots  \ \  
    &=  \ \   \cdots x_{2k-2} \xrightarrow{f_k}  x_{2k+1} \xrightarrow{g_{k+1}g_k} x_{2k+2} \cdots \rlap{ .}
  \end{align*}
  The induced non-commuting tensor product $\nctensor$ is, in the case $\ca V = \Set$, the so-called \emph{funny tensor product}; for an explicit description, see, \eg,~\cite[\S 2]{Weber2013Free}.
\end{rex}

There is no obstacle to extending the argument of the previous
proposition to  $n$-ary pre-universal maps; for example,
the ternary case involves a coproduct
$(A,a) \bxt \disc{B} \bxt \disc{C} + \disc{A} \bxt (B,b) \bxt \disc{C}
+ \disc{A} \bxt \disc{B} \bxt (C,c)$ in $\MndCA[A\bxt B\bxt C]$.
However, as noted above, doing so is supererogatory in the presence of
closed structure on our multicategory---as we will now see.

\subsection{Closure and representability of $\MultMndC$}\label{sub:clMultMnd}

Let us recall (\cf\cite{LambekJ:dedscs,Closedmulticats}):
\begin{defi}
  \label{def:closed-multicategory} A symmetric multicategory $\ca{M}$
  is \emph{closed} if, for every $B, C \in \ca{M}$, there is an object
  $[B,C] \in \ca{M}$ and a binary multimorphism $\mathsf{ev} \colon
  [B,C], B \rightarrow C$, composition with which induces bijections
  between multimorphism-sets
  \begin{equation*}
    \begin{array}{@{\,\,}r@{\,}l@{\ }c@{\ }l@{\,\,}}
      f \colon & A_1, \dots, A_n, B &\longrightarrow& C \\
      \noalign{\vskip 2pt}
      \hline
      \noalign{\vskip 2.5pt} 
      \bar f \colon & A_1, \dots, A_n &\longrightarrow& [B,C] \mathrlap{.}
    \end{array}
  \end{equation*}
\end{defi}
The relevance of this for us is the following result, which appears to
be folklore in this generality---though it can be deduced as a special
case of~\cite[Prop.~4.8]{Skewmulticategories}.
\begin{prop}
  \label{prop:closed-preuniversal-universal} If a multicategory $\ca M$
  is pre-representable and closed, then it is representable.
\end{prop}
\begin{proof}
  Since $\ca M$ is pre-representable, it has pre-universal binary
  morphisms $u_{A,B} \colon A,B \rightarrow A \otimes B$ for all
  objects $A,B$. We first show that these are universal. Indeed,
  given a multimorphism as on the first line in
  \begin{equation*}
    \begin{array}{@{\,\,}r@{\,}l@{\ }l@{\ }c@{\ }l@{\ }l@{\,\,}}
      f \colon & Z_1, \dots, Z_k, & A,\  B,&  Z_{k+1}, \dots, Z_n &\longrightarrow& C \\
      \noalign{\vskip 2pt}
      \hline
      \noalign{\vskip 2.5pt} 
      \bar f \colon && A,\ B &&\longrightarrow& [Z_1, [Z_2, \dots [Z_n, C]\cdots]\\
      \noalign{\vskip 2pt}
      \hline
      \noalign{\vskip 2.5pt} 
      {\bar f}^\sharp \colon && A\! \otimes\! B &&\longrightarrow& [Z_1, [Z_2, \dots [Z_n, C]\cdots]\\
      \noalign{\vskip 2pt}
      \hline
      \noalign{\vskip 2.5pt} 
      f^\sharp \colon & Z_1, \dots, Z_k, & A\! \otimes\! B,&  Z_{k+1}, \dots, Z_n &\longrightarrow& C
    \end{array}
  \end{equation*}
  we have a sequence of bijections as displayed. Here, from the first
  to the second line, we use a combination of closedness and symmetry
  isomorphisms; from the second to the third, we use
  pre-representability; and from the third to the fourth, we again use
  symmetry and closedness.

  An identical argument shows that the nullary pre-universal morphism
  $\ \ \rightarrow I$ is universal, so it remains to exhibit universal
  arrows of other arities. This is formally the same as the standard argument
  showing that nullary and binary products suffice for all finite
  products; the key fact needed is that universal arrows are closed
  under composition~\cite[Proposition~8.5(1)]{RepresentableMulticats}.
\end{proof}

Thus, if we can show that the multicategory of monads and monad
multimorphisms is closed, then we can conclude that it is also
representable. This is not the only way to proceed, and is not
necessarily even the simplest way to proceed; however, when it comes
to the commuting tensor product in the following section, this
approach definitively \emph{will} be the simplest one, and so we
match the form of the argument~here. Of course, the closure is
also an interesting fact in its own right. We outline the other approach in
\cref{rmk:directlymonoidal}.

To prove closedness, we will assume that the symmetric normal
oplax monoidal structure of the double category $\dcC$ is closed, as
in \cref{def:closeddouble}, and make the following two further
assumptions on $\dcC$. Note that, even if we were to prove
representability of $\MultMndC$ directly, some additional assumptions
would be needed; for example, it seems hard to avoid assuming that each
functor
$\disc{A} \bxt (\thg) \colon \MndCA[B] \rightarrow \MndCA[A \bxt B]$
preserves finite coproducts. As such, it is not unreasonable that we
should make additional assumptions at this juncture.

\begin{hyp}
\label{hyp:equalisers}
The category $\dcC_0$ has equalisers.
\end{hyp} 

\begin{hyp} \label{hyp:tabulators}
The double category $\dcC$ has $1$-tabulators, as in \cref{def:1tabulators}.
\end{hyp} 

To establish closedness, we will need three preparatory lemmas.

\begin{lem}\label{lemma1}
  The functor
  \begin{equation}
    \label{eq:dc0-action-on-endc}
    \begin{aligned}
      \dcC_0 \times \EndC &\rightarrow \EndC\\
      (A, (B,y)) & \mapsto (A,\id_A) \bxt (B,y)
    \end{aligned}
  \end{equation}
  has a parametrised right adjoint
  $\uEndC \colon \EndC^\mathrm{op} \times \EndC \rightarrow \dcC_0$.
\end{lem}
It is easy to see that the above functor in fact gives a
monoidal action of $(\dcC_0, \bxt)$ on $\EndC$. So by, for
example,~\cite[Section~6]{Janelidze2001A-note}, the parametrised right
adjoint $\uEndC$ provides the hom-objects for an enrichment
of $\EndC$ in $(\dcC_0, \bxt)$---whence our notation.

\begin{proof}
  For every $(B,y)\in\EndC$, we must exhibit a right adjoint namely
  $\uEndC\bigl((B,y),\thg\bigr) \colon \EndC \rightarrow \dcC_0$ for
  the functor
  $\id_{(\thg)} \bxt (B,y) \colon \dcC_0 \rightarrow \EndC$. Equivalently, we must show that for all $(B,y), (C,z) \in \EndC$,
  the functor
  $\EndC(\id_{(\thg)} \bxt y, z) \colon \dcC_0^\mathrm{op} \rightarrow
  \Set$ is representable; the representing object will then be
  $\uEndC(y,z)$.
  
  Note that the values of this
  functor arise as the following equalisers in $\Set$:
  \begin{equation*}
    \begin{tikzcd}
      \EndC(\id_{A} \bxt y, z)
      \ar[r] &
      \dcC_1\bigl((A,\id_{A}) \bxt (B,y), (C,z)\bigr)
      \ar[r,shift left=1,"{\mathfrak{s}}"]
      \ar[r,shift right=1,"{\mathfrak{t}}"']
      &
      \dcC_0(A \bxt B, C)\rlap{ .}
    \end{tikzcd}
  \end{equation*}
  Thus, if we can prove representability of the functors
  $\dcC_1\bigl(\id_{(\thg)} \bxt (B,y), (C,z)\bigr)$ and
  $\dcC_0(\thg \bxt B, C)$, then $\EndC(\id_{(\thg)} \bxt y, z)$ will
  be a pointwise equaliser of representable functors
  $\dcC_0 \rightarrow \Set$ and so also representable, given that $\dcC_0$ by assumption has equalisers.

  But $\dcC_0(\thg \bxt B, C)$ is representable since $(\dcC_0, \bxt)$ is closed, while $\dcC_1\bigl(\id_{(\thg)} \bxt (B,y), (C,z)\bigr)$ is representable since
  \begin{equation*}
    \dcC_1(\id_{(\thg)} \bxt y, z) \cong
    \dcC_1(\id_{(\thg)}, \boxhom{y, z}) \cong \dcC_0(\thg, T\boxhom{y,z})
  \end{equation*}
  using closedness of $(\dcC_1, \bxt)$ and 1-tabulators in $\dcC$.
  Working this through, we find that the values of the parametrised
  right adjoint to~\cref{eq:dc0-action-on-endc} are given by the
  following equalisers in $\dcC_0$:
    \begin{equation*}
      \begin{tikzcd}
      \uEndC(y, z)
      \ar[r] &
      T\boxhom{y,z}
      \ar[r,shift left=1,"{\pi_\mathfrak{s}}"]
      \ar[r,shift right=1,"{\pi_\mathfrak{t}}"']
      &
      \boxhom{B,C}\rlap{ .}
      \end{tikzcd} \qedhere
    \end{equation*}
\end{proof}
\begin{rex} \label{thm:running-12}
  $\VMat$ always satisfies \cref{hyp:equalisers,hyp:tabulators}, and
  is normal oplax monoidal closed so long as $\ca{V}$ is monoidal
  closed; we assume this henceforth. In this case,
  $\uEndC\bigl((B,y), (C,z)\bigr)$ is simply the hom-set
  $\EndC\bigl((B,y), (C,z)\bigr)$, \ie, the set of $\ca{V}$-graph
  morphisms from $(B,y)$ to $(C,z)$.
\end{rex}

\begin{lem}
  \label{lem:mnd-hom}
 The lifting  of the functor~\cref{eq:dc0-action-on-endc}
  \begin{align*}
    \dcC_0 \times \MndC &\rightarrow \MndC\\
    (A, (B,b)) & \mapsto \disc{A} \bxt (B,b)
  \end{align*}
  obtained via \cref{lem:tensoring-with-id-lifts}, has a
  parametrised right adjoint
  $\uMndC \colon \MndC^\mathrm{op} \times \MndC
  \rightarrow \dcC_0$.
\end{lem}
Again, the values of $\uMndC$ provide the
hom-objects for an enrichment of $\MndC$ in $\dcC_0$.
\begin{proof}
  Note first that, for any monads $(B,b)$ and $(C,c)$, the set
  $\MndC(b, c)$ is the limit of the diagram:
  \begin{equation}
    \label{eq:mnd-hom-diagram}
    \begin{tikzcd}[column sep=2.5em]
      \EndC\bigl((B,b), (C,c)\bigr)
      \ar[r,shift left=1,"\ell_1"]
      \ar[r,shift right=1,"r_1"']
      \ar[d,shift left=1,"{r_2}"]
      \ar[d,shift right=1,"{\ell_2}"']
      &
      \EndC\bigl((B,b \circ b), (C,c)\bigr) \\
      \EndC\bigl((B,\id_B), (C,c)\bigr) \mathrlap{,}
    \end{tikzcd}
  \end{equation}
  where $\ell_1$ and $r_1$, respectively $\ell_2$ and $r_2$, send a map $(f,\varphi)$ to
  the left and right sides of the first, respectively the second, equality
  in~\cref{eq:monad-map-axioms}. Moreover, this limit can be
  constructed from equalisers alone: we let $e$ be an equaliser of
  $\ell_1$ and $r_1$, and then form the equaliser of $\ell_2 e$ and
  $r_2 e$.

  Given $(B,b)$ and $(C,c) \in \MndC$, we now show that
  $\MndC\bigl(\disc{\thg} \bxt (B,b), (C,c)\bigr) \colon \dcC_0^\mathrm{op}
  \rightarrow \Set$ is representable. By the above, replacing $(B,b)$ by $(A\bxt B,\id_A\bxt b)$, the value of this
  functor at $A$ is given by the limit:
  \begin{equation*}
    \begin{tikzcd}[column sep=2.5em]
      \EndC\bigl((A\bxt B,\id_A\bxt b), (C,c)\bigr) \ar[r,shift left=1,"\ell_1"] \ar[r,shift
      right=1,"r_1"'] \ar[d,shift left=1,"{r_2}"] \ar[d,shift
      right=1,"{\ell_2}"'] &
      \EndC\bigl((A \bxt B, (\id_A \bxt b) \circ (\id_A \bxt b)), (C,c)\bigr) \\
      \EndC\bigl((A \bxt B,\id_{A \bxt B}), (C,c)\bigr) \mathrlap{,}
    \end{tikzcd}
  \end{equation*}
  and so  -- arguing as before --  it suffices to show
  that each vertex of this diagram is representable as a functor of
  its $A$-variable. For the top-left vertex, this follows immediately
  from the preceding lemma; for the other two vertices, it follows
  from the same lemma on noting that, by normality (\cref{def:normality}) of the oplax
  monoidal structure $\bxt$, we have natural isomorphisms
  $\id_{A \bxt B} \cong \id_{A} \bxt \id_{B}$ and
  $(\id_{A} \bxt b) \circ (\id_{A} \bxt b) \cong \id_{A} \bxt (b \circ
  b)$.
  Working this through, we find that $\uMndC(b,c)$ is the
  limit of a diagram
  \begin{equation*}
    \begin{tikzcd}[column sep=2.5em]
      \uEndC\bigl((B,b), (C,c)\bigr)
      \ar[r,shift left=1,"\underline{\ell_1}"]
      \ar[r,shift right=1,"\underline{r_1}"']
      \ar[d,shift left=1,"\underline{r_2}"]
      \ar[d,shift right=1,"\underline{\ell_2}"']
      &
      \uEndC\bigl((B,b \circ b), (C,c)\bigr) \\
      \uEndC\bigl((B,\id_{B}), (C,c)\bigr)\rlap{ ,}
    \end{tikzcd}
  \end{equation*}
  i.e., the ``enhancement'' of~\cref{eq:mnd-hom-diagram} from a
  diagram of sets to a diagram in $\dcC_0$.
\end{proof}

\begin{rex} \label{thm:running-13}
  In $\VMat$, the object
  $\uMndC\bigl((B,b), (C,c)\bigr)$ is simply the set
  $\MndC\bigl((B,b), (C,c)\bigr)$, \ie, the set of
  $\ca{V}$-functors $\ca{B} \rightarrow \ca{C}$ between the associated
  $\ca{V}$-categories of $(B,b)$ and $(C,c)$.
\end{rex}
\begin{lem}
  \label{cartesian-lifts-multimnd}
Let $(g, \gamma) \colon (B, b) \rightarrow (C,c)$  be a map in $\MndC$.
Assume that it is  is cartesian with respect to the Grothendieck fibration
  $\MndC \rightarrow \dcC_0$. The following ``enhanced
  cartesianness'' property with respect to the multicategory
  $\MultMndC$ holds: any monad multimorphism $(gh, \vec{\phi}) \colon (A_1, a_1), \dots, (A_n,
  a_n) \rightarrow (C,c)$ factorises uniquely through $(g,\gamma)$. More compactly, for any monads $(A_1, a_1), \dots, (A_n, a_n)$, the
  square
  \begin{equation*}
    \begin{tikzcd}[column sep=3.5em]
      \MultMndC\bigl((A_1, a_1), \dots, (A_n, a_n);\, (B,b)\bigr)
      \ar[r,"{(g, \gamma)} \circ_1 (\thg)"] \ar[d]&
      \MultMndC\bigl((A_1, a_1), \dots, (A_n, a_n);\, (C,c)\bigr) \ar[d]
      \\
      \dcC_0(A_1 \bxt \dots \bxt A_n, B) \ar[r,"g \circ (\thg)"] &
      \dcC_0(A_1 \bxt \dots \bxt A_n, C)
    \end{tikzcd}
  \end{equation*}
  is a pullback, where the vertical arrows project onto the object-components.
\end{lem}

\begin{proof}
We must show there is a unique monad multimorphism $(h, \vec{\phi}') \colon (A_1, a_1), \dots, (A_n, a_n) \rightarrow (B,b)$ which factors
$(gh, \vec{\phi})$ through $(g, \gamma)$. But
  the original monad multimorphism consists of monad maps 
  \begin{equation}
    \label{eq:cartesian-lifts-multimnd}
    (gh, \phi_i) \colon \disc{A_1} \bxt \cdots \bxt (A_i, a_i) \bxt \cdots \bxt \disc{A_n} \rightarrow (C,c)\rlap{ ,}
  \end{equation}
  so by the ordinary cartesianness of $(g, \gamma)$, there is a unique
  $(h, \phi'_i) \colon \disc{A_1} \bxt \cdots \bxt (A_i, a_i) \bxt
  \cdots \bxt \disc{A_n} \rightarrow (B,b)$ in $\MndC$ which
  factors~\cref{eq:cartesian-lifts-multimnd} through $(g, \gamma)$. It
  is clear that the $(h, \vec{\phi}')$ so defined is the
  unique monad multimorphism we seek.
\end{proof}

We now show that
the symmetric multicategory of monads and monad multimorphisms is closed. 

\begin{prop}
  \label{prop:multimorphisms-closed}
  The symmetric multicategory $\MultMndC$ is closed. More precisely,
  for any monads $(B,b), (C,c)$, there exists a monad
  $\nchom{(B,b), (C,c)}$
  and a binary
  multimorphism
  \begin{equation}
    \label{eq:counit-multimorphisms-closed}
    (e, \varepsilon_1, \varepsilon_2) \colon \nchom{(B,b), (C,c)}, (B, b) \rightarrow (C, c)
  \end{equation}
  composition with which induces bijections between multimorphism-sets
  \begin{equation}
    \label{eq:closed-adjunction-multimorphism}
    \begin{array}{r@{\,}l@{\ }c@{\ }l}
      (f, \phi_1, \dots, \phi_n, \psi) \colon& (A_1, a_1), \dots, (A_n, a_n), (B,b) &\longrightarrow& (C,c) \\
      \noalign{\vskip 3pt}
      \hline
      \noalign{\vskip 3pt} 
      (\bar f, \bar{\phi}_1, \dots, \bar{\phi}_n) \colon& (A_1, a_1), \dots, (A_n, a_n) &\longrightarrow& \nchom{(B,b),(C,c)}
    \end{array}
  \end{equation}
\end{prop}
Notice that due to symmetry, we do not differentiate between left and right closedness.

\begin{proof}
  Consider first the case where $(B,b)$ is the identity monad
  $\disc{B}$. In this case, it is easy to see that the final component \cref{eq:multcomp}
  of a multimorphism~\cref{eq:closed-adjunction-multimorphism}
  is forced to be the monad map
  \begin{equation}
    \label{eq:unique-final-component}
    (f,\eta) \colon \tparen{\disc{A}}_{[1,n]} \bxt \disc{B} \xrightarrow{\cong} \disc{A_1 \bxt \dots \bxt A_n} \xrightarrow{\disc{f}} \disc{C} \xrightarrow{(1_C, \eta)} (C,c)\rlap{ ,}
  \end{equation}
  Thus, giving~\cref{eq:closed-adjunction-multimorphism} is the same
  as giving a map
  $f \colon A_{[1,n]} \bxt B \rightarrow C$ in $\dcC_0$
  and monad maps
  \begin{equation}
    \label{eq:remaining-components}
    (f,\phi_i) \colon \tparen{\disc{A}}_{[1,i\mi1]} \bxt (A_i, a_i) \bxt \tparen{\disc{A}}_{[i+1,n]} \bxt \disc B \rightarrow (C,c)\rlap{ .}
  \end{equation}

  Now, since $\dcC$ is closed (\cref{def:closeddouble}), we may transpose $f$ to a map
  $\bar f \colon A_{[1,n]} \rightarrow \boxhom{B, C}$
  under the adjunction
  $(\thg) \bxt B \dashv \boxhom{B, \thg} \colon \dcC_0 \to \dcC_0$.
  But since by~\cref{lem:homming-out-of-id-lifts} this adjunction
  lifts to one
  $(\thg) \bxt \disc{B} \dashv \boxhom{\disc{B}, \thg} \colon \MndC
  \rightarrow \MndC$, we can also transpose
  each~\cref{eq:remaining-components} to a
  monad map
  \begin{equation*}
    (\bar f,{\bar \phi}_i) \colon {\disc{A}}_{[1,i\mi1]} \bxt (A_i, a_i) \bxt {\disc{A}}_{[i+1,n]} \rightarrow \boxhom{\disc{B},
      (C,c)}\rlap{ ,}
  \end{equation*}
  and so obtain a monad multimorphism
  $(\bar f, \bar{\phi}_1, \dots, \bar{\phi}_n) \colon (A_1, a_1),
  \dots, (A_n, a_n) \rightarrow \boxhom{\disc{B}, (C,c)}$. Thus,  when
  $(B,b) = \disc{B}$, the desired internal hom is the monad
  $\boxhom{\disc{B}, (C,c)}=(\boxhom{B,C},\boxhom{\id_B,c})$ 
  equipped with the binary
  multimorphism~\cref{eq:counit-multimorphisms-closed} whose:
  \begin{itemize}[itemsep=0.25\baselineskip]
  \item \textbf{Object-component} $e \colon \boxhom{B,C} \bxt B \rightarrow C$
    is the  counit at $C$ of the adjunction $(\thg) \bxt B \dashv \boxhom{B,
      \thg}$ in $\dcC_0$;
  \item \textbf{First morphism-component} $(e,\varepsilon_1) \colon \boxhom{\disc{B},
      (C,c)} \bxt \disc{B} \rightarrow (C,c)$ is the counit at $(C,c)$ of the
    lifted adjunction
    $(\thg) \bxt \disc{B} \dashv \boxhom{\disc{B}, \thg}$ in $\MndC$;
  \item \textbf{Second morphism-component} $(e, \varepsilon_2) \colon \disc{\boxhom{B,C}}
    \bxt \disc{B} \rightarrow (C,c)$ is determined uniquely as in~\cref{eq:unique-final-component}.
  \end{itemize}

  We now let $(B,b)$ be a general monad and construct the internal hom
  $\nchom{(B,b),(C,c)}$. Its underlying object $\nchom{b,c}_0$ in
  $\dcC_0$ will be the the object $\uMndC(b,c) \in \dcC_0$
  of~\cref{lem:mnd-hom}. As there, this object represents the functor
  $\MndC\bigl(\disc{(\thg)} \bxt (B,b), (C,c)\bigr) \colon \dcC_0^\op
  \rightarrow \Set$. On the other hand, the functor
  $\dcC_0(\thg \bxt B, C) \colon \dcC_0^\op \rightarrow \Set$ is
  represented by $\boxhom{B,C} \in \dcC_0$ since the latter is the internal hom of $\dcC$. Since, projecting to
  object-components determines a natural transformation
  \begin{equation*}
    \MndC\bigl(\disc{(\thg)} \bxt (B,b), (C,c)\bigr) \Rightarrow \dcC_0(\thg \bxt B, C)\rlap{ ,}
  \end{equation*}
  there is by the Yoneda lemma a map
  \begin{equation}
    \label{eq:representably-induced}
    i \colon \nchom{b,c}_0 \rightarrow \boxhom{B,C}
  \end{equation}
  in $\dcC_0$ with the property that factorisations of a map
  $\bar{f} \colon A \rightarrow \boxhom{B,C}$ of $\dcC_0$ through $i$ are in
  bijection with extensions of the exponential transpose
  $f \colon A \bxt B \rightarrow C$ of $\bar{f}$ to a monad morphism
  \begin{equation*}
    (f, \psi) \colon \disc{A} \bxt (B,b) \rightarrow (C,c)\rlap{ .}
  \end{equation*}
  Since $\dcC$ is fibrant, the forgetful functor
  $\MndC \rightarrow \dcC_0$ is a fibration; thus, we may reindex the
  monad $\boxhom{ \disc{B}, (C,c)} $ on $\boxhom{B,C}$
  along~\cref{eq:representably-induced} to obtain a monad
  $\nchom{b,c}$ on $\nchom{b,c}_0$ and a cartesian
  monad morphism
  \begin{equation}
    \label{eq:cartesian-monad-reindexing}
    (i, \tilde \imath) \colon (\nchom{b,c}_0, \nchom{b,c}) \rightarrow \boxhom{\disc{B}, (C,c)}\rlap{ .}
  \end{equation}

  We claim that $(\nchom{b,c}_0, \nchom{b,c})$ provides the
  desired internal hom. For indeed, giving a monad
  multimorphism~\cref{eq:closed-adjunction-multimorphism} is
  equivalent to giving the following two things:
  \begin{enumerate}[(i),itemsep=0.25\baselineskip]
  \item A monad
    multimorphism of the following form, where $\tilde f$ is uniquely determined by $f$ as
    in~\cref{eq:unique-final-component}:
    \begin{equation*}
      (f, \phi_1, \dots, \phi_n, \tilde f) \colon (A_1, a_1), \dots, (A_n, a_n), \disc{B} \rightarrow (C,c)\rlap{ ;}
    \end{equation*}
  \item A monad morphism of the following form, where \cref{not:indexconventions} is used:
    \begin{equation*}
      (f, \psi) \colon \tparen{\disc{A}}_{[1,n]} \bxt (B,b) \rightarrow (C,c)\rlap{ .}
    \end{equation*}
  \end{enumerate}
  Now by the case proved earlier, giving the monad multimorphism in
  (i) is equivalent to giving a multimorphism
  $(\bar f, \bar{\phi}_1, \dots, \bar{\phi}_n) \colon (A_1, a_1),
  \dots, (A_n, a_n) \rightarrow \boxhom{\disc{B}, (C,c)}$. On the
  other hand, it follows from our characterisation of $i$ that giving
  the monad morphism in (ii) is equivalent to giving a factorisation
  of $\bar f \colon A_{[1,n]} \rightarrow \boxhom{B,C}$
  through $i$ in $\dcC_0$. But by~\cref{cartesian-lifts-multimnd} applied to the
  cartesian monad map~\cref{eq:cartesian-monad-reindexing}, giving
  this factorisation of $\bar f$ through $i$ is equivalent to giving a
  factorisation of $(\bar f, \bar{\phi}_1, \dots, \bar{\phi}_n)$
  through $(i, \tilde \imath)$ in $\MultMndC$. 
  
  In this way, we have shown that there is a bijection between monad multimorphisms~\cref{eq:closed-adjunction-multimorphism} and ones
  $(A_1, a_1), \dots, (A_n, a_n) \rightarrow (\nchom{b,c}_0,
  \nchom{b,c})$, so exhibiting $(\nchom{b,c}_0, \nchom{b,c})$ as the desired internal hom.
 The evaluation map~\cref{eq:counit-multimorphisms-closed} is obtained by
  precomposing the evaluation map
  $\boxhom{\disc{B}, (C,c)}, \disc{B} \rightarrow (C,c)$ for
  $\disc{B}$ described above, by the cartesian monad map~\cref{eq:cartesian-monad-reindexing}.
\end{proof}
\begin{rex}\label{thm:running-14}
  In $\VMat$, the hom
  $\nchom{\ca B, \ca C} = (\nchom{b,c}_0, \nchom{b,c})$ between monads
  $\ca B = (B,b)$ and $\ca C = (C,c)$ is the $\ca{V}$-category given
  as follows. Its
  object-set is $\uMndC\bigl((B,b), (C,c)\bigr)$, which
  by~\cref{thm:running-13} is the set of $\ca{V}$-functors
  $\ca{B} \rightarrow \ca{C}$. As for its hom-objects, we have by
  cartesianness of~\eqref{eq:representably-induced} and the
  description of $\boxhom{\disc{B}, (C,c)}$ from \cref{ex:running16}
  that
  \begin{equation*}
    \nchom{\ca B, \ca C}(F, G) = \ca{C}^B\bigl((Fx)_{x \in B}, (Gx)_{x \in B}\bigr) = \textstyle\prod_{x \in B} \ca{C}(Fx, Gx)\rlap{ .}
  \end{equation*}
  In particular, when $\ca{V} = \Set$, this $\nchom{\ca{B}, \ca{C}}$ is the category of
  functors $\ca{B} \rightarrow \ca{C}$ and ``unnatural
  transformations'' $\alpha \colon F \Rightarrow_u G$, \ie, families
  of morphisms $\alpha_x \colon Fx \rightarrow Gx$ satisfying no
  further conditions.
\end{rex}
From this result, \cref{thm:nctensor-prerepresentable} and~\cref{prop:closed-preuniversal-universal}, we immediately obtain, under our running hypotheses:

\begin{cor}\label{thm:nctensor-representable}
The closed symmetric multicategory $\MultMndC$ of monads and monad multimorphisms is representable.
\end{cor}

Finally, the next theorem accomplishes this section's goal, namely that 
the non-commuting tensor product of monads determines a 
a symmetric monoidal closed structure on the category $\Mnd(\dc{C})$.
We state explicitly all our assumptions for reference.

\begin{thm} \label{thm:nctensor-smc} Let $(\dcC, \bxt, I)$ be a
  symmetric normal oplax monoidal closed double category. Under
  \cref{hyp:fibrancy,hyp:equalisers,hyp:tabulators,hyp:coproducts-in-mnd},
  the category $\MndC$ of monads and monad morphisms in $\dcC$ admits
  the structure of a symmetric monoidal closed category, with tensor
  product $\nctensor$ given by the non-commuting tensor product of monads,
  and unit object given by the discrete monad $\disc{I}$ on the unit
  object of $\dcC$.
\end{thm} 

\begin{proof} 
The existence of the symmetric monoidal structure follows from \cref{thm:nctensor-representable} via 
the symmetric version of \cite[Def.~9.6]{RepresentableMulticats}; the closedness then follows from \cite[Prop.~4.3]{Closedmulticats}.
\end{proof} 

\begin{rmk}\label{rmk:directlymonoidal}
  As already mentioned, there is a more direct way of constructing the non-commuting monoidal structure, which we have avoided since it will be less useful for the work to come. However, it may be helpful to see how it goes. We start from the definition of $\nctensor$ as in~\eqref{eq:nctensor}, and then construct associativity and unitality isomorphisms by hand, using an assumption mentioned above: that each functor $\disc{A} \bxt (\thg) \colon \MndCA[B] \rightarrow \MndCA[A \bxt B]$ preserves coproducts of monads. This will certainly be the case if $\dcC$ is closed by virtue of~\cref{lem:homming-out-of-id-lifts}(ii).
In this situation the associativity constraint
\begin{displaymath}
\begin{tikzcd}
(A\bxt B)\bxt C\ar[rr,tick,"(a\nctensor b)\nctensor c"]\ar[d]\ar[drr,phantom,"\Two"] && (A\bxt B)\bxt C\ar[d] \\
A\bxt(B\bxt C)\ar[rr,tick,"a\nctensor(b\nctensor c)"'] && A\bxt(B\bxt C)
\end{tikzcd}
\end{displaymath}
has as vertical maps the associativity structure isomorphisms of $(\dc{C}_0,\bxt)$, and $2$-morphism data given by the composite of invertible $2$-morphisms
\begin{align*}
(a\nctensor b)\nctensor c&=(a\nctensor b)\bxt\disc{C}+\disc{A\bxt B}\bxt c=\left(a\bxt\disc{B}+\disc{A}\bxt b\right)\bxt\disc{C}+\disc{A\bxt B}\bxt c\\
&\cong a\bxt\disc B\bxt \disc C+(\disc A\bxt b)\bxt \disc C+(\disc A\bxt\disc B)\bxt c \\
&\cong a\bxt\disc B\bxt \disc C+\disc A\bxt (b\bxt \disc C)+\disc A\bxt(\disc B\bxt c) \\
&\cong a\bxt\disc {B\bxt C}+\disc A\bxt(b\bxt\disc C+\disc B\bxt c) \\
&= a\nctensor(b\nctensor c) \mathrlap{,}
\end{align*}
where \cref{lemma1} and normality of $\dc{C}$ are used. The construction of the unit isomorphisms is similar.
\end{rmk}

\section{The commuting tensor product of monads}
\label{sec:commuting}

Our objective in this section is to refine the non-commuting tensor
product of monads in $\dcC$ of the preceding section to a
\emph{commuting tensor product}. Once again, this tensor product will
underlie a symmetric monoidal closed structure on $\MndC$, and once
again, we will construct it by exhibiting a symmetric closed
multicategory, now of monads and \emph{commuting monad
  multimorphisms}.

\subsection{The multicategory of monads and commuting multimorphisms}
As the nomenclature suggests, commutativity will be a property of a monad
multimorphism (\cref{thm:monad-multimorphism}), which is most easily described in the binary case; we
thus begin with this. Much as in~\cite{GarnerLopezFranco}, a key role
will be played by the
following maps derived from the oplax monoidal structure of $(\dcC,\bxt)$: 
\begin{defi}\label{def:sigma-tau}
  Given horizontal cells $a_1 \colon A_1 \tor A_1'$ and
  $a_2 \colon A_2 \tor A_2$ in $\dcC$, we define the globular
  $2$-morphisms $\sigma_{a_1, a_2}$ and $\tau_{a_1, a_2}$ as to the
  left and right in:
  \begin{equation*}
    \begin{tikzcd}[column sep=1.5in]
 A_1 \bxt A_2 \ar[r,tick,"a_1 \bxt a_2"]\ar[d,equal]\ar[dr,phantom,"\Two \cong"] & A'_1 \bxt A'_2 \ar[d,equal]\\
 A_1 \bxt A_2 \ar[r,tick,"(a_1 \circ \hid_{A_1}) \bxt (\hid_{A_2} \circ a_2)"]\ar[d,equal]\ar[dr,phantom,"\Two \xi"] & A'_1 \bxt A'_2\ar[d,equal] \\
 A_1 \bxt A_2 \ar[r,tick,"(a_1 \bxt \hid_{A_2}) \circ (\hid_{A_1} \bxt a_2)"] & A'_1 \bxt A'_2
\end{tikzcd} \quad \text{and} \quad
    \begin{tikzcd}[column sep=1.5in]
 A_1 \bxt A_2 \ar[r,tick,"a_1 \bxt a_2"]\ar[d,equal]\ar[dr,phantom,"\Two \cong"] & A'_1 \bxt A'_2 \ar[d,equal] \\
 A_1 \bxt A_2 \ar[r,tick,"(\hid_{A_1} \circ a_1) \bxt (a_2 \circ \hid_{A_2})"]\ar[d,equal]\ar[dr,phantom,"\Two \xi"] & A'_1 \bxt A'_2\ar[d,equal] \\
 A_1 \bxt A_2 \ar[r,tick,"(\hid_{A_1} \bxt a_2) \circ (a_1 \bxt \hid_{A_2})"] & A'_1 \bxt A'_2\rlap{ .}
\end{tikzcd}
  \end{equation*}
\end{defi}

\begin{defi} \label{defi:commuting-binary-monad-multimorphism} We say
  that a binary monad multimorphism
  $(f, \phi_1, \phi_2) \colon (A_1, a_1), (A_2, a_2) \rightarrow
  (B,b)$ is \emph{commuting} if the following diagram commutes in
  $\EndC$:
  \begin{equation}
    \label{eq:binary-commutativity-hexagon}
    \begin{tikzcd}[column sep = 1em]
      &[-9em] (A_1 \bxt A_2, (a_1 \bxt \hid_{A_2}) \hcomp (\hid_{A_1} \bxt a_2)) \ar[r, "{(f, \phi_1 \hcomp \phi_2)}"] &[2em] (B, b \circ b)  \ar[dr, "{(1, \mu)}"] &[-0.5em] \\
      (A_1 \bxt A_2, a_1 \bxt a_2) \ar[ur, "{(1,\sigma)}"] \ar[dr, "{(1,\tau)}"'] & & & (B, b) \mathrlap{ .} \\
      & (A_1 \bxt A_2, (\hid_{A_1} \bxt a_2) \hcomp (a_1 \bxt
      \hid_{A_2})) \ar[r, "{(f,\phi_2 \hcomp \phi_1)}"'] & (B, b \circ b) \ar[ur, "{(1,
        \mu)}"'] &
    \end{tikzcd}
  \end{equation}
\end{defi}
\begin{rex}\label{thm:running-27}
  For $\VMat$, under our standing assumption that $\ca{V}$ is symmetric
  monoidal closed with colimits, we have that:
  \begin{align*}
    (a_1 \bxt a_2)[(y_1, y_2); (x_1, x_2)] &= a_1[y_1; x_1] \otimes a_2[y_2; x_2] \\
    ((a_1 \bxt \hid_{A_2}) \hcomp (\hid_{A_1} \bxt a_2))[(y_1, y_2); (x_1, x_2)] &= \textstyle\sum_{(z_1, z_2)} a_1[y_1; z_1] \otimes \delta_{y_2z_2} \otimes \delta_{z_1x_1} \otimes a_2[z_2; x_2] \\
    ((\hid_{A_1} \bxt a_2) \hcomp (a_1 \bxt \hid_{A_2}))[(y_1, y_2); (x_1, x_2)] &= \textstyle\sum_{(z_1, z_2)} \delta_{y_1z_1} \otimes a_2[y_2; z_2] \otimes a_2[z_1; x_1] \otimes \delta_{z_2x_2}\rlap{ .}
  \end{align*}
  On summing over the $\delta$'s, the second and third of these
  are isomorphic to $a_1[y_1; x_1] \otimes a_2[y_2; x_2]$ and
  $a_2[y_2; x_2] \otimes a_1[y_1; x_1]$ respectively, and modulo these
  isomorphisms, the maps $\sigma$ and $\tau$ are just the identity and
  the symmetry $\beta$ of $\ca{V}$, respectively. Thus, we see that a
  $\ca{V}$-functor of two variables
  $F \colon \ca{A}_1, \ca{A}_2 \rightarrow \ca{B}$ is commuting
  precisely when each diagram of the form
  \begin{equation}
    \label{eq:v-bifunctor}
    \begin{tikzcd}[column sep = 1em]
      &[-8em] {\ca{A}_1}(y_1, x_1) \otimes {\ca{A}_2}(y_2, x_2) \ar[r, "{F(\thg, y_2) \otimes F(x_1, \thg)}"] &[5em] {\ca{B}}(F(y_1, y_2), F(x_1, y_2)) \otimes {\ca{B}}(F(x_1, y_2), F(x_1, x_2))  \ar[dr, "{\mu}"] &[-14em] \\
      {\ca{A}_1}(y_1, x_1) \otimes {\ca{A}_2}(y_2, x_2) \ar[ur, "{1}"] \ar[dr, "{\beta}"'] & & & {\ca{B}}(F(y_1, y_2)) \otimes {\ca{B}}(F(x_1, x_2))\mathrlap{ .} \\
      & {\ca{A}_2}(y_2, x_2) \otimes {\ca{A}_1}(y_1, x_1) \ar[r, "{F(y_1, \thg) \otimes F(\thg, x_2)}"'] & {\ca{B}}(F(y_1, y_2), F(y_1, x_2)) \otimes {\ca{B}}(F(y_1, x_2), F(x_1, x_2)) \ar[ur, "{
        \mu}"'] &
    \end{tikzcd}
  \end{equation}
  commutes in $\ca{V}$. As shown in~\cite[\S III.4]{Kelly1971Coherence}, this is precisely what is required for
  $F$ to extend to a bifunctor
  $\ca{A}_1 \otimes \ca{A}_2 \rightarrow \ca{B}$, where $\otimes$ 
  is the usual tensor product of $\ca{V}$-categories over a symmetric
  monoidal base $\ca{V}$.

  Note that in this case, the notion of commutativity reconstructs
  something with a simpler and more direct definition. The reason for
  this, as noted in~\cref{thm:running-6}, is that, when $\ca{V}$ is
  symmetric monoidal, $\VMat$ provides a slightly degenerate instance
  of our framework. To lift this degeneracy, we may consider a base
  $\ca{V}$ which is not symmetric monoidal, but only normal duoidal
  with respect to tensor products $\circ$ and $\bxt$. Then $\MultMndC$
  is the multicategory of $(\ca{V}, \circ)$-categories and
  $(\ca{V}, \circ)$-functors of several variables---but the notion of
  commutativity involves also the other tensor product $\bxt$. Indeed,
  the maps $\sigma$ and $\tau$ will then (to within
  codomain-isomorphism) have components of the form
  \begin{equation*}
    a_1[y_1; x_1] \bxt a_2[y_2; x_2] \xrightarrow{\sigma}
    a_1[y_1; x_1] \circ a_2[y_2; x_2] \quad \text{and} \quad 
    a_1[y_1; x_1] \bxt a_2[y_2; x_2] \xrightarrow{\tau}
    a_2[y_2; x_2] \circ a_1[y_1; x_1]\rlap{ ;}
  \end{equation*}
  and the appropriate analogue of~\eqref{eq:v-bifunctor} will feature
  ${\ca{A}_1}(y_1, x_1) \bxt {\ca{A}_2}(y_2, x_2)$ farthest to the
  left, and with each other $\otimes$ replaced by $\circ$. In this
  way, we re-find precisely Definition~10 of~\cite{GarnerLopezFranco},
  expressing the notion of $\ca{V}$-bifunctor over a normal duoidal
  base $\ca{V}$.
\end{rex}
We wish to extend this definition to arbitrary monad multimorphisms.
The idea is that a multimorphism
$(f, \vec \phi) \co (A_1, a_1), \dots, (A_n, a_n) \to (B, b)$ should
commute just when it commutes in each pair $(i,j)$ of distinct input
arguments. There are two ways to go about setting this up formally.
One involves defining ``$n$-ary interchanger'' $2$-morphisms
\begin{equation}\label{eq:nary-interchanger}
  \begin{tikzcd}[column sep=.4in]
    A_1\bxt \dots \bxt A_n \ar[rr,tick,"(y_1\circ x_1)\bxt\,\cdots\,\bxt(y_n\circ x_n)"]\ar[d,equal]\ar[drr,phantom,"\Two \xi^{(n)}"] && C_1\bxt \cdots \bxt C_n\ar[d,equal] \\
    A_1\bxt\cdots\bxt A_n\ar[r,tick,"x_1\bxt \cdots\bxt x_n"'] & B_1\bxt\cdots\bxt B_n\ar[r,tick,"y_1\bxt \cdots\bxt y_n"'] & C_1\bxt\cdots\bxt C_n
  \end{tikzcd}
\end{equation}
built recursively from the binary interchangers $\xi$, using these to
define $n$-ary analogues of the maps $\sigma$ and $\tau$
of~\cref{def:sigma-tau}, and then setting up a hexagon, similar
to~\cref{eq:binary-commutativity-hexagon} but more elaborate, which
expresses directly the $(i,j)$-commutativity of $(f, \vec \phi)$.
However, we prefer to avoid this, and adopt the following ``quick and
dirty'' definition, which uses \cref{not:indexconventions}:

\begin{defi} \label{defi:commuting-monad-multimorphism} Let
  $(f, \vec \phi) = (f, \phi_1, \dots, \phi_n) \co (A_1, a_1), \dots,
  (A_n, a_n) \to (B, b)$ be a monad multimorphism and 
  $1 \leqslant i < j \leqslant n$. We say that $(f, \vec \phi)$
  is \emph{$(i,j)$-commuting} if the binary multimorphism
\begin{equation*}
    (f, \phi_i, \phi_j) \colon \tparen{\disc A}_{[1, i\mi 1]} \bxt (A_i, a_i) \bxt \tparen{\disc A}_{[i+1, j\mi 1]},\,\, (A_j, a_j) \bxt \tparen{\disc A}_{[j+1, n]} \rightarrow (B,b)
  \end{equation*}
  is commuting. Where there is no scope for confusion, we may also say
  that $(f, \vec \phi)$ is \emph{$(A_i, A_j)$-commuting}. We say that
  $(f, \smash{\vec \phi})$ is \emph{commuting}, we mean that it is
  $(i,j)$-commuting for all $1 \leqslant i < j \leqslant n$.
\end{defi}
The reason we term this definition ``quick and dirty'', rather than
merely quick, is that we have chosen, in a slightly non-canonical way,
to attach the tensor of discrete monads
$\tparen{\disc A}_{[i+1, j\mi 1]}$ to the first input argument of
$(f, \phi_i, \phi_j)$; we could equally have attached it to the second
input argument, or indeed, have split it into two parts, one attached
to the first and one to the second. It is not hard to show that these
choices are irrelevant, for example by showing that each of them
yields the same notion of commutativity as the approach involving the
$n$-ary interchangers~\eqref{eq:nary-interchanger}.

The following lemma ensures that the symmetric multicategory structure of monads and monad multimorphisms (\cref{prop:MultMndmulticategory}) restricts to commuting monad multimorphisms in an appropriate way.

\begin{lem}
  \label{lem:properties-of-commutativity}
  Let $(f, \vec \phi) \colon (B_1, b_1), \dots,
  (B_n, b_n) \to (C, c)$ and $(g, \vec \psi) \colon (A_1, a_1), \dots,
  (A_m, a_m) \to (B_k, b_k)$ be monad multimorphisms.

  \begin{enumerate}[(i)]
\item If $(f, \vec \phi)$ is $(i,j)$-commuting, and $\sigma \in
  \mathfrak{S}_n$, then $(f, \vec \phi) \circ \sigma$ is
  $(\sigma(i),\sigma(j))$-commuting.
\item If $(f, \vec \phi)$ is $(B_i,B_j)$-commuting for $i,j \neq k$, then
  $(f, \vec \phi) \circ_k (g, \vec \psi)$ is $(B_i,B_j)$-commuting.
\item If $(f, \vec \phi)$ is $(B_i,B_k)$-commuting, then
  $(f, \vec \phi) \circ_k (g, \vec \psi)$ is $(B_i,A_j)$-commuting for
  all $1 \leqslant j \leqslant m$.
\item If $(g, \vec \psi)$ is $(A_i,A_j)$-commuting,
  then $(f, \vec \phi) \circ_k (g, \vec \psi)$ is
  $(A_i,A_j)$-commuting.
  \end{enumerate}
\end{lem}
\begin{proof}
  This is a straightforward diagram chase using the definitions of
  commutativity and composition of monad multimorphisms, and the
  irrelevance of the positioning of the medial elements
  $\tparen{\disc A}_{[i+1, j\mi 1]}$.
\end{proof}
It immediately follows that:
\begin{prop}\label{prop:CMultMnd}
  The monads and commuting monad multimorphisms form a sub-symmetric
  multicategory $\CommMultMndC$ of
  $\MultMndC$.
\end{prop}

\subsection{Pre-representability of $\CommMultMndC$}

Our goal now is to show that the symmetric multicategory of commuting
multimorphisms in $\dcC$ is representable. We have already seen one
case where already know this is true; indeed, when $\dcC = \VMat$ for
a symmetric monoidal closed cocomplete $\ca{V}$, we argued
in~\cref{thm:running-27} that the universal binary commuting
multimorphism should be represented by the usual tensor product of
$\ca{V}$-categories. However, in general it is the case that further
work will be required.

As in \cref{sub:prerepMultMnd}, we begin this work by constructing
pre-universal nullary and binary commuting multimorphisms. For this,
we assume that $\dcC$ satisfies
\cref{hyp:fibrancy,hyp:coproducts-in-mnd,hyp:equalisers,hyp:tabulators}
of \cref{sec:non-commuting}, as well as the following two further
hypotheses:

\begin{hyp} \label{hyp:free-monad} Each forgetful functor
  $\MndCA \rightarrow \EndCA$ has a left adjoint
  $\freemnd \colon \EndCA \rightarrow \MndCA$.
\end{hyp}
    
\begin{hyp} \label{hyp:coequalisers} Each category $\MndCA$ has
  coequalisers.
\end{hyp}

Note that, since we have already assumed $\dcC$ to be
fibrant,~\cref{hyp:free-monad} together
with~\cref{prop:strongerfreemonads} tells us that $\dcC$ admits free
monads in the sense of~\cref{def:free-monads}, \ie, that the forgetful
functor $\MndC \rightarrow \EndC$ has a left adjoint over $\dcC_0$, denoted by $\freemnd(-)$.
Similarly, fibrancy,~\cref{hyp:coequalisers} and
\cref{thm:fibrancy-coequalizers}(ii) together tell us that $\MndC$ admits all
coequalisers of identity-on-objects parallel pairs. 

The following is the counterpart of \cref{thm:nctensor-prerepresentable}, establishing the existence of nullary and binary pre-universal multimorphisms.
\begin{prop}
  \label{prop:universal-commuting-binary}
  The multicategory $\CommMultMndC$ is pre-representable.
  \end{prop}
\begin{proof}
  Since nullary multimorphisms always commute, the universal nullary
  multimorphism $\ \rightarrow \disc{I}$ is also a universal nullary
  commuting multimorphisms. For the binary case, we must show that,
  for all $(A_1,a_1), (A_2,a_2) \in \MndC$, there is a pre-universal
  binary commuting multimorphism
  \begin{equation}
    \label{eq:universal-commuting-binary}
    (1_{A_1 \bxt A_2}, \eta_1, \eta_2) \colon (A_1,a_1), (A_2,a_2) \rightarrow (A_1 \bxt A_2, a_1 \ctensor a_2)\rlap{ .}
  \end{equation}

  In order to construct this, we start from the universal (non-commuting) binary
  multimorphism
  $(1_{A_1 \bxt A_2}, \iota_1, \iota_2) \colon (A_1,a_1), (A_2,a_2)
  \rightarrow (A_1 \bxt A_2, a_1 \nctensor a_2)$
  of~\cref{eq:universal-binary-multimorphism}, and consider the
  hexagon~\cref{eq:universal-commutativity-hexagon} in $\EndC$ which
  would express its $(1,2)$-commutativity:
  \begin{equation}
    \label{eq:universal-commutativity-hexagon}
    \begin{tikzcd}[column sep = 1em]
      &[-9em] (A_1 \bxt A_2, (a_1 \bxt \hid_{A_2}) \hcomp (\hid_{A_1} \bxt a_2)) \ar[r, "{(1, \iota_1 \hcomp \iota_2)}"] &[2em] (A_1 \bxt A_2, (a_1\nctensor a_2) \circ (a_1\nctensor a_2))  \ar[dr, "{(1, \mu)}"] &[-8em] \\
      (A_1 \bxt A_2, a_1 \bxt a_2) \ar[ur, "{(1,\sigma)}"] \ar[dr, "{(1,\tau)}"'] & & & (A_1 \bxt A_2, a_1\nctensor a_2) \mathrlap{ .} \\
      & (A_1 \bxt A_2, (\hid_{A_1} \bxt a_2) \hcomp (a_1 \bxt \hid_{A_2}))  \ar[r, "{(1,\iota_2 \hcomp \iota_1)}"'] & (A_1 \bxt A_2, (a_1\nctensor a_2) \hcomp (a_1\nctensor a_2))  \ar[ur, "{(1, \mu)}"'] & 
    \end{tikzcd} 
  \end{equation}

  There is, of course, no reason to believe that this hexagon will
  commute in general, but we can quotient $a_1\nctensor a_2$ to bring this about.
  Since $a_1\nctensor a_2$ bears a monad structure, we may transpose the
  two sides of the hexagon above to obtain a parallel pair of maps in
  $\MndC$, as to the left in:
  \begin{equation}
    \label{eq:commuting-coequaliser}
    \begin{tikzcd}
      (A_1 \bxt A_2, \freemnd(a_1 \bxt a_2)) \ar[r, shift left = 0.3em, "{(\id, \upsilon_1)}"]
      \ar[r, shift right = 0.3em, "{(\id, \upsilon_2)}"'] &
      (A_1 \bxt A_2, a_1 \nctensor a_2) \ar[r, two heads, "{(\id, q)}"] &
      (A_1 \bxt A_2, a_1 \ctensor a_2)\rlap{ .}
    \end{tikzcd}
  \end{equation}
  Now by the remarks above, this identity-on-objects pair admits a
  coequaliser in $\MndC$ as to the right. We claim that the object
  $(A_1 \bxt A_2, a_1 \ctensor a_2)$ so obtained is pre-universal, when
  equipped with the multimorphism~\cref{eq:universal-commuting-binary}
  given by the composite
  \begin{equation*}
    (\id, q) \circ (1_{A_1 \bxt A_2}, \iota_1, \iota_2) \colon (A_1,a_1), (A_2,a_2) \rightarrow (A_1 \bxt A_2, a_1 \ctensor a_2)\rlap{ .}
  \end{equation*}

  To establish the claim, let
  $(f, \phi_1, \phi_2) \colon (A_1,a_1), (A_2,a_2) \rightarrow (B,b)$
  be a commuting monad multimorphism. Since $(f, \phi_1, \phi_2)$ is
  in particular a monad multimorphism, it factors uniquely through the
  universal binary
  multimorphism~\cref{eq:universal-binary-multimorphism}, as
  $(f, \bar \phi) \colon (A_1\bxt A_2, a_1 \nctensor a_2) \rightarrow
  (B,b)$ say. It thus suffices to show that this $(f, \bar \phi)$
  factors (necessarily uniquely) through the coequaliser
  $(\id, q) \colon a \nctensor b \twoheadrightarrow a \ctensor b$. In
  other words, we must show that composing with $(f,\bar \phi)$
  makes equal the maps $(\id, \upsilon_1)$ and $(\id, \upsilon_2)$
  in~\cref{eq:commuting-coequaliser}.
  Transposing under the free-forgetful adjunction, this is equally to
  ask that the hexagon~\cref{eq:universal-commutativity-hexagon}
  commutes upon postcomposition by $(f, \bar \phi)$. But since
  $(f, \bar \phi)$ is a monad morphism and since
  $(f, \bar \phi) \circ (\id, \iota_1, \iota_2) = (f, \phi_1,
  \phi_2)$, this postcomposition is precisely the hexagon expressing
  the $(1,2)$-commutativity of $(f, \phi_1, \phi_2)$.
\end{proof}
Like before, there is no obstacle to extending the above result to
exhibit universal $n$-ary commuting multimorphisms: we start from the
universal non-commuting $n$-ary multimorphism, and then take repeated
coequalisers to impose $(i,j)$-commutativity for all $i < j$. But,
like before, we do not need to give this argument in full, since it
will follow once we have established \emph{closedness} of our
multicategory.

\subsection{Closure and representability of $\CommMultMndC$}

In this section, we
build the internal homs which exhibit
$\CommMultMndC$ as a closed multicategory. We do so using the
corresponding internal homs of $\MultMndC$ (see
\cref{prop:multimorphisms-closed}), and as a first step we record how
these interact with commutativity.

\begin{lem}
  \label{lem:commutative-maps-into-nchom}
  A monad multimorphism
  $(f, \phi_1, \dots, \phi_n, \psi) \colon (A_1, a_1), \dots, (A_n,
  a_n), (B, b) \rightarrow (C,c)$ is $(A_i, A_j)$-commuting if and
  only if its transpose
  $(\bar f, \bar \phi_1, \dots, \bar \phi_n) \colon (A_1, a_1), \dots,
  (A_n, a_n) \rightarrow \nchom{(B,b), (C,c)}$ is
  $(A_i, A_j)$-commuting.
\end{lem}
\begin{proof}
  Note that by \cref{defi:commuting-monad-multimorphism}, $(f, \vec \phi, \psi)$ is $(A_i, A_j)$-commuting if and
  only if
  \begin{equation}
    \label{eq:commutative-maps-into-nchom-1}
    (f, \phi_i, \phi_j, \psi) \colon \tparen{\disc A}_{[1, i\mi 1]} \bxt (A_i, a_i) \bxt \tparen{\disc A}_{[i+1, j\mi 1]},\,\, (A_j, a_j) \bxt \tparen{\disc A}_{[j+1, n]} \bxt (B,b) \rightarrow (C,c)
  \end{equation}
  is $(A_i, A_j)$-commuting, while its transpose is $(A_i,
  A_j)$-commuting if and only if
  \begin{equation}
    \label{eq:commutative-maps-into-nchom-2}
    (f, \bar \phi_i, \bar \phi_j) \colon \tparen{\disc A}_{[1, i\mi 1]} \bxt (A_i, a_i) \bxt \tparen{\disc A}_{[i+1, j\mi 1]},\,\, (A_j, a_j) \bxt \tparen{\disc A}_{[j+1, n]}  \rightarrow \nchom{(B,b), (C,c)}
  \end{equation}
  is $(A_i, A_j)$-commuting. It is easy to see
  that~\cref{eq:commutative-maps-into-nchom-1} and
  \cref{eq:commutative-maps-into-nchom-2} are themselves paired under
  transpose, so that without loss of generality we may reduce to the
  case $n = 2$.

  On doing so, to say that
  $(\bar f, \bar \phi_1, \bar \phi_2) \colon (A_1, a_1), (A_2, a_2)
  \rightarrow \nchom{(B,b), (C,c)}$ is $(A_1, A_2)$-commuting is to
  say that the following hexagon commutes:
  \begin{equation*}
    \begin{tikzcd}[column sep = 1em]
      &[-9em] (A_1 \bxt A_2, (a_1 \bxt \hid_{A_2}) \hcomp (\hid_{A_1} \bxt a_2)) \ar[r, "{(\bar f, {\bar \phi}_1 \hcomp {\bar \phi}_2)}"] &[2em] (\nchom{b,c}_0, \nchom{b,c} \circ \nchom{b,c})  \ar[dr, "{(1, \mu)}"] &[-5em] \\
      (A_1 \bxt A_2, a_1 \bxt a_2) \ar[ur, "{(1,\sigma)}"] \ar[dr, "{(1,\tau)}"'] & & & (\nchom{b,c}_0, \nchom{b,c}) \\
      & (A_1 \bxt A_2, (\hid_{A_1} \bxt a_2) \hcomp (a_1 \bxt
      \hid_{A_2})) \ar[r, "{(\bar f, {\bar \phi}_2 \hcomp {\bar \phi}_1)}"'] & (\nchom{b,c}_0, \nchom{b,c} \circ \nchom{b,c}) \ar[ur, "{(1,
        \mu)}"'] &
    \end{tikzcd}
  \end{equation*}
where the underlying object $\nchom{b,c}_0$ in $\dcC_0$ is the the object $\uMndC(b,c) \in \dcC_0$ of~\cref{lem:mnd-hom} like before. Clearly, we have commutativity at the level of underlying maps in
  $\dcC_0$; so it suffices to show commutativity after postcomposition
  by the map~\cref{eq:representably-induced}, which is cartesian with respect to both forgetful functors $\EndC \rightarrow \dcC_0$ and $\MndC \rightarrow \dcC_0$. This postcomposition is equally well the
  hexagon:
\begin{equation*}
    \begin{tikzcd}[column sep = 1em]
      &[-9em] (A_1 \bxt A_2, (a_1 \bxt \hid_{A_2}) \hcomp (\hid_{A_1} \bxt a_2)) \ar[r, "{(\tilde f, {\tilde \phi}_1 \hcomp {\tilde \phi}_2)}"] &[2em] (\boxhom{B,C}, \boxhom{\hid_B,c} \circ \boxhom{\hid_B,c})  \ar[dr, "{(1, \mu)}"] &[-6.5em] \\
      (A_1 \bxt A_2, a_1 \bxt a_2) \ar[ur, "{(1,\sigma)}"] \ar[dr, "{(1,\tau)}"'] & & & (\boxhom{B,C}, \boxhom{\hid_B,c}) \\
      & (A_1 \bxt A_2, (\hid_{A_1} \bxt a_2) \hcomp (a_1 \bxt
      \hid_{A_2})) \ar[r, "{(\tilde f, {\tilde \phi}_2 \hcomp {\tilde \phi}_1)}"'] & (\boxhom{B,C}, \boxhom{\hid_B,c} \circ \boxhom{\hid_B,c}) \ar[ur, "{(1,
        \mu)}"'] &
    \end{tikzcd}
  \end{equation*}
  where $\tilde f$, $\tilde \phi_1$, $\tilde \phi_2$ are transposes of
  $f$, $\phi_1$, $\phi_2$ under the adjunctions
  $(\thg) \bxt B \dashv \boxhom {B, \thg}$ and
  $(\thg) \bxt \hid_B \dashv \boxhom {\hid_B, \thg}$.

  Now, by the
  definition of the multiplication $\mu$ of the monad $\bigl( \boxhom{B,C}, \boxhom{ \hid_B,c} \bigr)$ in~\cref{eq:mu-for-internal-hom} together with
  the naturality of $\zeta$ therein, the transpose of the composite
  $(1, \mu)(\tilde f, \tilde \phi_1 \circ \tilde \phi_2)$ under
  $(\thg) \bxt (B, \hid_B) \dashv \boxhom {(B, \hid_B), \thg}$ is the composite
  \begin{equation*}
    \begin{tikzcd}[column sep = 4em]
      (A_1 \bxt A_2 \bxt B, ((a_1 \bxt \hid_{A_2}) \hcomp (\hid_{A_1} \bxt a_2)) \bxt \hid_B) \ar[d,"{(1, \zeta)}"] & (C,c)\\
      (A_1 \bxt A_2 \bxt B, (a_1 \bxt \hid_{A_2} \bxt \hid_B) \hcomp (\hid_{A_1} \bxt a_2 \bxt \hid_B)) \ar[r,"{(f, \phi_1 \circ \phi_2)}"'] & (C,c \circ c) \ar[u,"{(1, \mu)}"]
    \end{tikzcd}
  \end{equation*}
  and correspondingly for $(1, \mu)(\tilde f, \tilde \phi_2 \circ
  \tilde \phi_1)$. Thus, the hexagon above commutes just when the the outside
  of the diagram
  \begin{equation*}
    \begin{tikzcd}[column sep = 1em]
      &[-15em] |[xshift=-3em]| {(A_1\!\bxt A_2\!\bxt\!B, ((a_1 \bxt \hid) \!\circ\! (\hid \bxt a_2)) \bxt \hid)} \ar[rr, "{(1, \zeta)}"] &[1em]&[-15em] (A_1\!\bxt A_2\!\bxt\!B, (a_1 \bxt \hid \bxt \hid) \!\circ\! (\hid \bxt a_2 \bxt \hid)) \ar[r, "{(f, {\phi}_1 \hcomp {\phi}_2)}"] &[1.3em]
      (C, c \!\circ\! c)  \ar[dr, "{(1, \mu)}"] &[-3em] \\
      (A_1\!\bxt A_2\!\bxt\!B, (a_1 \bxt a_2) \bxt \hid) \ar[ur, "{(1,\sigma \bxt 1)}"] \ar[dr, "{(1,\tau \bxt 1)}"'] \ar[rr,"\cong"] & & {(A_1\!\bxt A_2\!\bxt\!B, (a_1 \bxt a_2) \bxt \hid)} \ar[ur,"{(1, \sigma)}"] \ar[dr,"{(1, \tau)}"] & & & (C,c) \mathrlap{ .} \\
      & |[xshift=-3em]| (A_1\!\bxt A_2\!\bxt\!B, ((\hid \bxt a_2) \!\circ\! (a_1 \bxt \hid)) \bxt \hid) \ar[rr, "{(1, \zeta)}"] && (A_1\!\bxt A_2\!\bxt\!B, (a_1 \bxt \hid \bxt \hid) \!\circ\! (\hid \bxt a_2 \bxt \hid)) \ar[r, "{(f, {\phi}_2 \hcomp {\phi}_1)}"'] & (C,c \!\circ\! c) \ar[ur, "{(1,
        \mu)}"'] &
    \end{tikzcd}
  \end{equation*}
  commutes. Since the left two regions of this diagram commute by the
  normal oplax monoidal coherence axioms, and since the associativity
  constraint centre left is invertible, we conclude that the outside
  commutes just when the right hexagon commutes. But this hexagon
  expresses precisely the $(A_1, A_2)$-commutativity of
  $(f, \phi_1, \phi_2, \psi)$. Thus, we have shown that
  $(\bar f, \bar \phi_1, \bar \phi_2) \colon (A_1, a_1), (A_2, a_2)
  \rightarrow \nchom{(B,b), (C,c)}$ is $(A_1, A_2)$-commuting just
  when $(f, \phi_1, \phi_2, \psi) \colon (A_1, a_1), (A_2, a_2), (B,b)
  \rightarrow (C,c)$ is so, as required.
\end{proof}

We will require one further result before exhibiting the closed
structure of $\CommMultMndC$; informally speaking, this establishes
that, for a binary multimorphism
$(A_1, a_1), (A_2, a_2) \rightarrow (B,b)$, the (generalised) elements
of $(A_1, a_1)$ which commute with $(A_2, a_2)$ are closed under
composition.

\begin{lem}
  \label{lem:compositionality-of-commutativity}
  Let $(f, \phi_1, \phi_2) \colon (A_1, a_1), (A_2, a_2) \rightarrow
  (B,b)$ be a monad multimorphism, and consider the
  hexagon~\eqref{eq:binary-commutativity-hexagon} which would express
  the commutativity of $f$.
  \begin{enumerate}[(i)]
  \item This hexagon always
    commutes on precomposition by the map
    \begin{equation*}
      (A_1 \bxt A_2, \hid_{A_1} \bxt
      a_2) \xrightarrow{(1, \eta \bxt 1)} (A_1 \bxt A_2, a_1 \bxt a_2)\rlap{ .}
    \end{equation*}
  \item If this hexagon 
    commutes on precomposition by maps
    \begin{equation*}
      (A_1 \bxt A_2, a_1' \bxt a_2) 
      \xrightarrow{(1, u \bxt 1)} (A_1 \bxt A_2, a_1 \bxt a_2) \quad \text{and} \quad 
       (A_1 \bxt A_2, a_1'' \bxt a_2)
      \xrightarrow{(1, v \bxt 1)} (A_1 \bxt A_2, a_1 \bxt a_2)
    \end{equation*}
    then it also commutes on precomposition by the map
    \begin{equation}
      \label{eq:uv-composite}
      (A_1 \bxt A_2, (a_1' \circ a_1'') \bxt a_2)
      \xrightarrow{(1, (u \circ v) \bxt 1)} (A_1 \bxt A_2, (a_1 \circ a_1) \bxt a_2) \xrightarrow{(1, \mu \bxt 1)} (A_1 \bxt A_2, a_1 \bxt a_2)\rlap{ .} \end{equation}
  \end{enumerate}
\end{lem}
\begin{proof}
  Both parts are proved in an identical fashion, so we concentrate on
  the more involved case of (ii).
  For this, note that the top path
  around~\eqref{eq:binary-commutativity-hexagon} becomes the composite
  below left on precomposition by~\eqref{eq:uv-composite}, which by
  naturality of $\sigma$, is equally the composite to the right:
  \begin{equation*}
    \begin{tikzcd}[column sep = 1.4em]
      A_1 \bxt A_2 \ar[rr,tick,"(a_1' \circ a_1'') \bxt a_2"] \ar[drr,phantom,"\Two (u \circ v) \bxt 1"] \ar[d,equal] &&
      A_1 \bxt A_2 \ar[d,equal] \\
      A_1 \bxt A_2 \ar[rr,tick,"(a_1 \circ a_1) \bxt a_2"] \ar[drr,phantom,"\Two \mu \bxt 1"] \ar[d,equal] &&
      A_1 \bxt A_2 \ar[d,equal] \\
      A_1 \bxt A_2 \ar[rr,tick,"a_1 \bxt a_2"] \ar[d,equal] \ar[drr,phantom,"\Two \sigma"]&&
      A_1 \bxt A_2 \ar[d,equal]\\
      A_1 \bxt A_2 \ar[r,tick,"\hid_{A_1} \bxt a_2"] \ar[d,"f"'] \ar[dr,phantom,"\Two \phi_2"] &
      A_1 \bxt A_2 \ar[r, tick,"a_1 \bxt \hid_{A_2}"] \ar[d,"f"] \ar[dr,phantom,"\Two \phi_1"] &
      A_1 \bxt A_2 \ar[d,"f"] \\
      B \ar[r,tick,"b"'] \ar[d,equal] \ar[drr,phantom,"\Two \mu"] &
      B \ar[r,tick,"b"'] &
      B \ar[d,equal] \\
      B \ar[rr,tick,"b"'] && B
    \end{tikzcd} = 
\begin{tikzcd}[column sep = 1.4em]
      A_1 \bxt A_2 \ar[rr,tick,"(a_1' \circ a_1'') \bxt a_2"] \ar[drr,phantom,"\Two \sigma"] \ar[d,equal] &&
      A_1 \bxt A_2 \ar[d,equal] \\
      A_1 \bxt A_2 \ar[r,tick,"\hid_{A_1} \bxt a_2"] \ar[d,equal]  &
      A_1 \bxt A_2 \ar[r, tick,"(a_1' \circ a_1'') \bxt \hid_{A_2}"] \ar[d,equal] \ar[dr,phantom,"\Two (u \circ v) \bxt 1"] &
      A_1 \bxt A_2 \ar[d,equal] \\
      A_1 \bxt A_2 \ar[r,tick,"\hid_{A_1} \bxt a_2"] \ar[d,equal]  &
      A_1 \bxt A_2 \ar[r, tick,"(a_1 \circ a_1) \bxt \hid_{A_2}"] \ar[d,equal] \ar[dr,phantom,"\Two \mu \bxt 1"] &
      A_1 \bxt A_2 \ar[d,equal] \\
      A_1 \bxt A_2 \ar[r,tick,"\hid_{A_1} \bxt a_2"] \ar[d,"f"'] \ar[dr,phantom,"\Two \phi_2"] &
      A_1 \bxt A_2 \ar[r, tick,"a_1 \bxt \hid_{A_2}"] \ar[d,"f"] \ar[dr,phantom,"\Two \phi_1"] &
      A_1 \bxt A_2 \ar[d,"f"] \\
      B \ar[r,tick,"b"'] \ar[d,equal] \ar[drr,phantom,"\Two \mu"] &
      B \ar[r,tick,"b"'] &
      B \ar[d,equal] \\
      B \ar[rr,tick,"b"'] && B \mathrlap{.} 
    \end{tikzcd}
  \end{equation*}
  Now using the fact that $(f,\phi_1)$ is a monad morphism, this
  pasting is equally the pasting to the left in:
  \begin{equation*}
    \begin{tikzcd}[column sep = 1.4em]
      A_1 \bxt A_2 \ar[rrr,tick,"(a_1' \circ a_1'') \bxt a_2"] \ar[drrr,phantom,"\Two \sigma"] \ar[d,equal] &&&
      A_1 \bxt A_2 \ar[d,equal] \\
      A_1 \bxt A_2 \ar[r,tick,"\hid_{A_1} \bxt a_2"] \ar[d,equal]  &
      A_1 \bxt A_2 \ar[rr, tick,"(a_1' \circ a_1'') \bxt \hid_{A_2}"] \ar[d,equal] \ar[drr,phantom,"\Two \cong"] &&
      A_1 \bxt A_2 \ar[d,equal] \\
      A_1 \bxt A_2 \ar[r,tick,"\hid_{A_1} \bxt a_2"] \ar[d,equal]  &
      A_1 \bxt A_2 \ar[rr, tick,"(a_1' \circ a_1'') \bxt (\hid_{A_2} \circ \hid_{A_2})"] \ar[d,equal] \ar[drr,phantom,"\Two \xi"] &&
      A_1 \bxt A_2 \ar[d,equal] \\
      A_1 \bxt A_2 \ar[r,tick,"\hid_{A_1} \bxt a_2"] \ar[d,equal]  &
      A_1 \bxt A_2 \ar[r, tick,"a_1'' \bxt \hid_{A_2}"] \ar[d,equal] \ar[dr,phantom,"\Two v \bxt 1"] &
      A_1 \bxt A_2 \ar[r, tick,"a_1' \bxt \hid_{A_2}"] \ar[d,equal] \ar[dr,phantom,"\Two u \bxt 1"] &
      A_1 \bxt A_2 \ar[d,equal] \\
      A_1 \bxt A_2 \ar[r,tick,"\hid_{A_1} \bxt a_2"] \ar[d,"f"'] \ar[dr,phantom,"\Two \phi_2"] &
      A_1 \bxt A_2 \ar[r, tick,"a_1 \bxt \hid_{A_2}"] \ar[d,"f"] \ar[dr,phantom,"\Two \phi_1"] &
      A_1 \bxt A_2 \ar[r, tick,"a_1 \bxt \hid_{A_2}"] \ar[d,"f"] \ar[dr,phantom,"\Two \phi_1"] &
      A_1 \bxt A_2 \ar[d,"f"] \\
      B \ar[r,tick,"b"] \ar[d,equal] &
      B \ar[r, tick,"b"] \ar[d,equal] \ar[drr,phantom,"\Two \mu"] &
      B \ar[r, tick,"b"] &
      B \ar[d,equal] \\
      B \ar[r,tick,"b"'] \ar[d,equal] \ar[drrr,phantom,"\Two \mu"] &
      B \ar[rr,tick,"b"'] &&
      B \ar[d,equal] \\
      B \ar[rrr,tick,"b"'] &&& B
    \end{tikzcd} = 
    \begin{tikzcd}[column sep = 1.4em]
      A_1 \bxt A_2 \ar[rrr,tick,"(a_1' \circ a_1'') \bxt a_2"] \ar[drrr,phantom,"\cong"] \ar[d,equal] &&&
      A_1 \bxt A_2 \ar[d,equal] \\
      A_1 \bxt A_2 \ar[rrr,tick,"(a_1' \circ a_1'') \bxt (\hid_{A_2} \circ a_2)"] \ar[drrr,phantom,"\Two \xi"] \ar[d,equal] &&&
      A_1 \bxt A_2 \ar[d,equal] \\
      A_1 \bxt A_2 \ar[rr,tick,"a_1' \bxt a_2"] \ar[d,equal]  \ar[drr,phantom, "\Two \sigma"] &&
      A_1 \bxt A_2 \ar[r, tick,"a_1' \bxt \hid_{A_2}"] \ar[d,equal]  &
      A_1 \bxt A_2 \ar[d,equal] \\
      A_1 \bxt A_2 \ar[r,tick,"\hid_{A_1} \bxt a_2"] \ar[d,equal]  &
      A_1 \bxt A_2 \ar[r, tick,"a_1'' \bxt \hid_{A_2}"] \ar[d,equal] \ar[dr,phantom,"\Two v \bxt 1"] &
      A_1 \bxt A_2 \ar[r, tick,"a_1' \bxt \hid_{A_2}"] \ar[d,equal] \ar[dr,phantom,"\Two u \bxt 1"] &
      A_1 \bxt A_2 \ar[d,equal] \\
      A_1 \bxt A_2 \ar[r,tick,"\hid_{A_1} \bxt a_2"] \ar[d,"f"'] \ar[dr,phantom,"\Two \phi_2"] &
      A_1 \bxt A_2 \ar[r, tick,"a_1 \bxt \hid_{A_2}"] \ar[d,"f"] \ar[dr,phantom,"\Two \phi_1"] &
      A_1 \bxt A_2 \ar[r, tick,"a_1 \bxt \hid_{A_2}"] \ar[d,"f"] \ar[dr,phantom,"\Two \phi_1"] &
      A_1 \bxt A_2 \ar[d,"f"] \\
      B \ar[r,tick,"b"] \ar[d,equal] \ar[drr,phantom,"\Two \mu"] &
      B \ar[r, tick,"b"]  &
      B \ar[r, tick,"b"] \ar[d,equal]&
      B \ar[d,equal] \\
      B \ar[rr,tick,"b"'] \ar[d,equal] \ar[drrr,phantom,"\Two \mu"] &&
      B \ar[r,tick,"b"'] &
      B \ar[d,equal] \\
      B \ar[rrr,tick,"b"'] &&& B
    \end{tikzcd}
  \end{equation*}
  Here, we have additionally used the naturality of $\xi$ and the unit
  constraints of $\dcC$ to commute $u \bxt 1$ and $v \bxt 1$ past
  them. Now, since $\sigma$ is built using the interchanger $\xi$ and
  coherences for $\dcC$, a straightforward coherence calculation shows
  that the first three rows to the left above are equally well the
  first three rows to the right; therein, we have also used
  associativity for the monad $(B,b)$ to rewrite the bottom two rows.

  Considering the pasting so obtained, note that by naturality of
  $\sigma$, we may move $v \bxt 1$ above $\sigma$, and on doing so,
  the resulting diagram contains a copy of the top path
  of~\eqref{eq:binary-commutativity-hexagon}, precomposed by
  $v \bxt 1$. By our assumption, this is equal to the bottom path
  of~\eqref{eq:binary-commutativity-hexagon} precomposed by
  $v \bxt 1$; so rewriting by this, and using naturality of $\sigma$
  again, we see that the pasting above right is equally the pasting to
  the left in:
  \begin{equation*}
    \begin{tikzcd}[column sep = 1.4em]
      A_1 \bxt A_2 \ar[rrr,tick,"(a_1' \circ a_1'') \bxt a_2"] \ar[drrr,phantom,"\cong"] \ar[d,equal] &&&
      A_1 \bxt A_2 \ar[d,equal] \\
      A_1 \bxt A_2 \ar[rrr,tick,"(a_1' \circ a_1'') \bxt (\hid_{A_2} \circ a_2)"] \ar[drrr,phantom,"\Two \xi"] \ar[d,equal] &&&
      A_1 \bxt A_2 \ar[d,equal] \\
      A_1 \bxt A_2 \ar[rr,tick,"a_1' \bxt a_2"] \ar[d,equal]  \ar[drr,phantom, "\Two \tau"] &&
      A_1 \bxt A_2 \ar[r, tick,"a_1' \bxt \hid_{A_2}"] \ar[d,equal]  &
      A_1 \bxt A_2 \ar[d,equal] \\
      A_1 \bxt A_2 \ar[r, tick,"a_1'' \bxt \hid_{A_2}"] \ar[d,equal] \ar[dr,phantom,"\Two v \bxt 1"] &
      A_1 \bxt A_2 \ar[r,tick,"\hid_{A_1} \bxt a_2"] \ar[d,equal]  &
      A_1 \bxt A_2 \ar[r, tick,"a_1' \bxt \hid_{A_2}"] \ar[d,equal] \ar[dr,phantom,"\Two u \bxt 1"] &
      A_1 \bxt A_2 \ar[d,equal] \\
      A_1 \bxt A_2 \ar[r, tick,"a_1 \bxt \hid_{A_2}"] \ar[d,"f"'] \ar[dr,phantom,"\Two \phi_1"] &
      A_1 \bxt A_2 \ar[r,tick,"\hid_{A_1} \bxt a_2"] \ar[d,"f"] \ar[dr,phantom,"\Two \phi_2"] &
      A_1 \bxt A_2 \ar[r, tick,"a_1 \bxt \hid_{A_2}"] \ar[d,"f"] \ar[dr,phantom,"\Two \phi_1"] &
      A_1 \bxt A_2 \ar[d,"f"] \\
      B \ar[r,tick,"b"] \ar[d,equal] \ar[drr,phantom,"\Two \mu"] &
      B \ar[r, tick,"b"]  &
      B \ar[r, tick,"b"] \ar[d,equal]&
      B \ar[d,equal] \\
      B \ar[rr,tick,"b"'] \ar[d,equal] \ar[drrr,phantom,"\Two \mu"] &&
      B \ar[r,tick,"b"'] &
      B \ar[d,equal] \\
      B \ar[rrr,tick,"b"'] &&& B
    \end{tikzcd} = \begin{tikzcd}[column sep = 1.4em]
      A_1 \bxt A_2 \ar[rrr,tick,"(a_1' \circ a_1'') \bxt a_2"] \ar[drrr,phantom,"\cong"] \ar[d,equal] &&&
      A_1 \bxt A_2 \ar[d,equal] \\
      A_1 \bxt A_2 \ar[rrr,tick,"(a_1' \circ \hid_{A_1} \circ a_1'') \bxt (\hid_{A_2} \circ a_2 \circ \hid_{A_2})"] \ar[drrr,phantom,"\Two \xi^{(3)}"] \ar[d,equal] &&&
      A_1 \bxt A_2 \ar[d,equal] \\
      A_1 \bxt A_2 \ar[r, tick,"a_1'' \bxt \hid_{A_2}"] \ar[d,equal] \ar[dr,phantom,"\Two v \bxt 1"] &
      A_1 \bxt A_2 \ar[r,tick,"\hid_{A_1} \bxt a_2"] \ar[d,equal]  &
      A_1 \bxt A_2 \ar[r, tick,"a_1' \bxt \hid_{A_2}"] \ar[d,equal] \ar[dr,phantom,"\Two u \bxt 1"] &
      A_1 \bxt A_2 \ar[d,equal] \\
      A_1 \bxt A_2 \ar[r, tick,"a_1 \bxt \hid_{A_2}"] \ar[d,"f"'] \ar[dr,phantom,"\Two \phi_1"] &
      A_1 \bxt A_2 \ar[r,tick,"\hid_{A_1} \bxt a_2"] \ar[d,"f"] \ar[dr,phantom,"\Two \phi_2"] &
      A_1 \bxt A_2 \ar[r, tick,"a_1 \bxt \hid_{A_2}"] \ar[d,"f"] \ar[dr,phantom,"\Two \phi_1"] &
      A_1 \bxt A_2 \ar[d,"f"] \\
      B \ar[r,tick,"b"] \ar[d,equal] \ar[drrr,phantom,"\Two \mu^{(3)}"] &
      B \ar[r, tick,"b"]  &
      B \ar[r, tick,"b"] &
      B \ar[d,equal] \\
      B \ar[rrr,tick,"b"'] &&& B \mathrlap{.}
    \end{tikzcd}
  \end{equation*}
  Finally, using oplax monoidal coherence, we can rewrite the top
  three rows as the top two rows to the right above; we also rewrite
  the bottom two rows using the evident notation
  $\mu^{(3)} = \mu(\mu \circ 1) = \mu(1 \circ \mu)$. Note that the
  resulting diagram is perfectly symmetrical; whence, by a
  symmetric argument, we may also reduce the lower side
  of~\eqref{eq:binary-commutativity-hexagon}, precomposed
  by~\eqref{eq:uv-composite}, to the same composite. This
  establishes~(ii).

  We leave the simpler case of (i) to the reader, who by following the
  same kind of argument as above, should have no difficulty in verifying that, on precomposition by $(1, \eta \bxt 1)$,
  both paths around~\eqref{eq:binary-commutativity-hexagon} are equal
  to
  $(f, \phi_2) \colon (A_1 \bxt A_2, \hid_{A_1} \bxt a_2) \rightarrow
  (B,b)$.
\end{proof}

We now proceed to exhibit the closed structure of $\CommMultMndC$. For
this, we will need, in addition to our previous assumptions:

\begin{hyp}
\label{hyp:equalisers-in-horizontal} 
$\dcC$ has local equalisers.
\end{hyp}
Since $\dcC_0$ has equalisers by~\cref{hyp:equalisers}, it follows
from this assumption and~\cref{lem:local-to-global-limits} that $\dcC$ has
parallel equalisers. Thus, by \cref{lem:limits-in-end-mnd}(i), so does  $\EndC$.

\begin{prop}
  \label{prop:commuting-multimorphisms-closed}
  The symmetric multicategory $\CommMultMndC$ is closed. That is, for all
  $(B,b), (C,c) \in \MndC$, there exists a hom-object $\chom{(B,b), (C,c)}$ and binary multimorphism
  \begin{equation*}
    (e^c, \varepsilon^c_1, \varepsilon^c_2) \colon \chom{(B,b), (C,c)}, (B, b) \rightarrow (C, c)
  \end{equation*}
  composition with which establishes a bijection of multimorphism-sets
  \begin{equation}\label{eq:commuting-internal-hom-bij}
    \begin{array}{r@{\,}l@{\ }c@{\ }ll}
      (f, \phi_1, \dots, \phi_n, \psi) \colon& (A_1, a_1), \dots,
      (A_n, a_n), (B,b) &\longrightarrow& (C,c) & \text{ commuting}\\
      \noalign{\vskip 3pt}
      \hline
      \noalign{\vskip 3pt} 
      (\bar f, \bar{\phi}_1, \dots, \bar{\phi}_n) \colon& (A_1, a_1), \dots, (A_n, a_n) &\longrightarrow& \chom{(B,b),(C,c)}& \text{ commuting}
    \end{array}
  \end{equation}
\end{prop}
\begin{proof}
  Consider the binary multimorphism
  $(e, \varepsilon_1, \varepsilon_2) \colon \nchom{(B,b), (C,c)},
  (B, b) \rightarrow (C, c)$ providing, as
  in~\cref{eq:counit-multimorphisms-closed}, the evaluation map of the
  non-commuting internal hom. The following is the hexagon which would
  express that this is a commuting binary multimorphism:
  \begin{equation}
    \label{eq:internal-hom-hexagon}
    \begin{tikzcd}[column sep = 0.6cm, row sep = 0.6cm]
      &[-11em] (\nchom{b,c}_0 \bxt B,(\nchom{b,c} \bxt \hid_{B}) \hcomp (\hid_{\nchom{b,c}_0} \bxt b)) \ar[r, "{(e, \varepsilon_1 \hcomp \varepsilon_2)}"] &[2em] (C,c \circ c)  \ar[dr, "{(1, \mu)}"] &[-1em] \\
      (\nchom{b,c}_0 \bxt B, \nchom{b,c} \bxt b) \ar[ur, "{(1,\sigma)}"] \ar[dr, "{(1,\tau)}"'] & & & (C,c) \mathrlap{ .} \\
      & (\nchom{b,c}_0 \bxt B, (\hid_{\nchom{b,c}_0} \bxt b) \hcomp (\nchom{b,c} \bxt \hid_{B}))   \ar[r, "{(e,\varepsilon_2 \hcomp \varepsilon_1)}"'] & (C,c \hcomp c)  \ar[ur, "{(1, \mu)}"'] & 
    \end{tikzcd} 
  \end{equation}
  Like before, there is no reason to expect this diagram to commute in general.
  However, the composites around the top and bottom sides constitute a
  parallel pair of maps $(\nchom{b,c}_0, \nchom{b,c}) \bxt (B,b)
  \rightarrow (C,c)$, and taking their transposes with respect to
  the closed structure of $\EndC$, we obtain a parallel pair as to the
  right in:
  \begin{equation}
    \label{eq:commuting-hom-equaliser}
    \begin{tikzcd}
      (\chom{b,c}_0, \chom{b,c})  \ar[r, tail arrow, "{(1, \kappa)}"] &
      (\nchom{b,c}_0, \nchom{b,c}) \ar[r, shift left = 0.3em, "{(i, \nu_1)}"]
      \ar[r, shift right = 0.3em, "{{(i, \nu_2)}}"'] &
      \boxhom{(B,b),(C,c)} \rlap{ ,}
    \end{tikzcd}
  \end{equation}
  where $i \colon \nchom{b,c}_0 \rightarrow \boxhom{B,C}$ is as
  in~\cref{eq:representably-induced}.
  By~\cref{hyp:equalisers-in-horizontal} and the remarks following,
  this pair admits an equaliser in $\EndC$ as displayed. We will show
  that the subobject of $(\nchom{b,c}_0, \nchom{b,c})$ so obtained is
  in fact a submonad, and that it represents the desired internal hom
  in $\CommMultMndC$.

  Towards the first of these goals, observe that a map
  $(f, \phi) \colon (A, x) \rightarrow (\nchom{b,c}_0, \nchom{b,c})$
  in $\EndC$ factors through $(1, \kappa)$ precisely when
  $(\thg) \circ (f, \phi)$ renders the two maps $(i, \nu_1)$ and
  $(i, \nu_2)$ equal; which, by transposing under the adjunction
  $(\thg) \bxt (B,b) \dashv \boxhom{(B,b), \thg}$ on $\EndC$, happens
  just when the hexagon~\eqref{eq:internal-hom-hexagon} commutes on
  precomposition by $(f \bxt 1_B, \phi \bxt 1_B)$. In particular,
  \eqref{eq:internal-hom-hexagon} will commute on precomposition by
  $(1, \kappa \bxt 1)$. Using these observations, we can now show that
  $(\chom{b,c}_0, \chom{b,c})$ is a submonad of
  $(\nchom{b,c}_0, \nchom{b,c})$. It is enough to show that
  there are factorisations as to the left and right in:
  \begin{equation*}
    \begin{tikzcd}
      (\chom{b,c}_0, \hid_{\chom{b,c}_0}) \ar[r,dashed]
      \ar[d,equal] &
      (\chom{b,c}_0, \chom{b,c}) \ar[d,"{(1,\kappa)}"] \\
      (\nchom{b,c}_0, \hid_{\nchom{b,c}_0}) \ar[r,"{(1, \eta)}"] &
      (\nchom{b,c}_0, \nchom{b,c})
    \end{tikzcd} \qquad \quad 
    \begin{tikzcd}
      (\chom{b,c}_0, \chom{b,c} \circ \chom{b,c}) \ar[r,dashed]
      \ar[d,"{(1, \kappa \circ \kappa)}"] &
      (\chom{b,c}_0, \chom{b,c}) \ar[d,"{(1,\kappa)}"] \\
      (\nchom{b,c}_0, \nchom{b,c} \circ \nchom{b,c}) \ar[r,"{(1, \mu)}"] &
      (\nchom{b,c}_0, \nchom{b,c})
    \end{tikzcd}
  \end{equation*}
  where $\eta$ and $\mu$ are the maps exhibiting the monad structure
  of $\nchom{b,c}$. But to say that $(1, \eta)$ factors through
  $(1, \kappa)$ is equivalent to saying that~\eqref{eq:internal-hom-hexagon}
  commutes on precomposition by $(1, \eta \bxt 1)$, which is indeed the case
  by~\cref{lem:compositionality-of-commutativity}(i). On the other
  hand,   to say
  that $(1, \mu)(1, \kappa \circ \kappa)$ factors through
  $(1, \kappa)$ is to say that~\eqref{eq:internal-hom-hexagon}
  commutes on precomposition by 
  \begin{equation*}
      (\chom{b,c}_0 \bxt B, (\chom{b,c} \circ \chom{b,c}) \bxt B)
      \xrightarrow{(1, (\kappa \circ \kappa) \bxt 1)} (\nchom{b,c}_0 \bxt B, (\nchom{b,c} \circ \nchom{b,c}) \bxt B) \xrightarrow{(1, \mu \bxt 1)} (\nchom{b,c}_0 \bxt B, \nchom{b,c} \bxt B)\rlap{ .}
  \end{equation*}
  But we observed above that~\eqref{eq:internal-hom-hexagon} commutes
  on precomposition by $(1, \kappa \bxt 1)$; whence this follows
  by~\cref{lem:compositionality-of-commutativity}(ii). Thus
  $(\chom{b,c}_0, \chom{b,c})$ is a submonad of
  $(\nchom{b,c}_0, \nchom{b,c})$; it remains to prove that it provides
  an internal hom in $\CommMultMndC$.

  To this end, we show the following: given
  $(f, \vec \phi, \psi) \colon (A_1, a_1), \dots, (A_n, a_n), (B,b)
  \rightarrow (C,c)$ a monad multimorphism, its transpose
  $(\bar f, \bar{\phi}_1, \dots, \bar{\phi}_n) \colon (A_1, a_1),
  \dots, (A_n, a_n) \rightarrow (\nchom{b,c}_0, \nchom{b,c})$ factors
  through the subobject
  $(\chom{b,c}_0, \chom{b,c}) \rightarrowtail (\nchom{b,c}_0,
  \nchom{b,c})$ just when $(f, \vec \phi, \psi)$ is
  $(A_i, B)$-commuting for all $i$. By arguing as at the start of the
  proof of~\cref{lem:commutative-maps-into-nchom}, we may immediately
  reduce to the case $n = 1$. Thus, suppose that
  $(f, \phi, \psi) \colon (A,a), (B,b) \rightarrow (C,c)$. To say that
  $(\bar f, \bar\phi) \colon (A,a) \rightarrow (\nchom{b,c}_0,
  \nchom{b,c})$ factors through $(\chom{b,c}_0, \chom{b,c})$ is
  precisely to say that~\cref{eq:internal-hom-hexagon} commutes on
  precomposition by $(\bar f, \bar \phi)$. But by naturality of
  $\sigma$ and $\tau$ and definition of transpose, this is precisely
  to say that the following hexagon commutes:
  \begin{equation*}
    \begin{tikzcd}[column sep = 1em]
      &[-7em] (A \bxt B, (a \bxt \hid_{B}) \hcomp (\hid_{A} \bxt b)) \ar[r, "{(f, \phi_1 \hcomp \phi_2)}"] &[2em] (C, c \circ c)  \ar[dr, "{(1, \mu)}"] &[-1em] \\
      (A \bxt B, a \bxt b) \ar[ur, "{(1,\sigma)}"] \ar[dr, "{(1,\tau)}"'] & & & (C, c) \mathrlap{ ,} \\
      & (A \bxt B, (\hid_{A} \bxt b) \hcomp (a \bxt
      \hid_{B})) \ar[r, "{(f,\phi_2 \hcomp \phi_1)}"'] & (C, c \circ c) \ar[ur, "{(1,
        \mu)}"'] &
    \end{tikzcd}
  \end{equation*}
  on in other words, that $(f, \phi,\psi)$ is $(A,B)$-commuting.

  We now establish that $(\chom{b,c}_0, \chom{b,c})$ is the internal
  hom of $(B,b)$ and $(C,c)$ in $\CommMultMndC$. Indeed, by~\cref{prop:multimorphisms-closed}, we have bijections between
  sets of (not necessarily commuting) monad multimorphisms of the form:
  \begin{equation*}
    \begin{array}{r@{\,}l@{\ }c@{\ }l}
      (f, \phi_1, \dots, \phi_n, \psi) \colon& (A_1, a_1), \dots, (A_n, a_n), (B,b) &\longrightarrow& (C,c) \\
      \noalign{\vskip 3pt}
      \hline
      \noalign{\vskip 3pt} 
      (\bar f, \bar{\phi}_1, \dots, \bar{\phi}_n) \colon& (A_1, a_1), \dots, (A_n, a_n) &\longrightarrow& \nchom{(B,b),(C,c)}
    \end{array}
  \end{equation*}
  By~\cref{lem:commutative-maps-into-nchom}, a multimorphism as to the
  top is $(A_i, A_j)$-commuting for all $i<j$ just when the
  corresponding multimorphism as below is commuting. By the result
  just established, a multimorphism as to the top is $(A_i,
  B)$-commuting for all $i$ just when the corresponding multimorphism
  as below factors through $\chom{(B,b), (C,c)}$. Thus, the above
  bijection restricts to a
  bijection~\eqref{eq:commuting-internal-hom-bij} as desired. 
\end{proof}
\begin{rex}\label{thm:running-15}
  We saw in \cref{thm:running-14} that, when $\ca{V}$ is symmetric
  monoidal closed with products, the non-commuting hom
  $(\nchom{b,c}_0, \nchom{b,c})$ between monads $\ca B = (B,b)$ and
  $\ca C = (C,c)$ is the $\ca{V}$-category $\nchom{\ca{B}, \ca{C}}$
  with as objects, the $\ca{V}$-functors $\ca{B} \rightarrow \ca{C}$,
  and with homs given by the $\ca{V}$-objects of ``unnatural
  $\ca{V}$-transformations'', \ie
  $\nchom{\ca{B}, \ca{C}}(F, G) = \textstyle\prod_{x \in B} \ca{C}(Fx,
  Gx)$.

  We now describe the corresponding commuting hom
  $\chom{\ca{B}, \ca{C}}$. This is defined from the non-commuting one
  by the equaliser~\eqref{eq:commuting-hom-equaliser}. The objects of
  this equaliser are the same, namely, the $\ca{V}$-functors
  $\ca{B} \rightarrow \ca{C}$. As for the homs, these are given
  componentwise by the equaliser
  \begin{equation*}
    \chom{\ca B, \ca C}(F,G) \xrightarrow{} \textstyle\prod_{x \in B}\ca{C}(Fx,Gx) \rightrightarrows \prod_{y,y' \in B} [\ca{B}(y,y'),\ca{C}(Fy, Gy')]
  \end{equation*}
  in $\ca{V}$, where the parallel maps seen here have respective
  components at $(y,y') \in B \times B$ given by
  \begin{gather*}
    \textstyle\prod_{x \in B} \ca{C}(Fx,Gx) \xrightarrow{\pi_{y}} \ca{C}(Fy,Gy) \xrightarrow{\ \bar \circ\ } [\ca{C}(Gy,Gy'), \ca{C}(Fy,Gy')] \xrightarrow{[G_{yy'}, 1]} [\ca{B}(y,y'), \ca{C}(Fy,Gy')]\\
    \textstyle\prod_{x \in B} \ca{C}(Fx,Gx) \xrightarrow{\pi_{y'}} \ca{C}(Fy',Gy') \xrightarrow{\ \bar \circ\ } [\ca{C}(Fy',Gy'), \ca{C}(Fy,Fy')] \xrightarrow{[1, F_{yy'}]} [\ca{B}(y,y'), \ca{C}(Fy,Gy')]\rlap{ .}
  \end{gather*}
  This equaliser is precisely the $\ca{V}$-enriched end
  $\int_{x \in \ca{B}}\ca{C}(Fx,Gx)$ giving the $\ca{V}$-object of
  $\ca{V}$-natural transformations from $F$ to
  $G$~\cite[\S
  2.2]{Kelly}, so that $\chom{\ca{B},\ca{C}}$ is
    the usual functor $\ca{V}$-category.
\end{rex}

Copying the proof of~\cref{thm:nctensor-representable} mutatis
mutandis we conclude that:
\begin{prop} \label{prop:CMultMndCrep}
The symmetric multicategory $\CommMultMndC$ of monads and commuting monad
multimorphisms is representable.
\end{prop}

Now arguing as before and in analogy to \cref{thm:nctensor-smc}, we obtain the main result of this section. See \cref{tab:assumptions} for the list of assumptions.

\begin{thm} \label{thm:ctensor-monads} \label{thm:mnd-c-has-comm-tensor}
  Let $(\dcC, \bxt, I)$ be a symmetric normal oplax monoidal closed
  double category. Under
  \cref{hyp:fibrancy,hyp:coproducts-in-mnd,hyp:equalisers,hyp:tabulators,hyp:free-monad,hyp:coequalisers,hyp:equalisers,hyp:equalisers-in-horizontal}, 
  the category $\MndC$ of monads and monad morphisms admits the
  structure of a symmetric monoidal closed category, with tensor
  product $\ctensor$ given by the commuting tensor product of monads,
  and unit object given by the discrete monad $\disc{I}$ on the unit
  object $I$ of $\dcC$.
\end{thm}

\section{The commuting tensor product of bimodules}
\label{sec:bimodules}

Our aim in this section is to extend the commuting tensor
product of monads and monad morphisms constructed in \cref{sec:commuting} to
monad bimodules (\cf \cref{defi:bimodules}).
In our applications in \cref{sec:applications}, this will yield the notion of commuting
tensor product of profunctors (between categories) and of symmetric multiprofunctors
(between symmetric multicategories). 

We continue to work with a symmetric normal oplax monoidal  double
category $(\dcC, \bxt, I)$ whose monoidal structure is closed, and for
which $\dcC$ satisfies
\cref{hyp:fibrancy,hyp:coproducts-in-mnd,hyp:equalisers,hyp:tabulators,hyp:free-monad,hyp:coequalisers,hyp:equalisers-in-horizontal}.
As we have seen (\cref{thm:ctensor-monads}), these conditions are sufficient for there to be a
``commuting'' monoidal structure on the category of monads and monad
morphisms in $\dc C$. In order to extend this to bimodules, we make the following additional
assumption. 

\begin{hyp} \label{hyp:refl-coequalisers}
The double category $\dcC$ has stable local reflexive coequalisers, as in \cref{defi:parallel_local}.
\end{hyp}
Thanks to this, we can apply \cref{thm:bim-c-fibrant} in order to
assemble the monads, monad morphisms and monad bimodules in $\dcC$
into a double category written $\BimC$. As we are assuming $\dcC$ is
fibrant (\cref{hyp:fibrancy}), the double category $\BimC$ is fibrant
by \cref{thm:bim-c-fibrant}.

We recall also some notation from \cref{sec:bimodules-defn}. Given a
bimodule $p \co (A, a) \tickar (B,b)$, its left and right
actions are notated as $\lambda \colon b \circ p
\Rightarrow p$ and $\rho \colon p \circ a \Rightarrow p$, but, as in
\cref{thm:action-notation}, we will often combine these into the
action $2$-morphism
$\alpha \colon b \circ p \circ a \Rightarrow p$ given by $\rho
(\lambda 1) = \lambda(1 \rho)$.
Finally, given another bimodule $q \co (B,b) \tickar (C,c)$,
we write $q \bimcomp p \co (A,a) \tickar (C,c)$ for their composite, defined via the
reflexive coequaliser in~\eqref{equ:bim-comp}. 

\subsection{The multicategory of bimodules and commuting bimodule multimorphisms}
We now embark on constructing the ``commuting'' monoidal structure on
$\BimC$. For the vertical category $\BimC_0$, this is a mere rephrasing of 
\cref{thm:mnd-c-has-comm-tensor}.

\begin{prop}   \label{thm:ctensor-monoidal-vertical} \leavevmode
 The vertical category of $\BimC$ admits a symmetric monoidal closed structure given on objects by the commuting tensor product of monads.
\end{prop}

\begin{proof} The vertical category $\BimC_0$ is $\Mnd(\dc{C})$, the category
of monads and monad morphisms in $\dcC$, so the claim 
follows from  \cref{thm:mnd-c-has-comm-tensor}.
\end{proof}

In order to extend the commuting tensor product to the category
$\BimC_1$, \ie the category of bimodules and bimodule morphisms, we
follow a similar pattern to that followed in
\cref{sec:non-commuting,sec:commuting}: we show that the category
$\BimC_1$ is the underlying category of a representable symmetric multicategory and
therefore admits a symmetric monoidal structure. Besides giving us the symmetric
monoidal structure, this also allows us to characterise the commuting tensor product of bimodules
by a universal property. For this, we now introduce the notion of
commuting multimorphism of bimodules (cf.~\cref{defi:commuting-monad-multimorphism}).

\begin{defi} \label{def:bimodule-multimorphism}
Let $p_1 \co (A_1, a_1) \tickar (B_1, b_1), \ldots, p_n \co (A_n, a_n) \tickar (B_n, b_n)$
and $q \co (C, c) \tickar (D, d)$ be bimodules. A \emph{commuting bimodule multimorphism}
\[
(f, \vec{\phi}, g, \vec{\psi}, \theta) \co (p_1, \ldots, p_n) \rightarrow q
\]
consists of:
\begin{enumerate}[(i)] 
\item \label{item:comm-bim-1} a commuting monad multimorphism $(f, \vec{\phi}) \co (A_1, a_1), \ldots, (A_n, a_n) \to (C,c)$;
\item \label{item:comm-bim-2} a commuting monad multimorphism $(g, \vec{\psi}) \co (B_1, b_1),
  \ldots, (B_n, b_n) \to (D,d)$; and
\item \label{item:comm-bim-3} a 2-morphism
\[
\begin{tikzcd}[column sep = huge]
A_1 \bxt \ldots \bxt A_n \ar[r,tick, "p_1 \bxt \ldots \bxt p_n"] \ar[d, "f"'] 
\ar[dr, phantom, pos=0.4, "\Two \theta"]
& B_1 \bxt \ldots \bxt B_n \ar[d, "g"] \\
C \ar[r,tick, "q"'] & D\rlap{ ,}
\end{tikzcd}
\]
\end{enumerate}
such that for each $1 \leqslant i \leqslant n$, we have the following
condition expressing compatibility with the actions of the bimodules:

\begin{equation*}
  \begin{gathered}
    \begin{tikzcd}[row sep = 1.2cm, column sep = 7em]
      A_{[1,n]}
      \ar[r,tick, "p_{[1, i\mi1]} \bxt (b_i \circ p_i \circ a_i) \bxt p_{[i+1, n]}" {yshift=3pt}] 
      \ar[d,equal] \ar[dr, phantom, "\Two 1 \bxt \alpha_i \bxt 1"]  & 
      B_{[1,n]}
      \ar[d, equal] \\
      A_{[1,n]} 
      \ar[r,tick, "p_{[1,n]}"] \ar[d, "f"'] 
      \ar[dr, phantom, "\Two \theta"]  & 
      B_{[1,n]}
      \ar[d, "g"]  \\
      C 
      \ar[r,tick, "q"'] & 
      D 
    \end{tikzcd} 
  \end{gathered} = 
  \begin{gathered}
    \begin{tikzcd}[row sep = 0.8cm,column sep = 2em]
      A_{[1,n]}
      \ar[rrr,tick, "p_{[1, i\mi1]} \bxt (b_i \circ p_i \circ a_i) \bxt p_{[i+1, n]}"] 
      \ar[d,equal] 
      \ar[drrr, phantom, "\cong"]  &[5.7em] &  &[5.7em] 
      B_1 \bxt B_2 
      \ar[d,equal] \\
      A_{[1,n]}
      \ar[rrr,tick, "\tparen{\hid \circ p \circ \hid}_{[1, i\mi1]} \bxt (b_i \circ p_i \circ a_i) \bxt \tparen{\hid \circ p \circ \hid}_{[i+1, n]}"] 
      \ar[d,equal] 
      \ar[drrr, phantom, "\Two \xi"]  & &  &
      B_1 \bxt B_2 
      \ar[d,equal] \\
      A_{[1,n]}
      \ar[r,tick, "\tparen{\hid_B}_{[1,i \mi 1]} \bxt a_i \bxt \tparen{\hid_B}_{[i + 1, n]}" {yshift=2pt}] 
      \ar[d, "f"']   
      \ar[dr, phantom, "\Two \phi_i"] & 
      A_{[1,n]}
      \ar[r,tick, "p_{[1,n]}"] 
      \ar[d,"f"] 
      \ar[dr, phantom, "\Two \theta"] &
      B_{[1,n]}
      \ar[r,tick, "\tparen{\hid_B}_{[1,i \mi 1]} \bxt b_i \bxt \tparen{\hid_B}_{[i + 1, n]}" {yshift=2pt}]  
      \ar[d,"g"]
      \ar[dr, phantom, "\Two \psi_i"]&  
      B_{[1,n]}
      \ar[d, "g"] \\
      C 
      \ar[r,tick, "c"'] 
      \ar[d,equal] 
      \ar[drrr, phantom, "\Two \alpha"]   & 
      C 
      \ar[r,tick, "q"'] & 
      D 
      \ar[r,tick, "d"'] & 
      D  
      \ar[d,equal]  \\
      C 
      \ar[rrr,tick, "q"'] &  &  & 
      D \mathrlap{.}
    \end{tikzcd}\!\!\!\!
  \end{gathered}
\end{equation*}
Here, the interchanger $2$-morphisms labelled as $\xi$ are built inductively from
the binary interchangers of \cref{def:oplaxmonoidal} in the evident manner.
\end{defi}

\begin{prop} \label{prop:CMultBimC}
 There is a symmetric multicategory $\CommMultBimC$ with bimodules as objects and 
 commuting bimodule multimorphisms as maps.
\end{prop}

\begin{proof}
  Consider conditions \labelcref{item:comm-bim-1,item:comm-bim-2,item:comm-bim-3} for a
  commuting bimodule multimorphism. 
  \cref{item:comm-bim-1,item:comm-bim-2} compose as they do in
  $\CommMultMnd$, while \lcnamecref{item:comm-bim-3}~\labelcref{item:comm-bim-3} composes as
  in (the underlying multicategory of) the monoidal category $\dcC_1$.
  It is routine to verify the bimodule compatibility conditions of a
  multimorphism are stable under this composition.
\end{proof}
Note that there is no obstruction to defining a corresponding
multicategory $\mathsf{MultBim}(\mathbb{C})$ of
not-necessarily-commuting bimodule multimorphisms. However, our
primary interest is in the commuting case and---by contrast with the
situation of the preceding two sections---the argument for the case of
commuting bimodule multimorphisms can be given without developing the
non-commuting case first. As such, we leave further consideration of
$\mathsf{MultBim}(\mathbb{C})$ to the reader.

\subsection{Pre-representability of $\CommMultBimC$}\label{sub:prerepmultibim}

Our aim is to show that the multicategory $\CommMultBimC$ of
bimodules and commuting bimodule multimorphisms in $\dcC$ is
representable. Our strategy for doing so will diverge from that of the
previous sections, but will at least begin in the same way: by establishing
pre-representability.

We first prove a lemma that rephrases the conditions to be a binary
module multimorphism as in \cref{def:bimodule-multimorphism}. Recall
the universal commuting binary
multimorphism~\eqref{eq:universal-commuting-binary} associated to
monads $(A_1, a_1)$ and $(A_2, a_2)$, which has the form
\begin{equation*}
  (1_{A_1 \bxt A_2}, \eta_1, \eta_2) \colon (A_1, a_1), (A_2, a_2) \rightarrow (A_1 \bxt A_2, a_1 \ctensor a_2)\rlap{ .}
\end{equation*}
By construction, the two sides of the
hexagon~\eqref{eq:binary-commutativity-hexagon} expressing
commutativity of this multimorphism are equal; let us write their
common composite in $\End(\dcC)$ as:
\begin{equation}
\begin{tikzcd}
\label{equ:bxt-to-ctensor}
(A_1 \bxt A_2, a_1 \boxtimes a_2) \ar[r, "{(1,\pi)}"] & (A_1 \bxt A_2, a_1 \ctensor a_2)\rlap{ .}
\end{tikzcd}
\end{equation}

\begin{lem} \label{thm:compact-bim-multimorphism}
Let $p_1 \co (A_1, a_1) \tickar (B_1, b_1)$,  $p_2 \co (A_2, a_2) \tickar (B_2, b_2)$
and $q \co (C, c) \tickar (D, d)$ be bimodules. Let $(f, \phi_1,  \phi_2) \co (A_1, a_1),  (A_2, a_2) \to (C,c)$
and $(g, \psi_1,  \psi_2) \co (B_1, b_1), (B_2, b_2) \to (D,d)$ be  commuting monad multimorphisms.
A 2-morphism 
\[
\begin{tikzcd}[column sep = large]
A_1 \bxt A_2 \ar[r,tick, "p_1 \bxt p_2"] \ar[d, "f"'] 
\ar[dr, phantom, "\Two \theta"]
& B_1 \bxt B_2 \ar[d, "g"] \\
C \ar[r,tick, "q"'] & D 
\end{tikzcd}
\]
determines  a commuting bimodule multimorphism $(f, \phi_1, \phi_2, g,
\psi_1, \psi_2, \theta) \co p_1, p_2 \rightarrow q$ if and only if
\begin{equation} 
\label{equ:compact-bim-multimorphims}
\begin{tikzcd}[column sep = 2.7cm, row sep = 1.6cm]
A_1  \bxt A_2
	\ar[r,tick, "(b_1 \circ p_1 \circ a_1) \bxt  (b_2 \circ p_2 \circ a_2)"] 
	\ar[d,equal] 
	\ar[dr, phantom, "\Two \alpha_1 \bxt  \alpha_2"] & 
B_1 \bxt  B_2 \ar[d, equal] \\
A_1 \bxt  A_2  \ar[d, "f"'] 
	\ar[r,tick, "p_1 \bxt p_2"'] 
	\ar[dr, phantom, "\Two \theta"] & 
B_1 \bxt  B_2 
	\ar[d, "g"]  \\
C 
	\ar[r,tick, "q"'] & 
D  
\end{tikzcd}  = 
\begin{tikzcd}
A_1 \bxt A_2  
	\ar[rrr,tick, "(b_1 \circ p_1 \circ a_1) \bxt  (b_2 \circ p_2 \circ a_2)"] 
	\ar[d,equal] 
	\ar[drrr, phantom, "\Two \xi"] &
 & & 
 B_1 \bxt  B_2  \ar[d, equal] \\
  A_1 \bxt  A_2  \ar[r,tick, "a_1 \bxt a_2"]  \ar[d, equal] \ar[dr, phantom, "\Two \pi"] & 
  A_1 \bxt  A_2  \ar[r,tick, "p_1 \bxt p_2"] \ar[d, equal] & 
  B_1 \bxt  B_2 \ar[r,tick, "b_1 \bxt b_2"] \ar[d, equal]  \ar[dr, phantom, "\Two \pi"] &
  B_1 \bxt  B_2   \ar[d, equal] \\
  A_1 \bxt  A_2  \ar[r,tick, "a_1 \ctensor a_2"] \ar[d, "f"']  \ar[dr, phantom, "\Two \bar{\phi}"]  & 
  A_1 \bxt  A_2  \ar[r,tick, "p_1 \bxt p_2"] \ar[d, "f"'] \ar[dr, phantom, "\Two \theta"]   & 
  B_1 \bxt  B_2 \ar[r,tick, "b_1 \ctensor b_2"]  \ar[d, "g"]  \ar[dr, phantom, "\Two \bar{\psi}"]  &
  B_1 \bxt  B_2  \ar[d, "g"]  \\
  C \ar[r,tick, "c"']   \ar[d, equal] \ar[drrr, phantom, "\Two \alpha"]  &
  C \ar[r,tick, "q"'] &
  D \ar[r,tick, "d"'] &
  D \ar[d, equal] \\
  C \ar[rrr,tick, "q"'] & & & D \mathrlap{,}
  \end{tikzcd}
\end{equation} 
where $(f, \bar{\phi})$ and $(g, \bar{\psi})$ are the (unary) monad
maps induced from $(f, \phi_1, \phi_2)$ and $(g, \psi_1, \psi_2)$ by
the universality of the commuting tensor product of monads.
\end{lem}

\begin{proof} The equation in the statement implies the conditions in
  \cref{def:bimodule-multimorphism} for $i = 1,2$ on precomposing both
  sides of the equation in \eqref{equ:compact-bim-multimorphims} with
  the unit map $\eta \circ p_i \circ \eta \colon p_i \rightarrow b_i
  \circ p_i \circ a_i$. In the converse direction, we start from the
  right hand side of~\eqref{equ:compact-bim-multimorphims}. The
  composites $\bar \phi \circ \pi$, resp.~$\bar \psi \circ \pi$ can be
  rewritten as the bottom, resp., upper side of the hexagon expressing commutativity of
  $\phi$, resp.~$\psi$, as follows:
  \begin{equation}
    \label{eq:bimodule-binary-reformulation}
    \begin{tikzcd}
      A_1 \bxt A_2  
      \ar[rrrrr,tick, "(b_1 \circ p_1 \circ a_1) \bxt  (b_2 \circ p_2 \circ a_2)"] 
      \ar[d,equal] 
      \ar[drrrrr, phantom, "\Two \xi"] &
      & & & &
      B_1 \bxt  B_2  \ar[d, equal] \\
      A_1 \bxt  A_2  \ar[rr,tick, "a_1 \bxt a_2"]  \ar[d, equal] \ar[drr, phantom, "\Two \tau"] & & 
      A_1 \bxt  A_2  \ar[r,tick, "p_1 \bxt p_2"] \ar[d, equal] & 
      B_1 \bxt  B_2 \ar[rr,tick, "b_1 \bxt b_2"] \ar[d, equal]  \ar[drr, phantom, "\Two \sigma"] & &
      B_1 \bxt  B_2   \ar[d, equal] \\
      A_1 \bxt  A_2  \ar[r,tick, "a_1 \bxt \hid"] \ar[d, "f"']  \ar[dr, phantom, "\Two \phi_1"]  & 
      A_1 \bxt  A_2  \ar[r,tick, "\hid \bxt a_2"] \ar[d, "f"']  \ar[dr, phantom, "\Two \phi_2"]  & 
      A_1 \bxt  A_2  \ar[r,tick, "p_1 \bxt p_2"] \ar[d, "f"'] \ar[dr, phantom, "\Two \theta"]   & 
      B_1 \bxt  B_2 \ar[r,tick, "\hid \bxt b_2"]  \ar[d, "g"]  \ar[dr, phantom, "\Two \psi_2"]  &
      B_1 \bxt  B_2 \ar[r,tick, "b_1 \bxt \hid"]  \ar[d, "g"]  \ar[dr, phantom, "\Two \psi_1"]  &
      B_1 \bxt  B_2  \ar[d, "g"]  \\
      C \ar[r,tick, "c"']   \ar[d, equal] \ar[drr, phantom, "\Two \mu"]  &
      C \ar[r,tick, "c"']      &
      C \ar[r,tick, "q"'] \ar[d,equal] &
      D \ar[r,tick, "d"']    \ar[d, equal] \ar[drr, phantom, "\Two \mu"] &
      D \ar[r,tick, "d"']&
      D \ar[d, equal] \\
      C \ar[rr,tick, "c"']   \ar[d, equal] \ar[drrrrr, phantom, "\Two \alpha"]  &&
      C \ar[r,tick, "q"'] &
      D \ar[rr,tick, "d"'] &&
      D \ar[d, equal] \\
      C \ar[rrrrr,tick, "q"'] & & & & & D \mathrlap{.}
    \end{tikzcd}
  \end{equation}
  The top two rows may now be rewritten as:
  \begin{equation*}
    \begin{tikzcd}
      A_1 \bxt A_2  
      \ar[rrrrr,tick, "(b_1 \circ p_1 \circ a_1) \bxt  (b_2 \circ p_2 \circ a_2)"] 
      \ar[d,equal] 
      \ar[drrrrr, phantom, "\cong"] &
      & & & &
      B_1 \bxt  B_2  \ar[d, equal] \\
      A_1 \bxt A_2  
      \ar[rrrrr,tick, "(b_1 \circ p_1 \circ a_1) \bxt  (\hid \circ (b_2 \circ p_2 \circ a_2) \circ \hid)"] 
      \ar[d,equal] 
      \ar[drrrrr, phantom, "\Two \xi"] &
      & & & &
      B_1 \bxt  B_2  \ar[d, equal] \\
      A_1 \bxt  A_2  \ar[r,tick, "a_1 \bxt \hid"] \ar[d, equal]  & 
      A_1 \bxt  A_2  \ar[rrr,tick, "p_1 \bxt (b_2 \circ p_2 \circ a_2)"] \ar[d, equal]  \ar[drrr, phantom, "\cong"]  & & & 
      B_1 \bxt  B_2  \ar[r,tick, "b_1 \bxt \hid"] \ar[d, equal] & 
      B_1 \bxt  B_2  \ar[d, equal]  \\
      A_1 \bxt  A_2  \ar[r,tick, "a_1 \bxt \hid"] \ar[d, equal]  & 
      A_1 \bxt  A_2  \ar[rrr,tick, "(\hid \circ p_1 \circ \hid) \bxt (b_2 \circ p_2 \circ a_2)"] \ar[d, equal]  \ar[drrr, phantom, "\Two \xi"]  & & & 
      B_1 \bxt  B_2  \ar[r,tick, "b_1 \bxt \hid"] \ar[d, equal] & 
      B_1 \bxt  B_2  \ar[d, equal]  \\
      A_1 \bxt  A_2  \ar[r,tick, "a_1 \bxt \hid"] & 
      A_1 \bxt  A_2  \ar[r,tick, "\hid \bxt a_2"] & 
      A_1 \bxt  A_2  \ar[r,tick, "p_1 \bxt p_2"] & 
      B_1 \bxt  B_2 \ar[r,tick, "\hid \bxt b_2"]  &
      B_1 \bxt  B_2 \ar[r,tick, "b_1 \bxt \hid"]  &
      B_1 \bxt  B_2 
    \end{tikzcd}
  \end{equation*}
  At the same time, we can rewrite the bottom two rows
  of~\eqref{eq:bimodule-binary-reformulation} via the bimodule axioms
  to become two instances of $\alpha$. On doing so, we recognise the
  middle of the resulting diagram as being the right-hand side of the
  bimodule multimorphism axiom for $i = 2$. Rewriting by this, we
  obtain the diagram:
  \begin{equation*}
    \begin{tikzcd}
      A_1 \bxt A_2  
      \ar[rrr,tick, "(b_1 \circ p_1 \circ a_1) \bxt  (b_2 \circ p_2 \circ a_2)"] 
      \ar[d,equal] 
      \ar[drrr, phantom, "\cong"] &[-0.5em]
      &[1em] &[-0.5em]
      B_1 \bxt  B_2  \ar[d, equal] \\
      A_1 \bxt A_2  
      \ar[rrr,tick, "(b_1 \circ p_1 \circ a_1) \bxt  (\hid \circ (b_2 \circ p_2 \circ a_2) \circ \hid)"] 
      \ar[d,equal] 
      \ar[drrr, phantom, "\Two \xi"] 
      & & &
      B_1 \bxt  B_2  \ar[d, equal] \\
      A_1 \bxt  A_2  \ar[r,tick, "a_1 \bxt \hid"] \ar[d, equal]  & 
      A_1 \bxt  A_2  \ar[r,tick, "p_1 \bxt (b_2 \circ p_2 \circ a_2)" {yshift=2pt}] \ar[d, equal]  \ar[dr, phantom, "\Two 1 \bxt \alpha_2"]  & 
      B_1 \bxt  B_2  \ar[r,tick, "b_1 \bxt \hid"] \ar[d, equal] & 
      B_1 \bxt  B_2  \ar[d, equal]  \\
      A_1 \bxt  A_2  \ar[r,tick, "a_1 \bxt \hid"'] \ar[d, "f"']  \ar[dr,phantom,"\Two \phi_1"]& 
      A_1 \bxt  A_2  \ar[r,tick, "p_1 \bxt p_2"'] \ar[d, "f"']  \ar[dr, phantom, "\Two \theta"]  & 
      B_1 \bxt  B_2  \ar[r,tick, "b_1 \bxt \hid"'] \ar[d, "g"'] \ar[dr,phantom,"\Two \psi_1"] & 
      B_1 \bxt  B_2  \ar[d, "g"]  \\
      C  \ar[r,tick, "c"'] \ar[d, equal] \ar[drrr, phantom, "\Two \alpha"] & 
      C  \ar[r,tick, "q"']     & 
      D  \ar[r,tick, "d"']  & 
      D  \ar[d, "g"]  \\
      C \ar[rrr,tick, "q"'] & & & D \mathrlap{.}
    \end{tikzcd}
  \end{equation*}
  Now using naturality of $\xi$ and the identity constraints of $\dc
  C$, we can pull $1 \bxt \alpha_2$ to the top, and rewrite again
  using the bimodule multimorphism axiom for $i = 1$, so obtaining the
  left-hand side of~\eqref{equ:compact-bim-multimorphims} as desired.
\end{proof} 

\begin{prop}\label{prop:commuting-bimodules-prerepresentable}
The symmetric multicategory $\CommMultBimC$ of bimodules and commuting bimodule morphisms
is pre-representable.
\end{prop}
\begin{proof}
  The nullary case is immediate: the representing object is the
  identity bimodule $I$ on the discrete monad $\nabla{I}$. For the
  binary case, consider the upper parts of the two sides of the
  equality~\eqref{equ:compact-bim-multimorphims}:
  \begin{equation*}
    \begin{tikzcd}[column sep = 2.7cm, row sep = 1.6cm]
      A_1  \bxt A_2
      \ar[r,tick, "(b_1 \circ p_1 \circ a_1) \bxt  (b_2 \circ p_2 \circ a_2)"] 
      \ar[d,equal] 
      \ar[dr, phantom, "\Two \alpha_1 \bxt  \alpha_2"] & 
      B_1 \bxt  B_2 \ar[d, equal] \\
      A_1 \bxt  A_2  
      \ar[r,tick, "p_1 \bxt p_2"'] &
      B_1 \bxt  B_2 
    \end{tikzcd}  \quad \quad 
    \begin{tikzcd}
      A_1 \bxt A_2  
      \ar[rrr,tick, "(b_1 \circ p_1 \circ a_1) \bxt  (b_2 \circ p_2 \circ a_2)"] 
      \ar[d,equal] 
      \ar[drrr, phantom, "\Two \xi"] &
      & & 
      B_1 \bxt  B_2  \ar[d, equal] \\
      A_1 \bxt  A_2  \ar[r,tick, "a_1 \bxt a_2"]  \ar[d, equal] \ar[dr, phantom, "\Two \pi"] & 
      A_1 \bxt  A_2  \ar[r,tick, "p_1 \bxt p_2"] \ar[d, equal] & 
      B_1 \bxt  B_2 \ar[r,tick, "b_1 \bxt b_2"] \ar[d, equal]  \ar[dr, phantom, "\Two \pi"] &
      B_1 \bxt  B_2   \ar[d, equal] \\
      A_1 \bxt  A_2  \ar[r,tick, "a_1 \ctensor a_2"]   & 
      A_1 \bxt  A_2  \ar[r,tick, "p_1 \bxt p_2"]    & 
      B_1 \bxt  B_2 \ar[r,tick, "b_1 \ctensor b_2"]   &
      B_1 \bxt  B_2  
    \end{tikzcd}
  \end{equation*}
  These are, respectively, maps
  \begin{align*}
    u \colon F_{a_1, b_1}(p_1) \bxt F_{a_2,
      b_2}(p_2) & \rightarrow p_1 \bxt p_2 \\
    \text{and}\ \ v \colon F_{a_1, b_1}(p_1) \bxt F_{a_2,
      b_2}(p_2) & \rightarrow F_{a_1 \ctensor a_2, b_1 \ctensor b_2}(p_1
    \bxt p_2)
  \end{align*}
  in $\dcC[A_1 \bxt A_2, B_1
  \bxt B_2]$, and so we obtain a parallel pair of bimodule maps as to
  the left in:
  \begin{equation} 
    \label{equ:key-for-bim-representable}
    \begin{tikzcd}[column sep = large]
      F_{a_1 \ctensor a_2, b_1 \ctensor b_2} \big( F_{a_1, b_1}(p_1) \bxt F_{a_2, b_2}(p_2) \big) 
      \ar[r, shift right, "\bar v"'] 
      \ar[r, shift left, "F_{a_1 \ctensor a_2, b_1 \ctensor b_2}(u)"] &[2em] 
      F_{a_1 \ctensor a_2, b_1 \ctensor b_2} (p_1 \bxt p_2) 
      \ar[r, two heads]  &
      p_1 \ctensor p_2\rlap{ .}
    \end{tikzcd}
  \end{equation}
  Here, the upper arrow is the free bimodule morphism on $u$, while
  $\bar v$ is the unique extension of $v$ to a bimodule morphism. It
  is easy to see that this
  parallel pair is reflexive, with common splitting given by $F_{a_1
    \ctensor a_2, b_1 \ctensor b_2}(\eta, \eta)$. Thus, since
  $\BimC[(A_1 \bxt A_2, a_1 \ctensor a_2), (B_1 \bxt B_2, b_1 \ctensor
  b_2)]$ has reflexive coequalisers, we can form the coequaliser of
  this pair, as to the right.

  By~\cref{thm:fibrancy-coequalizers}(i), this coequaliser is also a
  coequaliser in $\BimC_1$. Using this and the enhanced freeness
  property of free bimodules described
  in~\cref{thm:extra-fibred-free-bimod}, we see that a map out of
  $p_1 \ctensor p_2$ in $\BimC_1$ into an object
  $q \colon (C, c) \tor (D,d)$ is equally well given by:
  \begin{enumerate}[(i)]
  \item a map $(A_1, a_1) \ctensor (A_2, a_2) \rightarrow (C,c)$;
  \item a map $(B_1, b_1) \ctensor (B_2, b_2) \rightarrow (D,d)$; and
  \item a $2$-morphism
    \[
      \begin{tikzcd}[column sep = huge]
        A_1 \bxt \ldots \bxt A_n \ar[r,tick, "p_1 \bxt \ldots \bxt p_n"] \ar[d, "f"'] 
        \ar[dr, phantom, pos=0.4, "\Two \theta"]
        & B_1 \bxt \ldots \bxt B_n \ar[d, "g"] \\
        C \ar[r,tick, "q"'] & D\rlap{ ,}
      \end{tikzcd}
    \]
  \end{enumerate}
  such that the induced map
  $F_{a_1 \ctensor a_2, b_1 \ctensor b_2} (p_1 \bxt p_2)
  \xrightarrow{\bar \theta} q$ equalises the parallel pair
  in~\eqref{equ:key-for-bim-representable}. Clearly, the data
  (i)--(iii) above can be equated with the data (i)--(iii)
  of~\cref{def:bimodule-multimorphism}; and a short argument with
  adjointness shows that the equalising property is equivalent to the
  equality~\eqref{equ:compact-bim-multimorphims}, which,
  by~\cref{thm:compact-bim-multimorphism} is equally to say that these
  data describe a commuting bimodule bimorphism.
\end{proof}

\subsection{Representability of $\CommMultBimC$}\label{sub:repmultibim}

At this point, our development proceeds in a different direction to
previously. Rather than establish closedness of the multicategory
$\CommMultBimC$, and conclude
using~\cref{prop:closed-preuniversal-universal}, we will instead show
existence of \emph{all} pre-universal multimorphisms (not just nullary
and binary ones), and then show these are closed under composition;
this suffices for representability
by~\cite[Proposition~8.5]{RepresentableMulticats}.

We begin with the $n$-ary analogue of~\cref{thm:compact-bim-multimorphism}.

\begin{lem} \label{thm:compact-bim-multimorphism-general} Let
  $p_i \co (A_i, a_i) \tickar (B_i, b_i)$ (for
  $1 \leqslant i \leqslant n$) and $q \co (C, c) \tickar (D, d)$ be
  bimodules. Let
  $(f, \vec \phi) \co (A_1, a_1), \dots, (A_n, a_n) \to (C,c)$ and
  $(g, \vec \psi) \co (B_1, b_1), \dots, (B_n, b_n) \to (D,d)$ be
  commuting monad multimorphisms. A 2-morphism
\[
\begin{tikzcd}[column sep = large]
A_{[1,n]} \ar[r,tick, "p_{[1,n]}"] \ar[d, "f"'] 
\ar[dr, phantom, "\Two \theta"]
& B_{[1,n]} \ar[d, "g"] \\
C \ar[r,tick, "q"'] & D 
\end{tikzcd}
\]
determines  a commuting bimodule multimorphism $(f, \vec \phi, g,
\vec \psi, \theta) \co p_1, \dots, p_n \rightarrow q$ if and only if
\begin{equation} 
\label{equ:compact-bim-multimorphims-general}
\begin{tikzcd}[column sep = 2.7cm, row sep = 1.6cm]
A_{[1,n]}
	\ar[r,tick, "\tparen{b \circ p \circ a}_{[1,n]}"] 
	\ar[d,equal] 
	\ar[dr, phantom, "\Two \alpha_{[1,n]}"] & 
B_{[1,n]} \ar[d, equal] \\
A_{[1,n]}  \ar[d, "f"'] 
	\ar[r,tick, "p_{[1,n]}"'] 
	\ar[dr, phantom, "\Two \theta"] & 
B_{[1,n]} 
	\ar[d, "g"]  \\
C 
	\ar[r,tick, "q"'] & 
D  
\end{tikzcd}  = 
\begin{tikzcd}
A_{[1,n]}
	\ar[rrr,tick, "\tparen{b \circ p \circ a}_{[1,n]}"] 
	\ar[d,equal] 
	\ar[drrr, phantom, "\Two \xi"] &
 & & 
 B_{[1,n]}  \ar[d, equal] \\
  A_{[1,n]}  \ar[r,tick, "a_{[1,n]}"]  \ar[d, equal] \ar[dr, phantom, "\Two \pi"] & 
  A_{[1,n]}  \ar[r,tick, "p_{[1,n]}"] \ar[d, equal] & 
  B_{[1,n]} \ar[r,tick, "b_{[1,n]}"] \ar[d, equal]  \ar[dr, phantom, "\Two \pi"] &
  B_{[1,n]}   \ar[d, equal] \\
  A_{[1,n]}  \ar[r,tick, "a_1 \ctensor \dots \ctensor a_n"] \ar[d, "f"']  \ar[dr, phantom, "\Two \bar{\phi}"]  & 
  A_{[1,n]}  \ar[r,tick, "p_{[1,n]}"] \ar[d, "f"'] \ar[dr, phantom, "\Two \theta"]   & 
  B_1 \bxt  B_2 \ar[r,tick, "b_1 \ctensor \dots \ctensor b_n"]  \ar[d, "g"]  \ar[dr, phantom, "\Two \bar{\psi}"]  &
  B_{[1,n]}  \ar[d, "g"]  \\
  C \ar[r,tick, "c"']   \ar[d, equal] \ar[drrr, phantom, "\Two \alpha"]  &
  C \ar[r,tick, "q"'] &
  D \ar[r,tick, "d"'] &
  D \ar[d, equal] \\
  C \ar[rrr,tick, "q"'] & & & D \mathrlap{,}
  \end{tikzcd}
\end{equation} 
where $(f, \bar{\phi})$ and $(g, \bar{\psi})$ are the (unary) monad
maps induced from $(f, \phi_1, \phi_2)$ and $(g, \psi_1, \psi_2)$ by
the universality of the commuting tensor product of monads.
\end{lem}
\begin{proof}
  There is nothing to do for $n = 0$. For $n \geqslant 1$, we proceed
  by induction on $n$. The base case is trivial; for the inductive
  step, we follow the argument
  of~\cref{thm:compact-bim-multimorphism}, mutatis mutandis.
\end{proof}

\begin{prop}\label{prop:commuting-bimodules-all-pre-universal}
The symmetric multicategory $\CommMultBimC$ of bimodules and commuting bimodule morphisms
admits all pre-universal multimorphisms.
\end{prop}
\begin{proof}
  Exactly like \cref{prop:commuting-bimodules-prerepresentable},
  using~\cref{thm:compact-bim-multimorphism-general} in place
  of~\cref{thm:compact-bim-multimorphism}.
\end{proof}

The universal property of the $n$-fold commuting tensor product of
bimodules implies that the construction extends to a functor. The
following fact shall be used repeatedly in what follows.

\begin{lem} \label{thm:ctensor-preserves-coequalisers} The commuting tensor product functor
\[
\begin{tikzcd}
\prod_i \BimC[ (A_i, a_i), (B_i, b_i) ] 
	\ar[r, "\scriptstyle{\ctensor}"] & 
\BimC[ (A_1 \bxt \cdots \bxt A_n, a_1 \ctensor \dots \ctensor a_n), (B_1 \bxt \cdots \bxt B_n, b_1 \ctensor\dots \ctensor b_n)]
\end{tikzcd}
\]
preserves reflexive coequalisers.
\end{lem}

\begin{proof} Colimits commute with colimits and all the functors involved in
the definition of the commuting tensor product in~\eqref{equ:key-for-bim-representable} 
preserve reflexive coequalisers.
\end{proof}

To show closure of pre-universal morphisms under composition, we will
need one further result showing that the commuting tensor product of
free bimodules is free. First we construct the universal morphism
witnessing this fact; the construction is our counterpart of
\cite[Theorem~1.22]{DwyerW:BoardmanVtpo}. Here, $U$ is the forgetful
functor of \cref{thm:free-bimodule} with left adjoint $F$ as in
\cref{equ:free-bimodule}, sometimes suppressing the obvious indices
for simplicity.

\begin{lem} \label{thm:thm-1-22-dwyer-hess} Let $F_{a_i, b_i}(x_i) \co
  (A_i, a_i) \tickar (B_i, b_i)$ be the free bimodules on horizontal 
maps 
$x_i \co A_i \tickar B_i$, for each $1 \leqslant i
  \leqslant n$. There is a globular 2-morphism in $\dcC$
\[
\begin{tikzcd}[column sep = 4cm]
A_{[1,n]} \ar[r, tick,"UF_{a_1,b_1}(x_1) \bxt \dots \bxt UF_{a_n, b_n}(x_n)"] \ar[d, equal] \ar[dr, phantom, "\Two \omega"]  & B_1 \bxt B_2  \ar[d, equal] \\
A_{[1,n]} \ar[r, tick,"UF_{\ctensor_{i} a_i, \ctensor_i b_i}(x_1 \bxt \dots \bxt x_n)"'] & B_{[1, n]} \mathrlap{.}
\end{tikzcd}
\]
\end{lem}

\begin{proof} 
We define $\omega$ as the composite 2-morphism:
\[
\begin{tikzcd}[column sep = large]
A_{[1,n]}  \ar[rrr,tick, "(b_1 \hcomp x_1 \hcomp a_1) \bxt \dots \bxt (b_n \hcomp x_n \hcomp a_n) "]  \ar[d, equal] \ar[drrr, phantom,bend left=3, "\Two \xi"] 
	& 
	& 
	& 
B_{[1,n]}   \ar[d, equal]  \\
A_{[1,n]} \ar[r,tick, "a_{[1,n]}" ] \ar[d, equal] \ar[dr, phantom, "\Two \pi"]  & 
A_{[1,n]} \ar[r,tick, "x_{[1,n]}" ] \ar[d, equal] & 
B_{[1,n]} \ar[r,tick, "b_{[1,n]}"] \ar[d, equal] \ar[dr, phantom, "\Two \pi"]  &
B_{[1,n]}  \ar[d, equal] \\
A_{[1,n]} \ar[r,tick, "a_1 \ctensor \dots \ctensor a_n"'] &
A_{[1,n]} \ar[r,tick, "x_{[1,n]}"' ]&
B_{[1,n]} \ar[r,tick, "b_1 \ctensor \cdots \ctensor b_n"'] &
B_{[1,n]}  \mathrlap{,}
\end{tikzcd}
\]
where $\xi$ is formed using the interchange law of the oplax monoidal
double $\dc{C}$ and $\pi$ is the ($n$-ary analogue of) the 2-morphism
in~\cref{equ:bxt-to-ctensor}.
\end{proof}

\begin{prop}\label{prop:free-tensor-bimodules-is-free}
  Let $F_{a_i, b_i}(x_i) \co (A_i, a_i) \tickar (B_i, b_i)$ be the
  free bimodules on horizontal maps $x_i \co A_i \tickar B_i$, for
  each $1 \leqslant i \leqslant n$. The $2$-morphism $\omega$
  of~\cref{thm:thm-1-22-dwyer-hess}, together with the universal
  $n$-ary commuting monad multimorphisms on the domain and codomain
  monads, exhibits $F_{\otimes_i a_i, \otimes_i b_i}(x_1 \bxt \dots \bxt x_n)$ as the
  bimodule commuting tensor product of the $F_{a_i, b_i}(x_i)$'s.
\end{prop}
For example, in the binary case, this result tells us that we have an
isomorphism
\begin{equation}
  \label{equ:ctensor-of-free-bimodules}
  F_{a_1, b_1}(x_1) \ctensor F_{a_2, b_2}(x_2) \cong F_{a_1 \ctensor a_2, b_1 \ctensor b_2}(x_1 \bxt x_2)\rlap{ .}
\end{equation}
\begin{proof}
  It suffices in light of the definition of
  $\ctensor$ via the ($n$-ary analogue of the) 
  coequaliser~\eqref{equ:key-for-bim-representable}, to show that the
  following diagram is itself a coequaliser:
  \begin{equation*} 
    \begin{tikzcd}
      F \big( F^2_{a_1, b_1}(x_1) \bxt \cdots \bxt F^2_{a_n, b_n}(x_n) \big) 
      \ar[r, shift right, "\bar v"'] 
      \ar[r, shift left, "F(u)"] &[2em] 
      F (F_{a_1, b_1}(x_1) \bxt \cdots \bxt F_{a_n, b_n}(x_n)) 
      \ar[r, two heads, "\bar \omega"]  &
      F(x_1 \bxt \cdots \bxt x_n)\rlap{ ,}
    \end{tikzcd}
  \end{equation*}
  where each of the outer ``$F$'s'' is an $F_{\ctensor_i a_i,
    \ctensor_i b_i}$. But it is easy to see that this diagram is a split
  coequaliser, when equipped with the splittings $F(\eta \bxt \cdots
  \bxt \eta)$
  for $\bar \omega$ and $F(F\eta \bxt \cdots
  \bxt F\eta)$ for $F(u)$.
\end{proof}

\begin{thm} \label{thm:commuting-bimodules-representable}
The symmetric multicategory $\CommMultBimC$ of bimodules and commuting bimodule morphisms
is representable.
\end{thm}
\begin{proof}
  Since we know that all pre-universal multimaps exist in
  $\CommMultBimC$, it suffices as explained above to show that
  pre-universal morphisms are closed under composition. It suffices to
  do this for the placed composition $\circ_i$, and by symmetry we can
  reduce to the case of $\circ_1$. Thus, let $p_i \co (A_i, a_i)
  \tickar (B_i, b_i)$ be bimodules for $1 \leqslant i \leqslant n$,
  and let $1 < m < n$; we will show that the composite of
  pre-universal morphisms
  \begin{equation*}
    p_1, \dots, p_m, p_{m+1}, \dots, p_n \xrightarrow{\lambda_m, 1, \dots, 1}
    p_1 \ctensor \dots \ctensor p_m, p_{m+1}, \dots, p_n \xrightarrow{\lambda_{n-m+1}} (p_1 \ctensor \dots \ctensor p_m) \ctensor p_{m+1} \ctensor \dots \ctensor p_n
  \end{equation*}
  is again pre-universal. This composite factors through the
  universal $n$-ary multimorphism $\lambda_n$, say as
  \begin{equation*}
    p_1, \dots, p_m, p_{m+1}, \dots, p_n \xrightarrow{\ \lambda_m\ }
    p_1 \ctensor \dots \ctensor p_n \xrightarrow{\ \ \kappa\ \ } (p_1 \ctensor \dots \ctensor p_m) \ctensor p_{m+1} \ctensor \dots \ctensor p_n
  \end{equation*}
  and it suffices to show that $\kappa$ is invertible. Now,
  $p_1 \ctensor \dots \ctensor p_m$ is the coequaliser of the reflexive pair
  \begin{equation*}
    \begin{tikzcd}
      F \big( Fp_1 \bxt \cdots \bxt Fp_m \big) 
      \ar[r, shift right, "\bar v"'] 
      \ar[r, shift left, "F(u)"] &[2em] 
      F(p_1 \bxt \dots \bxt p_m)\rlap{ .}
    \end{tikzcd}
  \end{equation*}
  while for each $m+1 \leqslant i \leqslant n$, the bimodule $p_i$ is
  as in \cref{equ:bim-as-coeq-of-free} the coequaliser of the reflexive pair
  \begin{equation*}
    \begin{tikzcd}
      F^2p_i
      \ar[r, shift right, "\mu"'] 
      \ar[r, shift left, "F(\alpha)"] &[2em] 
      Fp_i\rlap{ .}
    \end{tikzcd}
  \end{equation*}
  It follows by~\cref{thm:ctensor-preserves-coequalisers} that $(p_1
  \ctensor \dots \ctensor p_m) \ctensor p_{m+1} \ctensor \dots
  \ctensor p_n$ is the coequaliser of the reflexive pair
  \begin{equation*}
    \begin{tikzcd}
      F \big( Fp_1 \bxt \cdots \bxt Fp_m \big) \ctensor F^2p_{m+1} \ctensor \dots \ctensor F^2p_n
      \ar[r, shift right, "\bar v \ctensor \alpha \ctensor \dots \ctensor \alpha"'] 
      \ar[r, shift left, "F(u) \ctensor \mu \ctensor \dots \ctensor \mu"] &[2em] 
      F(p_1 \bxt \dots \bxt p_m) \ctensor Fp_{m+1} \ctensor \dots \ctensor Fp_n\rlap{ .}
    \end{tikzcd}
  \end{equation*}
  and by~\cref{prop:free-tensor-bimodules-is-free} and a short
  calculation, this is isomorphic to the reflexive pair
  \begin{equation*}
    \begin{tikzcd}
      F \big( Fp_1 \bxt \cdots \bxt Fp_m \bxt Fp_{m+1} \bxt \dots \bxt Fp_n\big)
      \ar[r, shift right, "\bar v"'] 
      \ar[r, shift left, "F(u)"] &[2em] 
      F(p_1 \bxt \dots \bxt p_m \bxt p_{m+1} \bxt \dots \bxt p_n)
    \end{tikzcd}
  \end{equation*}
  defining $p_1 \ctensor \cdots \ctensor p_n$; whence $\kappa$ is
  invertible, as required.
\end{proof}

\begin{rmk} \label{rmk:summarycomparison}
We proceeded in a different way from~\cite{DwyerW:BoardmanVtpo}, who
first define the commuting tensor product of free bimodules 
and then extend it to all bimodules via a colimit diagram (see~\cite[Equation~(1.4)]{DwyerW:BoardmanVtpo}).  Here, we first isolated the universal property
that the commuting tensor of bimodules should satisfy, by introducing the notion of a 
commuting multimorphism of bimodules. The two constructions agree, as 
the colimit is equivalent to the coequalizer in~\cref{equ:key-for-bim-representable}.
\end{rmk} 

In a similar spirit to \cref{thm:nctensor-smc,thm:ctensor-monads}, the desired symmetric monoidal structure on the horizontal category of $\BimC$ is now
immediate.

\begin{prop} \label{thm:ctensor-monoidal-horizontal}
The horizontal category $\BimC_1$ of $\BimC$
admits the structure of a symmetric monoidal category, with tensor product the commuting tensor product of bimodules.
\end{prop}

\begin{proof} 
It is the underlying monoidal category of the symmetric multicategory of bimodules and commuting bimodule morphisms of \cref{thm:commuting-bimodules-representable}.
\end{proof}

\subsection{The symmetric normal oplax monoidal structure on $\BimC$}
The desired oplax monoidal structure on $\BimC$ as a double category consists of the monoidal structures on the vertical and horizontal category, established in 
\cref{thm:ctensor-monoidal-vertical,thm:ctensor-monoidal-horizontal} respectively, as well as the rest of the data in~\cref{def:oplaxmonoidal}.
In particular, it involves interchange maps, which relate the composition of bimodules with their
commuting tensor product. We begin by constructing these for free bimodules.

\begin{lem}[Interchange for free bimodules] \label{thm:interchange-free-bimodules}
For $i \in \{1, 2\}$, let $F_{a_i, b_i}(x_i) \co (A_i, a_i) \tickar (B_i, b_i)$ and $F_{b_i, c_i}(y_i) \co (B_i, b_i) \tickar (C_i, c_i)$ be free bimodules. Then there exists a bimodule 2-morphism
\[
\begin{tikzcd}[column sep = 3cm] 
(A_1 \bxt A_2, a_1 \ctensor a_2) \ar[rr,tick, "(F(y_1) \bimcomp F(x_1)) \ctensor ( F(y_2) \bimcomp F(x_2))"] \ar[d, equal] \ar[drr, phantom,  "\Two \xi'"] & & (C_1 \bxt C_2, c_1 \ctensor c_2)  \ar[d, equal] \\
(A_1 \bxt A_2, a_1 \ctensor a_2)   \ar[r,tick, "F(x_1) \ctensor F(x_2)"'] & (B_1 \bxt B_2,b_1\ot b_2) \ar[r, tick,"F(y_1) \ctensor F(y_2)"'] & (C_1 \bxt C_2, c_1 \ctensor c_2) 
\end{tikzcd}
\]
\end{lem}

\begin{proof} For the bimodule at the top of the diagram, we have an isomorphism
\begin{align*} 
(F_{b_1, c_1} (y_1) \bimcomp F_{a_1, b_1}(x_1)) \ctensor ( F_{b_2, c_2} (y_2) \bimcomp F_{a_2, b_2}(x_2)) & \cong 
F_{a_1, c_1}( y_1 \hcomp b_1 \hcomp x_1) \ctensor F_{a_2, c_2}( y_2 \hcomp b_2 \hcomp x_2) \\
& \cong F_{a_1 \ctensor a_2, c_1 \ctensor c_2}( ( y_1 \hcomp b_1 \hcomp x_1) \bxt ( y_2 \hcomp b_2 \hcomp x_2 ) )
\end{align*}
given by~\cref{equ:bimcomp-of-free-is-free}  and~\cref{equ:ctensor-of-free-bimodules}, respectively.
For the composite bimodules at the bottom of the diagram, we have an isomorphism 
\begin{align*} 
\big( F_{b_1, c_1}(y_1) \ctensor F_{b_2, c_2}(y_2) \big)  \bimcomp 
\big(  F_{a_1, b_1}(x_1) \ctensor F_{a_2, b_2}(x_2) \big)  & \cong 
\big ( F_{b_1 \ctensor b_2, c_1 \ctensor c_2}( y_1 \bxt y_2 ) \big) \bimcomp
\big( F_{a_1 \ctensor a_2, b_1 \ctensor b_2}( x_1 \bxt x_2) \big)   \\
 & \cong F_{a_1 \ctensor a_2, c_1 \ctensor c_2} \big( (y_1 \bxt y_2) \hcomp (b_1 \ctensor b_2) \hcomp (x_1 \bxt x_2) \big)
 \end{align*}
where we used~\cref{equ:ctensor-of-free-bimodules} and~\cref{equ:bimcomp-of-free-is-free}, respectively. 
Therefore, the required globular bimodule 2-morphism $\xi'$ is given by composing these bimodule isomorphisms with the 2-morphism obtained by applying the free
$(c_1 \ctensor c_2, a_1 \ctensor a_2)$-bimodule functor $F_{a_1 \ctensor a_2, c_1 \ctensor c_2}$ to the composite 2-morphism
\[
\begin{tikzcd}
A_1 \bxt A_2 \ar[rrr, tick,"(y_1 \hcomp b_1 \hcomp x_1) \bxt (y_2 \hcomp b_2 \hcomp x_2)"] \ar[d, equal] 
\ar[drrr, phantom,  "\Two \xi"] & & & C_1 \bxt C_2 \ar[d, equal] \\
A_1 \bxt A_2 \ar[r, tick,"x_1 \bxt x_2"] \ar[d, equal] & B_1 \bxt B_2 \ar[r,tick, "b_1 \bxt b_2"] \ar[d, equal] \ar[dr, phantom,  "\Two \pi"]  & B_1 \bxt B_2 \ar[r,tick, "y_1 \bxt y_2"] \ar[d, equal] & C_1 \bxt C_2 \ar[d, equal]  \\
A_1 \bxt A_2 \ar[r, tick,"x_1 \bxt x_2"']   & B_1 \bxt B_2 \ar[r, tick,"b_1 \ctensor b_2"'] & B_1 \bxt B_2 \ar[r, tick,"y_1 \bxt y_2"'] & C_1 \bxt C_2 \mathrlap{,}
\end{tikzcd} 
\]
where $\xi$ is the interchange for $\dcC$ and $\pi$ is the comparison 2-cell in~\cref{equ:bxt-to-ctensor}. 
\end{proof}

\begin{thm} \label{thm:ctensor-bimodules} Let $\dcC$ be a symmetric normal oplax monoidal closed 
double category. Under \cref{hyp:fibrancy,hyp:coproducts-in-mnd,hyp:equalisers,hyp:tabulators,hyp:free-monad,hyp:coequalisers,hyp:equalisers-in-horizontal,hyp:refl-coequalisers},  the double category $\BimC$ admits the structure of a symmetric oplax monoidal double category,
with tensor product on objects given by the commuting tensor product of monads.
\end{thm} 

\begin{proof} We need to verify the clauses of \cref{def:oplaxmonoidal,def:symmetricoplaxmon}.
As mentioned earlier, we have established the symmetric monoidal structure on $\BimC_0$ in \cref{thm:ctensor-monoidal-vertical} and on $\BimC_1$ in \cref{thm:ctensor-monoidal-horizontal}. By construction, the source and target functors are symmetric strict monoidal.

Next, we construct the interchange laws. For this, we extend the interchange law from free bimodules (\cref{thm:interchange-free-bimodules}) to arbitrary ones using the fact that every bimodule is a reflexive coequaliser of free ones.
Let $p_i \co (A_i, a_i) \tickar (B_i, b_i)$ and $q_i \co (B_i, b_i) \tickar (C_i, c_i)$, for $i \in \{ 1, 2 \}$, be bimodules. The required bimodule 2-morphism
\begin{equation}
\label{equ:bim-oplax-1}
\tilde{\xi} \co (q_1 \bimcomp p_1) \ctensor (q_2 \bimcomp p_2) \to (q_1 \ctensor q_2) \bimcomp (p_1 \ctensor p_2)
\end{equation}
is then induced by the universal property of coequalisers via the following diagram:
\begin{equation}
\label{equ:construction-bim-interchange}
\begin{tikzcd}
( F^2 q_1 \bimcomp F^2 p_1 ) \ctensor ( F^2 q_2 \bimcomp F^2 p_2 ) \ar[r, "(\ast)"]  \ar[d, shift right = 1ex] \ar[d, shift left = 1ex] & (F^2q_1 \ctensor F^2q_2) \bimcomp (F^2p_1 \ctensor F^2 p_2)  \ar[d, shift right = 1ex] \ar[d, shift left = 1ex] \\
( Fq_1 \bimcomp Fp_1 ) \ctensor ( Fq_2 \bimcomp F p_2 ) \ar[r, "(\ast \ast)"] \ar[d] & (Fq_1 \ctensor Fq_2) \bimcomp (Fp_1 \ctensor Fp_2) \ar[d] \\
(q_1 \bimcomp p_1) \ctensor (q_2 \bimcomp p_2) \ar[r, dotted, "\tilde{\xi}"] &  (q_1 \ctensor q_2) \bimcomp (p_1 \ctensor p_2) \mathrlap{.}
\end{tikzcd}
\end{equation}
Here, the left and right columns are obtained from the diagrams exhibiting the bimodules under consideration as 
coequalisers of free bimodules as in \cref{equ:bim-as-coeq-of-free}, via functors
that preserve coequalisers, and therefore are again coequaliser diagrams. Indeed, the operations involved are composition of bimodules, which preserves
reflexive coequalisers by \cref{thm:bim-c-fibrant}, and the commuting tensor, which preserves reflexive coequalisers by \cref{thm:ctensor-preserves-coequalisers}. 
The horizontal maps labelled $(\ast)$ and $(\ast \ast)$ are instances of the interchange 2-morphism $\xi'$ for free bimodules constructed in~\cref{thm:interchange-free-bimodules}.

Moreover, we have bimodule isomorphisms of the form
\begin{equation}
\label{equ:bim-oplax-2}
\begin{tikzcd}[column sep = 3cm]
(A_1, a_1) \ctensor (A_2, a_2) 
	\ar[r, "\id_{(A_1, a_1)} \ctensor \id_{(A_2, a_2)}"] 
	\ar[d,equal] 
	\ar[dr, phantom,  "\cong"] & 
(A_1, a_1) \ctensor (A_2, a_2) 
	\ar[d,equal] \\
(A_1, a_1) \ctensor (A_2, a_2) 
	\ar[r, "\id_{(A_1, a_1) \ctensor (A_2, a_2)}"'] & 
(A_1, a_1) \ctensor (A_2, a_2)
\end{tikzcd}
\end{equation}
\begin{displaymath}
\begin{tikzcd}[column sep = 2cm]
(I, \id_I) 
	\ar[rr, "\id_{(I, \id_I})"]
	 \ar[d,equal] 
	 \ar[drr, phantom,  "\cong"]  
	 & 
	 & 
(I, \id_I) 
	\ar[d, equal] \\
(I, \id_I) 
	\ar[r, "\id_{(I, \id_I)}"'] & 
(I, \id_I) 
	\ar[r, "\id_{(I, \id_I)}"'] & 
(I, \id_I) \mathrlap{,}
\end{tikzcd} \qquad
\begin{tikzcd}[column sep = large]
(I, \id_I) 
	\ar[r,"\id_{(I, \id_I})"] 
	\ar[d,  equal] 
	\ar[dr, phantom,  "\cong"]  & 
(I, \id_I) 
	\ar[d,  equal]  \\
(I, \id_I) 
	\ar[r,"\id_{(I, \id_I)}"'] & 
(I, \id_I) \mathrlap{.}
\end{tikzcd} 
\end{displaymath}
In fact, the latter of these is an identity.
These can be constructed easily, recalling that, for a  monad $(A,a)$, 
the horizontal identity $\id_{(A,a)} \co (A,a) \tickar (A,a)$ in $\BimC$
is given by the horizontal map $a \co A \tickar A$, viewed as a $(a,a)$-bimodule via the monad multiplication.

The coherence conditions for this structure to form an oplax monoidal double category which is furthermore symmetric, follow by the universal properties with which all these bimodule 2-morphisms
are constructed.
\end{proof} 

\begin{rmk} \label{thm:not-normal}
\cref{thm:ctensor-bimodules} does not claim that the oplax monoidal structure of the double category $\BimC$ is \emph{normal}, in the sense of \cref{def:normality}. Indeed, while  the bimodule 2-morphisms 
in~\eqref{equ:bim-oplax-2} are invertible, as required, the interchange maps in \eqref{equ:bim-oplax-1} do not seem to satisfy the appropriate normality conditions. To illustrate this point,  let $(A,a_1)$, $(A_2, a_2)$ be monads and consider 
the identity bimodule $\id_{(A_1, a_1)} \co (A_1, a_1) \tickar (A_1, a_1)$ and bimodules $p_2 \co (A_2, a_2) \tickar (B_2, b_2)$, $q_2 \co (B_2, b_2) \tickar (C_2, c_2)$. Then, there is no reason for which
the interchange map for these bimodules
\[
\tilde{\xi} \co (a_1 \bimcomp a_1) \ctensor (q_2 \bimcomp p_2) \to 
(a_1 \ctensor q_2) \bimcomp (a_1 \ctensor p_2)
\]
 constructed as in \eqref{equ:construction-bim-interchange} should be an isomorphism.
 This is essentially because the map $(\ast \ast)$ of~\eqref{equ:construction-bim-interchange}, which has the form
 \[
 (Fa_1 \bimcomp Fa_1) \ctensor (Fq_2 \bimcomp Fp_2) \xrightarrow{\xi'}
 (Fa_1 \ctensor Fq_2) \bimcomp (Fa_1 \ctensor Fp_2)
\]
does not seem to be an isomorphism either. Indeed,  the latter is obtained by composing isomorphisms with the result
of applying the free bimodule functor to the 2-morphism
\[
\begin{tikzcd}
A_1 \bxt A_2 \ar[rrr, tick,"(a_1 \hcomp a_1 \hcomp a_1) \bxt (a_2 \hcomp b_2 \hcomp p_2)"] \ar[d, equal] 
\ar[drrr, phantom,  "\Two \xi"] & & & C_1 \bxt C_2 \ar[d, equal] \\
A_1 \bxt A_2 \ar[r, tick,"a_1 \bxt p_2"] \ar[d, equal] & B_1 \bxt B_2 \ar[r,tick, "a_1 \bxt b_2"] \ar[d, equal] \ar[dr, phantom,  "\Two \pi"]  & B_1 \bxt B_2 \ar[r,tick, "a_1 \bxt q_2"] \ar[d, equal] & C_1 \bxt C_2 \ar[d, equal]  \\
A_1 \bxt A_2 \ar[r, tick,"a_1 \bxt p_2"']   & B_1 \bxt B_2 \ar[r, tick,"a_1 \ctensor b_2"'] & B_1 \bxt B_2 \ar[r, tick,"a_1 \bxt q_2"'] & C_1 \bxt C_2 \mathrlap{,}
\end{tikzcd}
\]
which is not invertible in general, as it involves the interchange of $\dcC$ on non-identity maps.
Attempting to present $(a_1 \bimcomp a_1) \ctensor (q_2 \bimcomp p_2)$ with a different 
coequaliser than the one on the left-hand side of \eqref{equ:construction-bim-interchange},
exploiting that some of the bimodules involved are identity bimodules, runs into the same issue.

Since the oplax monoidal structure of $\BimC$ does not seem to be normal, we do not extract a symmetric monoidal structure on the horizontal bicategory of $\BimC$ (see \cite[Section~4]{Paper1} for details). Far from being a problem, this subtle aspect provides extra motivation for our double categorical approach.
\end{rmk}

\section{Applications}
\label{sec:applications}

In this last section, we present two applications of our theory.
The first one recovers the known fact that the tensor product of (enriched) categories extends to a symmetric monoidal structure on the double category of (enriched) profunctors. The second one establishes 
a new result, which is obtained by replacing (enriched) categories by (enriched) symmetric
multicategories, and profunctors by bimodules between multicategories -- which we here refer to as symmetric multiprofunctors (see \cref{defi:smultprof} for details). More precisely, we show that the Boardman--Vogt tensor product of (enriched) symmetric multicategories (see e.g. \cite{ElmendorfA:percma}) extends to a symmetric oplax monoidal structure on the double category of symmetric multiprofunctors.

\subsection{Categories and profunctors} Let $\ca{V}$ be a symmetric monoidal closed and locally presentable category. We write $\VMat$ for the double category of $\ca{V}$-matrices,
as considered in our running example (starting in \cref{thm:running-1}).
Although we have presented many aspects of this example in the course
of the paper, we give a self-contained treatment here, both for the
reader's convenience, and also in order to highlight the structural
similarities with our second application to follow. In particular, in
\cref{thm:mat-satisfies-1,thm:mat-satisfies-2} below, we re-verify the
fact that $\VMat$ satisfies all the assumptions required to apply our general results, building on~\cite{SweedlerTheoryDouble,Paper1}. We begin by establishing the symmetric monoidal closed structure. Note that this is a genuine monoidal structure, rather than an oplax one.

\begin{lem} \label{thm:mat-satisfies-1}
 The double category $\VMat$ of $\ca{V}$-matrices admits a symmetric monoidal closed 
structure, given on objects by the cartesian product of sets.
\end{lem}

\begin{proof} This is \cite[Examples~3.2.2 and~3.2.12]{SweedlerTheoryDouble}, but we recall some
definitions in preparation for our second application in \cref{subsec:symseq}.
The vertical category of $\VMat$ is $\Set$, which is cartesian closed. We write $\boxhom{B,C}$ for the set of functions between $B$ and $C$ and $\lambda(f) \co A \to \boxhom{B, C}$ for the adjoint transpose of a function $f \co A \times B \to C$.

This closed structure can be extended to the horizontal category of $\VMat$. For $x \co A \tickar A'$, $y \co B \tickar B'$,
their tensor matrix $x \otimes y \co A \times A' \tickar B \times B'$ is defined by letting
\[
(x \otimes y)\big[ (b,b'); (a,a') \big] \defeq x[a';a] \otimes y[b';b] \mathrlap{.}
\]
For $ y \co  B \tickar B'$, $z \co C \tickar C'$, their internal hom $\boxhom{y, z} \co
\boxhom{B, C} \tickar \boxhom{B', C'}$ is defined by letting
\begin{equation*}
\boxhom{y,z}[ g'; g] \defeq
	\prod_{b \in B, b'  \in B'}
	\Big[
	y[ b'; b] , \
	z\big[ g'(b'); g(b) \big] \mathrlap{,}
	\Big]
	\end{equation*}
using the internal hom in $\ca{V}$.
Here, observe that both $y[ b'; b]$ and $z \big[  g'(b'); g(b) \big]$ are in $\ca{V}$.

To conclude the proof that the symmetric monoidal structure of the horizontal category of $\VMat$ is closed, one needs to show that there is a bijection between 2-morphisms on the left and on the right:
\[
\begin{tikzcd}
A \times B \ar[r,tick, "x \otimes y"] \ar[d, "f"']  \ar[dr, phantom,  "\Two \phi"] & A' \times B' \ar[d, "f'"] \\
C \ar[r,tick, "z"'] & C' \mathrlap{,}
\end{tikzcd} \qquad
\begin{tikzcd}
A  \ar[r,tick, "x"]  \ar[d, "\lambda(f)"'] \ar[dr, phantom,  "\ \Two \psi"]  & A'  \ar[d, "\lambda(f')"] \\
\boxhom{B, C} \ar[r, "\boxhom{y,z}"'] & \boxhom{B',C'}  \mathrlap{.}
\end{tikzcd}
\] 
To see this, note that a 2-morphism $\phi$ as above has components of the form
\[
\phi_{(a',b'), (a,b)} \co x[a';a] \otimes y[b';b]  \to z[ f'(a', b'); f(a,b)] \mathrlap{,}
\]
while a 2-morphism like $\psi$ has components of the form
\[
\psi_{a, a'} \co x(a',a) \to \prod_{b \in B, b' \in B'} \big[ y[b';b] , z[ f'(a',b'); f(a,b)]\big] \mathrlap{,}
\]
see~\cite[Proposition~4.3]{VCocats} for more details.
\end{proof} 

The next lemma deals with structure of the double category $\VMat$ required to apply
the general theory of the previous sections, as detailed already in \cref{tab:assumptions}.

\begin{lem} \label{thm:mat-satisfies-2}  The double category $\VMat$ 
satisfies \cref{hyp:fibrancy,hyp:coproducts-in-mnd,hyp:equalisers,hyp:tabulators,hyp:free-monad,hyp:coequalisers,hyp:equalisers-in-horizontal,hyp:refl-coequalisers}.
\end{lem} 

\begin{proof}  We need to check that:
\begin{enumerate}[(I)]
\item $\VMat$ is fibrant,
\item each $\Mnd_A(\VMat)$ has coproducts,
\item $(\VMat)_0$ has equalisers,
\item $\VMat$ has 1-tabulators,
\item each forgetful functor $\Mnd_A(\VMat) \to \End_A(\VMat)$ has a left adjoint,
\item each $\Mnd_A(\VMat)$ has local coequalisers,
\item $\VMat$ has local equalisers,
\item $\VMat$ has stable local reflexive coequalisers.
\end{enumerate}
Part~(I) is well-known. 
For part~(III), note that $(\VMat)_0 = \Set$. The existence of 1-tabulators (\cref{def:1tabulators}) as required for part~(IV) is known, and was discussed in \cref{thm:running-5}. 

Next, recall that, for sets $A$ and $B$, the category $\VMat[A,B]$ defined in~\eqref{Matlp} is locally
presentable. In particular, $\VMat$ has local equalisers, giving us (VII).
From the proof of \cite[Proposition~3.4.8]{SweedlerTheoryDouble}, 
local small colimits in $\VMat$ are stable. In particular, (VIII) holds.
Furthermore, (V) holds by known results on free monoids~\cite{KellyGM:unittc,FreeMonoids}, since e.g. each category $\End_A(\VMat)=\VCat(A\times A,\ca{V})$ has countable coproducts preserved by horizontal composition
in each variable.
It follows from this and fibrancy that $\VMat$ has free monads (\cref{def:free-monads})
(see also comments below \cref{def:free-monads}).   But now, since
 $\VMat$ has stable local reflexive coequalisers, free monads, and local small colimits,
 $\Mnd_A(\VMat)$ has local small colimits by part~(iii) of \cref{lem:limits-in-end-mnd}
 (see also comments below \cref{lem:local-to-global-limits} and \cref{lem:limits-in-end-mnd}).
 In particular, we obtain (II) and (VI).
\end{proof} 

When $\dcC=\VMat$, the theory developed in \cref{sec:non-commuting} and \cref{sec:commuting} recovers well-known facts.
Monad multimorphisms in $\VMat$ are called sesquifunctors (see e.g.~\cite{GarnerLopezFranco})
and the non-commuting tensor product defined in \cref{eq:nctensor} is
the so-called \emph{funny tensor product} of $\ca{V}$-categories, while, 
commuting monad multimorphisms are $n$-ary functors and the commuting tensor
product of \eqref{eq:commuting-coequaliser} is the usual tensor product of $\ca{V}$-categories.

We now consider the double category of bimodules
in $\VMat$, which is exactly the double category $\VProf$ of $\ca{V}$-profunctors. Indeed,
as mentioned above, monads in $\VMat$ are $\ca{V}$-categories.  Furthermore, monad
morphisms are $\ca{V}$-functors and monad bimodules are $\ca{V}$-profunctors (e.g.~\cite[Example~11.8]{Framedbicats}). We therefore set:
\[
\VProf =  \Bim(\VMat) \mathrlap{.}
\]
Here, for small $\ca{V}$-categories $A$ and $B$, the hom-category of horizontal maps and globular 2-cells between~$A$ and~$B$ in $\VProf$ is
\[
\VProf[A, B] = \VCat(B^\op \ctensor A, \ca{V}) \mathrlap{.}
\]
We write $p \co A \tickar B$ to denote a profunctor from $A$ to $B$, \ie a functor $p \co B^\op \ctensor A \to \ca{V}$. Our theory  applies to recover the following known fact~\cite{ConstrSymMonBicatsFun}.

\begin{thm}  \label{thm:app-prof} Let $\ca{V}$ be a symmetric monoidal closed locally presentable category. 
The double category $\VProf$ of $\ca{V}$-categories, $\ca{V}$-functors, and $\ca{V}$-profunctors admits a symmetric monoidal structure, given on objects by the tensor product of $\ca{V}$-categories.
\end{thm}

\begin{proof} This is an application of~\cref{thm:ctensor-bimodules}, in this case where the oplax monoidal double structure is actually (pseudo) monoidal. The hypotheses are verified
in~\cref{thm:mat-satisfies-1,thm:mat-satisfies-2}. A direct calculation shows that a monad multimorphism
is commuting if and only if it is a $\ca{V}$-functor in many variables. Hence, the commuting
tensor coincides with the tensor product of $\ca{V}$-categories.
\end{proof}

As noted in~\cref{thm:running-6}, we may generalise the above from
the familiar situation of a symmetric monoidal base $\ca{V}$ to that
of a normal duoidal base $\ca{V}$; on doing so, we reconstruct exactly
the ``many-object'' theory of commuting tensor products of
$\ca{V}$-categories over such a base given in~\cite{GarnerLopezFranco}.

\begin{cor} The bicategory of $\ca{V}$-categories, $\ca{V}$-profunctors, and $\ca{V}$-natural transformations
admits a symmetric monoidal structure, given on objects as the tensor product of $\ca{V}$-categories.
\end{cor}

\begin{proof} This follows from~\cref{thm:app-prof} via \cite{Paper1}.
\end{proof}

\subsection{Symmetric multicategories and symmetric multiprofunctors} \label{subsec:symseq}

We now consider a locally presentable category $\ca{V}$, which we view as symmetric monoidal with respect to its finite products. The restriction
to cartesian monoidal structure was already imposed in~\cite{Paper1} in order to construct the \emph{arithmetic product} of categorical symmetric sequences, namely the oplax monoidal structure of our base double category $\VCatSym$, which we now recall.
 
First of all, we set up some preliminary notation. Let $A$ be small $\ca{V}$-category $A$. For $n \in \mathbb{N}$, let $S^nA$ be the $\ca{V}$-category with objects $n$-tuples
 $\vec{a} = (a_1, \ldots, a_n)$ of objects of $A$ and as hom-objects 
 \[
 S^nA \big( (a_1, \ldots, a_n), (a'_1, \ldots, a'_n)\big) \defeq 
 \coprod_{\sigma \in \mathfrak{S}_n} A(a_i, a'_\sigma(i))  \mathrlap{,}
 \]
 where $\mathfrak{S}_n$ is the $n$-th symmetric group. We then let 
 \[
 SA = \coprod_{n \in \mathbb{N}} S^n A \mathrlap{.}
 \]
  The $\ca{V}$-category $SA$ admits a symmetric strict monoidal structure whose tensor product, 
 written~$\vec{a}_1 \oplus \vec{a}_2$, is concatenation of sequences and whose unit is the empty 
 sequence, written $e$. The category $SA$ is in fact the free symmetric strict monoidal $\ca{V}$-category on $A$.
 
 The double category $\VCatSym$ has small $\ca{V}$-categories as objects, $\ca{V}$-functors as vertical maps and symmetric $\ca{V}$-sequences as horizontal maps. Explicitly, a categorical symmetric sequence $x \co A \rightsquigarrow B$ from $A$ to $B$ is
a $\ca{V}$-profunctor $x \co A \tickar SB$, i.e. a $\ca{V}$-functor $x\colon SB^\op
\times A\to\ca{V}$\footnote{Since $\ca{V}$ here is cartesian monoidal, we write $A\times B$ for the usual (pointwise) tensor product of $\ca{V}$-categories.}. 
The value of such a symmetric sequence on $(\vec{b}, a) \in SB^\op
\times A$ is written $x[\vec{b}; a]$. A 2-morphism in $\VCatSym$ of the form
\begin{equation}\label{2mapCatSym}
\begin{tikzcd}
A \ar[r,rightsquigarrow, "x"] \ar[d, "f"']  \ar[dr, phantom,  "\Two \phi"] & A' \ar[d, "f'"] \\
B \ar[r,rightsquigarrow,  "y"'] & B'
\end{tikzcd}
\end{equation}
is a $\ca{V}$-natural transformation with components $\phi_{\vec{a}', a} \co x[\vec{a}'; a] \to y[ Sf' (\vec{a}'); f(a)]$.
For small $\ca{V}$-categories $A$ and $B$, their hom-category (of horizontal maps between $A$ and $B$ and globular 2-cells) in $\VCatSym$ is
\begin{equation}
\label{equ:catsym-hom}
\VCatSym[A,B] = \VProf[A, SB] =  \VCat(SB^\op \times A, \ca{V}) \mathrlap{.}
\end{equation}

Let us also recall from~\cite{FioreM:carcbg,GambinoJoyal}
that, for $x \co A \rightsquigarrow B$ and $y \co B \rightsquigarrow C$,
their composite categorical sequence $y \hcomp x \co A \rightsquigarrow  C$ is defined by
\[
(y \hcomp x)[\vec{c}; a] \defeq \coend^{\vec{b} \in SB} y^e[\vec{c}; \vec{b}] \times
x[\vec{b}; a] \mathrlap{,}
\]
where, for $\vec{b} = (b_1, \ldots, b_n)  \in SB$,
\[
y^e[\vec{c}; \vec{b}] = \coend^{\vec{c}_1 \in SC} \cdots \coend^{\vec{c}_n \in SC} y[\vec{c}_1; b_1] \times \cdots
\times y[\vec{c}_n; b_n] \times SC(\vec{c}, \vec{c}_1 \oplus \cdots \oplus \vec{c}_n)\mathrlap{.}
\]

We now recall the definition of the arithmetic tensor product in $\VCatSym$ from~~\cite{Paper1}. On objects, $A\bxt B$ is defined to be their tensor product $A\times B$. Given categorical symmetric sequences $x \co A  \rightsquigarrow
A'$ and $y \co B \rightsquigarrow B'$, we define $x \bxt y \co A \bxt B \rightsquigarrow A' \bxt B'$ by letting, for $\vec{u} \in S(A' \bxt B')$ and $(a, b) \in A \bxt B$,
\begin{equation}
\label{equ:arith-prod-catsym}
(x \bxt y)[\vec{u}'; (a,b)]   \int^{\vec{a}' \in SA} \int^{\vec{b'} \in SB'} S(A' \bxt B')( \vec{u}', \vec{a}' \bxt \vec{b}') \times x[\vec{a}'; a] \times y[ \vec{b}'; b]  \mathrlap{,}
\end{equation}
where, for $\vec{a}' = \big( a'_i \big)_{1 \leq i \leq m}$ and
$\vec{b}' = \big( b_{j} \big)_{1 \leq j \leq n}$, 
\[
\vec{a}' \bxt \vec{b}' = \big( (a'_{1}, b'_{1}), (a'_2, b'_1), \ldots, (a'_m, b_1), \ldots, (a'_1, b'_n), \ldots, (a'_m, b'_n) \big) \mathrlap{.}
\]

The next two lemmas check that $\VCatSym$ satifies the assumptions needed to apply our general theory, expanding on~\cite[Theorem~10.9]{Paper1}. The first one is analogous to the previous
\cref{thm:mat-satisfies-1}.
Note that the closed structure generalises the  divided power closed structure of~\cite{DwyerW:BoardmanVtpo}, as we explain 
in~\cref{thm:divided-powers} below, and also requires that $\ca{V}$ is cartesian \emph{closed}.

\begin{lem} \label{thm:catsym-oplaxmonoidal}
The double category $\VCatSym$ of categorical symmetric $\ca{V}$-sequences admits a symmetric normal oplax monoidal closed structure, 
given on horizontal maps by the arithmetic product of categorical symmetric sequences.
\end{lem}

\begin{proof} The normal oplax monoidal structure of $\VCatSym$ is~\cite[Theorem~10.9]{Paper1}.
The symmetry of the tensor product is straightforward. 

For the closure of the monoidal structure, first note that the vertical fragment of $\VCatSym$ is $\VCat$ and that the arithmetic tensor product therein coincides with the tensor product of $\ca{V}$-categories, which is well-known to be part of a cartesian closed monoidal structure (as $\ca{V}$ is cartesian monoidal) whose internal hom is given by functor $\ca{V}$-categories, written $\boxhom{B,C}$ here. For~$f \co A \bxt B \to C$, we write $\lambda(f) \co A \to \boxhom{B, C}$ for its adjoint transpose.

Next, we extend this to the horizontal category of $\VCatSym$. For $y \co B \rightsquigarrow B'$ and $z \co C \rightsquigarrow C'$, 
 we define $\boxhom{y, z} \co \boxhom{B, C} \rightsquigarrow \boxhom{B', C'}$ by letting, for
$(g'_1, \ldots, g'_m) \in S\boxhom{B',C'}$ and $g \in \boxhom{B,C}$,
\begin{multline} 
\boxhom{y,z}[ g'_1, \ldots, g'_m; g] \defeq \\
	\prod_{n \in \mathsf{N}} 
	\int_{(b'_1, \ldots, b'_n) \in S_n B'}
	\int_{b \in B}
	\Big[
	y[ b'_1, \ldots, b'_n; b] , \
	z\big[ g'_1(b'_1), \ldots, g'_1(b'_n), \ldots, g'_m(b'_1), \ldots, g'_m(b'_n); g(b) \big] 
	\Big]
	\label{equ:closed-boxhom}
\end{multline}
using the internal hom of $\ca{V}$. Observe that $y[ b'_1, \ldots, b'_n; b]$ and $z \big[  g'_1(b'_1), \ldots, g'_1(b'_n), \ldots, g'_m(b'_1), \ldots, g'_m(b'_n); g(b) \big]$ are objects in $\ca{V}$.

To conclude the proof that horizontal category of $\VCatSym$ is closed, one needs to show that,
with the definitions above, there is a bijection between 2-cells on the left and on the right, as below:
\begin{equation}
\label{equ:closed-transposition}
\begin{tikzcd}
A \bxt B \ar[r,rightsquigarrow,  "x \bxt y"] \ar[d, "f"']  \ar[dr, phantom,  "\Two \phi"] & A' \bxt B' \ar[d, "f'"] \\
C \ar[r,rightsquigarrow,  "z"'] & C' \mathrlap{,}
\end{tikzcd} \qquad
\begin{tikzcd}
A  \ar[r,rightsquigarrow,  "x"]  \ar[d, "\lambda(f)"'] \ar[dr, phantom,  "\Two \psi"]  & A'  \ar[d, "\lambda(f')"] \\
\boxhom{B, C} \ar[r,rightsquigarrow,  "\boxhom{y,z}"'] & \boxhom{B',C'} \mathrlap{.}
\end{tikzcd}
\end{equation}
By the definition of 2-morphisms in $\VCatSym$ \cref{2mapCatSym}, a 2-morphism $\phi$ as on the left is a $\ca{V}$-natural transformation with components of the form
\[
\phi_{\vec{u}', (a,b)} \co (x \bxt y)[ \vec{u}; (a,b)] \to z[ (Sf') (\vec{u}'); f(a,b)]  \mathrlap{,}
\]
for $\vec{u}' \in S(A' \bxt B')$ and $(a,b) \in A \bxt B$. The domain of these maps is defined explicitly in~\eqref{equ:arith-prod-catsym}.
Instead, a 2-morphism $\psi$ as on the right above  is a $\ca{V}$-natural transformation with components of the form
\begin{equation}
\label{equ:cod-lambda-phi}
\psi_{\vec{a}', a} \co x[ \vec{a}'; a] \to \langle y, z \rangle [ S(\lambda f') (\vec{a}') ; \lambda f (a) ] 
\end{equation}
for $\vec{a}' \in SA'$ and $a \in A$. Note that $\lambda f(a) \in \langle B, C \rangle$ and $S(\lambda f') (\vec{a}') \in S \langle B', C' \rangle$. The former is simply the $\ca{V}$-functor $f(a,-)$. The latter, for $\vec{a}' = (a_1', \ldots, a_m')$, is the $m$-tuple of $\ca{V}$-functors $(f'(a'_1, -), \ldots, f'(a'_m, -))$. Hence, the codomain of the map in~\eqref{equ:cod-lambda-phi} can be simplified to 
\[
\langle y, z \rangle \big[ f'(a'_1,-), \ldots, f'(a_m, -); f(a,-)] \mathrlap{.}
\]
This, by the definition in~\eqref{equ:closed-boxhom}, is 
\[
\prod_{n \in \mathsf{N}}
\int_{\vec{b}' \in S_n B'}
	\int_{b \in B} 
	\Big[
	y[ \vec{b}'; b] , \ 
	z\big[ (Sf') (\vec{a}' \bxt \vec{b}') ; f(a,b) \big] 
	\Big]	\mathrlap{.}
\]
The required bijection now follows by a direct calculation with ends and coends which we omit.

In order for the monoidal closed structure on the vertical and horizontal fragment of the double category $\VCatSym$ to form an oplax monoidal \emph{closed} double structure, as in 
\cref{def:closeddouble} one should also verify that the source and target functors commute 
with the definitions of the closed structures, and that the evaluation square of the closed structure on $(\VCatSym)_1$ should have  the evaluations of the closed structure on $\VCat$ as vertical maps, as in
\[
\begin{tikzcd}[column sep = large] 
\langle B, C \rangle \bxt B 
	\ar[r,rightsquigarrow,  "{\langle y, z \rangle \bxt y}"] 
	 \ar[d, "\ev_{B,C}"']  
	 \ar[dr, phantom,  "\Two \ev_{y,z}"] 
	&  
\langle B' \bxt C' \rangle \bxt B' 
	 \ar[d, "\ev_{B', C'} '"] 
	\\
C 
	 \ar[r,rightsquigarrow,  "z"'] 
	& 
C' \mathrlap{.}
\end{tikzcd} 
\] 
These hold by the very definition of $\ev_{y,z}$ as the transpose of the identity on $\langle y, z \rangle$ 
across the bijection in~\eqref{equ:closed-transposition}.
We leave these routine calculations to the readers.
\end{proof}

\begin{rmk} \label{thm:divided-powers}
Just as the arithmetic product of categorical symmetric sequences defined in~\eqref{equ:arith-prod-catsym} generalises the arithmetic product of single-sorted symmetric sequences defined in~\cite{DwyerW:BoardmanVtpo}, the monoidal closed structure defined above generalises
the \emph{divided power} closed structure defined therein. Indeed, if we take $A = B = C = 1$ above, the formula for $\boxhom{y,z}$ reduces to the end
\[
\boxhom{y,z}(m) \defeq \int_{n \in S1} \big[ y(n), z(m \times n) \big]  \mathrlap{.}
\]
In particular, note that the first argument of $z$ in~\eqref{equ:closed-boxhom} is a sequence of length
$m \times n$.
\end{rmk}

We move on to the rest of the assumptions on the double category of symmetric sequences (\cref{tab:assumptions}), that allow for our general theory from the previous sections to be applied. 
The proof of \cref{thm:catsym-has-refl-coeq} is similar to that of \cref{thm:mat-satisfies-2}, but
there is an important difference: in our current setting, we do not have stability of all local small colimits, but only of 
sifted ones (and in particular reflexive coequalisers). This was one of the reasons for working
with \cref{hyp:refl-coequalisers}, rather than stronger versions thereof.

\begin{lem} \label{thm:catsym-has-refl-coeq} The double category $\VCatSym$ satisfies
\cref{hyp:fibrancy,hyp:coproducts-in-mnd,hyp:equalisers,hyp:tabulators,hyp:free-monad,hyp:coequalisers,hyp:equalisers-in-horizontal,hyp:refl-coequalisers}.
\end{lem}

\begin{proof}  We need to check the following:
\begin{enumerate}[(I)]
\item $\VCatSym$ is fibrant,
\item each $\Mnd_A(\VCatSym)$ has coproducts,
\item $(\VCatSym)_0$ has equalisers,
\item $\VCatSym$ has 1-tabulators,
\item each forgetful functor $\Mnd_A(\VCatSym) \to \End_A(\VCatSym)$ has a left adjoint,
\item each $\Mnd_A(\VCatSym)$ has coequalisers, 
\item $\VCatSym$ has local equalisers,
\item $\VCatSym$ has stable local reflexive coequalisers.
\end{enumerate}

Part~(I) can be proved along the lines of~\cite[Proposition~9.3]{Paper1} or by direct calculation.
For part~(III), observe that $(\VCatSym)_0 = \VCat$. 
For part~(IV), according to \cref{def:1tabulators} given a categorical symmetric sequence $x \co A \rightsquigarrow B$, its 1-tabulator $Tx$ is given by 
the $\ca{V}$-category whose set of objects is
\[
\coprod_{a \in A, b \in B} \ca{V}(I, x[(b); a]) \mathrlap{.}
\]
The morphisms of its underlying category $(f, g) \co (a',b',u') \to (a,b, u)$ are pairs consisting of $f \co a' \to a$ 
and $g \co b' \to b$ such that $u \cdot g = u' \cdot f$. The hom-objects $Tx( (a', b', u'), (a, b, u)) \in \ca{V}$
are given by a corresponding equaliser. The functors $\pi_{\mathfrak s} \co Tx \to A$ and $\pi_{\mathfrak s} \co Tx \to B$ are the two evident projections. For the 2-cell
 \begin{equation*}
    \begin{tikzcd}
      Tx\ar[r,rightsquigarrow,"\hid_{Tx}"]\ar[d,"\pi_{\mathfrak s}"']\ar[dr,phantom,"\Two \pi"] & {Tx}\ar[d,"\pi_{\mathfrak{t}}"] \\
      {A}\ar[r,rightsquigarrow,"x"'] & {B}
    \end{tikzcd}
  \end{equation*}
we need a $\ca{V}$-natural transformation $\pi \co \hid_{Tx} \to x \circ (S\pi_{\mathfrak t} \times \pi_{\mathfrak{s}})$. 
By the definition of horizontal maps in $\VCatSym$, the only non-trivial components of  
such a transformation are of the form  
\[
\pi_{((a',b',u')), (a,b,u)} \co Tx\big( (a',b', u'), (a,b,u)\big) \to x[(b');a]  \mathrlap{,}
\] 
 which are defined by sending $(f,g) \co (a',b',u') \to (a,b, u)$ to $u \cdot g$ (or, equivalently, $u' \cdot f)$.

Next, recall that, for $\ca{V}$-categories $A$ and $B$, the category $\VCatSym[A,B]$ defined in~\eqref{equ:catsym-hom} is locally
presentable. In particular, $\VCatSym$ has local equalisers, giving us (VII).
Also, by~\cite[Corollary~4.4.9]{GambinoJoyal}, composition in the horizontal bicategory of $\VCatSym$ 
preserves all colimits in the first variable and sifted colimits in the second variable.  
This implies (VIII) by \cref{defi:parallel_local}, as  reflexive coequalisers are sifted colimits, and (V) by known results on free monoids~\cite{KellyGM:unittc,FreeMonoids} applied to $\End_A(\VCatSym)=\VCat(SA^\op\otimes A,\ca{V})$.
It follows from (V) and fibrancy that $\VCatSym$ has free monads (\cref{prop:strongerfreemonads}). But now, since
 $\VCatSym$ has stable local reflexive coequalisers, free monads, local 
coproducts, and local reflexive coequalisers, each
 $\Mnd_A(\VMat)$ has coproducts and reflexive coequalisers by part~(iii) of \cref{lem:limits-in-end-mnd}, thus giving us (II) and (VI).
\end{proof} 

\begin{rmk} Let us note that the definition of the double category $\VCatSym$ and the
properties in \cref{thm:catsym-has-refl-coeq} hold for a general symmetric monoidal category 
$\ca{V}$, while the assumption that the monoidal structure is cartesian is needed for
\cref{thm:catsym-oplaxmonoidal}.
\end{rmk}

Monads in $\VCatSym$ are the \emph{symmetric substitutes} of \cite{DayB:abssec,DayBJ:laxmpoc},
which are similar to symmetric multicategories, the main differerence being that they have a category, rather than a mere set, of objects, whose maps can be thought of as a distinguished set of unary morphisms in a 
symmetric multicategory. In preparation for our main application is to multicategories, we record the following:

\begin{thm} \label{thm:bv-substitutes} 
Let $\ca{V}$ be locally presentable cartesian closed category, considered as a cartesian monoidal category.
The double category $\Bim(\VCatSym)$ of symmetric substitutes and their bimodules 
admits an symmetric oplax monoidal closed structure.
\end{thm}

\begin{proof} This follows from~\cref{thm:ctensor-bimodules}, which we can apply thanks to 
 \cref{thm:catsym-oplaxmonoidal} and \cref{thm:catsym-has-refl-coeq}.
\end{proof}

We define the double category of symmetric sequences $\VSym$ as the full double subcategory of $\VCatSym$ spanned by sets, viewed as discrete $\ca{V}$-categories, so that there is an inclusion
\begin{equation}
\label{equ:sym-in-catsym}
\VSym \hookrightarrow \VCatSym \mathrlap{.}
\end{equation}

Monads in $\VSym$ are exactly symmetric $\ca{V}$-multicategories (\cf \cite{BaezJ:higda}). Indeed, a monad~$(A,a)$ in  $\VSym$ gives a symmetric $\ca{V}$-multicategory with set of objects $A$ and hom-objects $a[x_1, \ldots, x_n;x] \in \ca{V}$, for $x_1, \ldots, x_n, x \in A$. The composition is given by the multiplication of the monad and the identity maps are given by the unit of the monad. Conversely, given a symmetric $\ca{V}$-multicategory $A$, we obtain a monad in $\VSym$ with
underlying set the set of objects of $A$ and symmetric sequence $\mathrm{hom}_A \co A \tickar A$.
See~\cite[Example~4.1.4]{GambinoJoyal} for more details. Furthermore, the free monad on a symmetric sequence is exactly the free symmetric multicategory on it. From now on, we will identify monads
in  $\VSym$ with symmetric $\ca{V}$-multicategories and simply write $A, B, C, \ldots$ to denote them.

\begin{lem} 
\label{thm:sym-has-refl-coeq}
$\VSym$ is fibrant and has stable local reflexive coequalisers.
\end{lem}

\begin{proof} The required structure and properties were established for $\VCatSym$ in \cref{thm:catsym-has-refl-coeq} and it is clear that they restrict to $\Sym$ along the inclusion in \eqref{equ:sym-in-catsym}.
\end{proof}

Thanks to \cref{thm:sym-has-refl-coeq}, we can consider the double category $\Bim(\VSym)$. The inclusion in~\cref{equ:sym-in-catsym} induces a commutative diagram of inclusions
\begin{equation*}
\begin{tikzcd}
\VSym \ar[d] \ar[r] & \VCatSym \ar[d]  \\
\Bim(\VSym) \ar[r]&  \Bim(\VCatSym)  \mathrlap{.}
\end{tikzcd}
\end{equation*}

We define
\[
\VSMultProf =  \Bim(\VSym)
\]
and refer to $\VSMultProf$ as the \emph{double category of symmetric $\ca{V}$-multiprofunctors}.
Indeed, as we explain below, $\VSMultProf$  is the analogue for symmetric $\ca{V}$-multicategories of the double category $\VProf$ of $\ca{V}$-categories, $\ca{V}$-functors and $\ca{V}$-profunctors.

We describe explicitly $\VSMultProf$. As mentioned, its objects are symmetric $\ca{V}$-multicategories.
Similarly, the vertical maps of $\Bim(\VSym)$ are symmetric $\ca{V}$-multifunctors. Thus, the vertical category of $\Bim(\VSym)$ is the category $\VSMultCat$ of symmetric $\ca{V}$-multicategories and symmetric $\ca{V}$-multifunctors. The horizontal maps of $\Bim(\VSym)$ were called bimodules in~\cite{GambinoJoyal}, but we prefer to refer to them as \emph{symmetric multiprofunctors} in order to emphasize the analogy with profunctors. We unfold the definition for the convenience of the readers.

\begin{defi} \label{defi:smultprof}
Let $A$ and $B$ be symmetric $\ca{V}$-multicategories.  A \emph{symmetric
$\ca{V}$-multiprofunctor} $p \co A \rightsquigarrow B$ consists of
\begin{enumerate}[(i)] 
\item for every $a \in A$ and $(b_1, \ldots, b_n) \in SB$, an object
\[
p[b_1, \ldots, b_n; a] \in \ca{V} \mathrlap{,}
\]
subject to equivariance conditions, so as to give a symmetric sequence $p \co A \rightsquigarrow B$; 
\item for every $\vec{a} = (a_1, \ldots, a_n) \in SA$, $\vec{b}_1, \ldots, \vec{b}_n \in SB$, and $a \in A$, 
a map
\[
\rho_{\vec{a}, \vec{b}_1, \ldots, \vec{b}_n, a} \co 
p[\vec{b}_1; a_1] \otimes \cdots \otimes p[\vec{b}_n; a_n] \otimes
\mathrm{hom}_A[\vec{a};a] \to
p[\vec{b}_1 \oplus \ldots \oplus \vec{b}_n; a]  \mathrlap{,}
\]
 subject to equivariance conditions, so as to give a right action $\rho \co p \circ \mathrm{hom}_A \Rightarrow  p$; 
\item for every $\vec{b} =  (b_1, \ldots, b_n) \in SB$, $\vec{b}_1, \ldots, \vec{b}_n \in SB$, and $a \in A$, 
a map
\[
\lambda_{\vec{b}, \vec{b}_1, \ldots, \vec{b}_n, a} \co
\mathrm{hom}_B[\vec{b}_1; b_1] \otimes \cdots \otimes \mathrm{hom}_B[\vec{b}_n; b_n] \otimes 
p[\vec{b}; a ] \to 
p[\vec{b}_1 \oplus \ldots \oplus \vec{b}_n; a] \mathrlap{,}
\]
subject to equivariance conditions, so as to give a left action $\lambda \co \mathrm{hom}_B \circ p \Rightarrow  p$;
\end{enumerate}
satisfying the bimodule compatibility conditions in \cref{defi:bimodules}.
\end{defi} 

Just as the values $f \in p[b;a]$ of a profunctor $p \co A \to B$ between categories can be understood as heterogenous morphisms~$f \co b \to a$, the values $f \co p[b_1, \ldots, b_n; a]$ of a symmetric multiprofunctor
$p \co A \rightsquigarrow B$ between symmetric multicategories can be understood as heterogenous multimorphisms 
$f \co b_1, \ldots b_n \to a$. From this point of view, the data in parts~(ii) and (iii) \cref{defi:smultprof}
can be readily understood as describing how to compose such multimorphisms 
with the multimorphisms of $A$ and of $B$, respectively. The equivariance, associativity,
unitality and bimodule compatibility conditions are then rather natural laws to impose on such
composition operations. As typical in this subject, these operations can be vividly represented
in terms of grafting operations on trees. Observe that the notation $\VSMultProf$ follows our convention of naming bimodule double categories by the name of their horizontal maps. 

We now obtain the main application of our general theory.

\begin{thm} \label{thm:bv-bimod} 
Let $\ca{V}$ be locally presentable cartesian closed category, considered as a cartesian monoidal category.
The double category $\VSMultProf$ of symmetric $\ca{V}$-multicategories, symmetric $\ca{V}$-multifunctors 
and symmetric $\ca{V}$-multiprofunctors
admits a symmetric oplax monoidal closed structure, given on objects by the Boardman--Vogt tensor product of symmetric $\ca{V}$-multicategories.
\end{thm}

\begin{proof} This follows from~\cref{thm:bv-substitutes} by observing that the 
commuting tensor product of symmetric substitutes restricts to the Boardman--Vogt tensor
product of symmetric multicategories. 
\end{proof}

\begin{rmk}  Part of \cref{thm:bv-bimod} is the statement that the vertical category of $\VSMultProf$
admits a symmetric monoidal structure, which is in fact closed by \cref{thm:ctensor-monads}.
This recovers the known fact that  $\VSMultCat$ admit a symmetric monoidal closed structure.
 In fact, by \cref{thm:nctensor-smc}, this category admits also a non-commuting tensor product, 
 which was constructed in~\cite{ElmendorfA:percma} and should be regarded as the counterpart for symmetric
$\ca{V}$-multicategories of the funny tensor product of $\ca{V}$-categories. 
Compare also the pushout in~\cref{eq:pushout} with \cite[Construction~4.10]{ElmendorfA:percma}.
\end{rmk} 

We write $\VOpdBim$ for the full double subcategory of $\VSMultProf$ spanned by symmetric $\ca{V}$-multicategories with one object, which are the same thing as $\ca{V}$-operads~\cite{KellyGM:opejpm,Mayoperad}. Following the terminology in the literature, we refer to 
the horizontal maps of $\VOpdBim$ as bimodules. An explicit definition of this notion can be readily obtained
by instantiating \cref{defi:smultprof} taking $A$ and $B$ therein to be operads, but we do not spell it out.
Our final theorem shows that the Boardman--Vogt tensor product of operads extends to a 
symmetric normal oplax monoidal closed structure on $\VOpdBim$.

\begin{thm} \label{corollary:oplaxoperads} 
Let $\ca{V}$ be locally presentable category, considered as a cartesian symmetric monoidal category.
The double category $\VOpdBim$ of symmetric $\ca{V}$-operads, symmetric $\ca{V}$-operad morphisms, and symmetric $\ca{V}$-operad
bimodules  admits a symmetric  oplax monoidal closed structure, given by the Boardman--Vogt tensor product.
\end{thm}

\begin{proof} This follows by~\cref{thm:bv-bimod} observing that the Boardman--Vogt tensor product of
symmetric multicategories restricts to operads.
\end{proof} 

While the definition of the Boardman--Vogt tensor product symmetric multicategories goes back to~\cite{BoardmanJ:homias} and its definition on operad bimodules to~\cite{DwyerW:BoardmanVtpo},
the interaction between the Boardman--Vogt tensor product and the composition operation of bimodules isolated  in~\cref{thm:bv-bimod,corollary:oplaxoperads}
appears to be new. For operads, this amounts to the fact that, 
given operad bimodules $p_1 \co A_1 \rightsquigarrow B_1$ and $p_2 \co A_2 \rightsquigarrow B_2$,
$q_1 \co B_1 \rightsquigarrow C_1$ and $q_2 \co B_2 \rightsquigarrow C_2$, we have an operad morphism
\[
\begin{tikzcd}[column sep=.4in]
 A_1\ctensor A_2\ar[rr,rightsquigarrow,"(q_1\bimcomp p_1) \ctensor (q_2 \bimcomp p_2)"]\ar[d,equal]\ar[drr,phantom,"\Two \xi"] && C_1\ctensor C_2\ar[d,equal] \\
 A_1\ctensor A_2\ar[r,rightsquigarrow,"p_1\ctensor p_2"'] & B_1\ctensor B_2\ar[r,rightsquigarrow,"q_1\ctensor q_2"'] & C_1\ctensor C_2 \mathrlap{,}
\end{tikzcd}
\]
satisfying appropriate coherence conditions. As discussed in \cref{thm:not-normal}, we do not expect the oplax monoidal structures of \cref{thm:bv-bimod} and \cref{corollary:oplaxoperads} 
to be normal. Hence, the Boardman--Vogt tensor product does not seem to yield a 
symmetric monoidal oplax structure on the horizontal bicategories of $\VSMultProf$ and $\VOpdBim$. Happily,  the double categories $\VSMultProf$ and $\VOpdBim$ 
have instead a genuine symmetric oplax monoidal structure.

\bibliographystyle{alpha}
\bibliography{References}

\end{document}